\documentclass{article}
\usepackage[utf8]{inputenc}
\usepackage{amssymb,amsmath, amsthm, mathabx,mathtools}
\usepackage{amsmath,amsfonts,amssymb,amsthm,epsfig,epstopdf,titling,url,array}
\usepackage{tikz}
\usetikzlibrary{shapes,snakes}
\usepackage{color, enumerate}
\usepackage{comment, authblk}
\usepackage{hyperref}
\usepackage{pgfplots}
\usepackage{bm}
\usepackage{stmaryrd}
\usepackage{graphicx}
\usepackage{caption}
\usepackage{subcaption}
\usepackage{a4wide}
\usepackage{titlesec}
\usepackage[a4paper, total={6in, 8in}]{geometry}
\usepackage{epstopdf}

\usepackage{sectsty}
\usepackage{lipsum}
\usepackage{algorithm}
\usepackage[noend]{algpseudocode}
\usepackage{algpseudocode}
\usepackage{pifont}

\titleformat*{\subsection}{\large\bfseries}
\numberwithin{equation}{section}

\pgfplotsset{compat=newest}
\pgfplotsset{plot coordinates/math parser=false}
\newlength\figureheight
\newlength\figurewidth

\DeclareMathOperator{\Cov}{Cov}

\usepackage{enumerate}
\usepackage{mathrsfs}
\usepackage{wrapfig}

\usepackage{color}

\numberwithin{equation}{section}

\newcommand{\rI}{\mathrm{I}}

\newcommand{\beq}{\begin{equation}}
\newcommand{\bEq}{\end{equation}}

\newcommand{\bx}{{\bf{x}}}
\newcommand{\by}{{\bf{y}}}

\newcommand{\bz} {{\bf {z}}}

\newcommand{\al}{\alpha}

\newcommand{\be}{\begin{equation}}
\newcommand{\ee}{\end{equation}}

\newcommand{\e}{{\varepsilon}}

\newcommand{\fa}{{\mathfrak a}}
\newcommand{\fb}{{\mathfrak b}}

\newcommand{\bU}{ {\bf  U}}

\usepackage{amsmath} %[intlimits]
\usepackage{amssymb}
\usepackage{amsthm}

\renewcommand{\cal}{\mathcal}

\newcommand{\wh}{\widehat}
\newcommand{\wt}{\widetilde}

\newcommand{\ii}{\mathrm{i}} %\newcommand{\mi}{\mathrm{i}}
\newcommand{\dd}{\mathrm{d}}

\renewcommand{\epsilon}{\varepsilon}
\renewcommand{\leq}{\leqslant}
\renewcommand{\geq}{\geqslant}

\renewcommand{\le}{\leq}
\renewcommand{\ge}{\geq}

\renewcommand{\P}{\mathbb{P}}
\newcommand{\E}{\mathbb{E}}
\newcommand{\R}{\mathbb{R}}
\newcommand{\C}{\mathbb{C}}
\newcommand{\N}{\mathbb{N}}

\newcommand{\wsG}{{\widehat{\cal G}}}
\newcommand{\wsH}{{\widehat{\cal H}}}

\DeclareMathOperator{\diag}{diag}
\DeclareMathOperator{\tr}{Tr}

\DeclareMathOperator{\re}{Re}
\DeclareMathOperator{\im}{Im}

\DeclareMathOperator{\OO}{O}
\DeclareMathOperator{\oo}{o}

\DeclareMathOperator{\bC}{\mathbf{C}}

\DeclareMathOperator{\bE}{\mathbf{E}}

\DeclareMathOperator{\bv}{\mathbf{v}}
\DeclareMathOperator{\bu}{\mathbf{u}}
\DeclareMathOperator{\bw}{\mathbf{w}}

\DeclareMathOperator{\bbN}{\mathbb{N}}
\DeclareMathOperator{\bbP}{\mathbb{P}}

\DeclareMathOperator{\sI}{\mathcal{I}}

\theoremstyle{plain} %plain, definition, remark
\newtheorem{theorem}{Theorem}[section]
\newtheorem*{theorem*}{Theorem}
\newtheorem{lemma}[theorem]{Lemma}
\newtheorem{assumption}[theorem]{Assumption}
\newtheorem*{lemma*}{Lemma}
\newtheorem{corollary}[theorem]{Corollary}
\newtheorem*{corollary*}{Corollary}
\newtheorem{proposition}[theorem]{Proposition}
\newtheorem*{proposition*}{Proposition}
\newtheorem{claim}[theorem]{Claim}

\newtheorem{definition}[theorem]{Definition}
\newtheorem*{definition*}{Definition}
\theoremstyle{remark}

\newtheorem*{example*}{Example}
\newtheorem{remark}[theorem]{Remark}

\newtheorem*{remark*}{Remark}
\newtheorem*{remarks*}{Remarks}

\renewcommand{\Im}{{\rm{Im}}}

\allowdisplaybreaks

\title{Sample canonical correlation coefficients of high-dimensional random vectors with finite rank correlations}
\author[1]{Zongming Ma \thanks{E-mail: zongming@wharton.upenn.edu}}
\author[1]{Fan Yang  \thanks{E-mail: fyang75@wharton.upenn.edu}}%\\ %\href{mailto:fyang75@math.wisc.edu}{fyang75@math.wisc.edu}}
\affil[1]{Department of Statistics and Data Science, University of Pennsylvania}%, Madison, WI 53706, USA

\begin{document}
\maketitle

%\begin{frontmatter}
%\title{Sample canonical correlation coefficients of high-dimensional random vectors with finite rank correlations}
%%\title{A sample article title with some additional note\thanksref{t1}}
%\runtitle{Sample canonical correlation coefficients}
%%\thankstext{T1}{A sample additional note to the title.}
%
%\begin{aug}
%%%%%%%%%%%%%%%%%%%%%%%%%%%%%%%%%%%%%%%%%%%%%%%
%%%Only one address is permitted per author. %%
%%%Only division, organization and e-mail is %%
%%%included in the address.                  %%
%%%Additional information can be included in %%
%%%the Acknowledgments section if necessary. %%
%%%%%%%%%%%%%%%%%%%%%%%%%%%%%%%%%%%%%%%%%%%%%%%
%\author[A]{\fnms{Zongming} \snm{Ma}\ead[label=e1]{zongming@wharton.upenn.edu}}
%\and
%\author[A]{\fnms{Fan} \snm{Yang}\ead[label=e3,mark]{fyang75@wharton.upenn.edu}}
%%%%%%%%%%%%%%%%%%%%%%%%%%%%%%%%%%%%%%%%%%%%%%%
%%% Addresses                                %%
%%%%%%%%%%%%%%%%%%%%%%%%%%%%%%%%%%%%%%%%%%%%%%%
%\address[A]{Department of Statistics and Data Science, University of Pennsylvania, Philadelphia, PA 19104, USA
%\printead{e1,e3}}
%
%%\address[B]{Department of Statistics, University of Pennsylvania
%%\printead{e3}}
%\end{aug}

\begin{abstract}
Consider two random vectors $\wt{\bx}=A \mathbf z + \bC_1^{1/2}\mathbf x  \in \R^p$ and $\wt{\by}=B \mathbf z + \bC_2^{1/2}\mathbf y \in \mathbb R^q$, where $\mathbf x \in \R^p$, $\mathbf y \in \R^q$ and $\mathbf z\in \R^r$ are independent random vectors with i.i.d. entries of zero mean and unit variance, $\bC_1$ and $\bC_2$ are $p \times p$ and $q\times q$ deterministic population covariance matrices, and $A$ and $B$ are $p \times r$ and $q\times r$ deterministic factor loading matrices. With $n$ independent observations of $\wt{\mathbf x}$ and $\wt{\mathbf y}$, we study the sample canonical correlations between them. Under the sharp fourth moment condition on the entries of $\mathbf x$, $\mathbf y$ and $\mathbf z$, we prove the BBP transition for the sample canonical correlation coefficients (CCCs). More precisely, if a population CCC is below a threshold, then the corresponding sample CCC converges to the right edge of the bulk eigenvalue spectrum of the sample canonical correlation matrix and satisfies the famous Tracy-Widom law; if a population CCC is above the threshold, then the corresponding sample CCC converges to an outlier that is detached from the bulk eigenvalue spectrum. We prove our results in full generality, in the sense that they also hold for near-degenerate population CCCs and population CCCs that are close to the threshold. 
\end{abstract}

%\begin{keyword}[class=MSC2020]
%\kwd[Primary ]{60B20}
%\kwd{62E20}
%\kwd[; secondary ]{62H99}
%\end{keyword}
%
%\begin{keyword}
%\kwd{Canonical correlation analysis}
%\kwd{BBP transition}
%\kwd{Tracy-Widom law}
%\kwd{edge eigenvalues}
%\end{keyword}

%\begin{keyword}[class=MSC]
%	\kwd[Primary ]{60B20}
%	\kwd{62E20}
%	\kwd[; secondary ]{62H99}
%\end{keyword}
%
%\begin{keyword}
%	\kwd{Canonical correlation analysis}
%	\kwd{Tracy-Widom law}
%	\kwd{Local laws}
%\end{keyword}

%\end{frontmatter}

\section{Introduction}\label{sec_intro}

Since the seminal work by Hotelling \cite{Hotelling}, the canonical correlation analysis (CCA) has been one of the most classical methods to study the correlations between two random vectors. Given two random vectors $\wt{\mathbf x}\in \R^p$ and $\wt{\mathbf y}\in \R^q$, CCA seeks two sequences of orthonormal vectors, such that the projections of $\wt{\mathbf x}$ and $\wt{\by}$ onto these vectors have maximized correlations, and the corresponding sequence of correlations are called {\it canonical correlation coefficients} (CCCs). More precisely, we first find a pair of unit vectors $\mathbf a_1\in \R^p$ and $\mathbf b_1\in \R^q$ that maximizes the correlation $ \rho(\mathbf a,\mathbf b):=\text{Corr}(\mathbf a^\top \wt\bx, \mathbf b^\top \wt\by ).$ Then, $\rho_1:=\rho(\mathbf a_1,\mathbf b_1)$ is the first CCC and $(\mathbf a_1^\top \wt\bx, \mathbf b_1^\top \wt\by )$ is the first pair of canonical variables. Suppose we have obtained the first $k$ CCCs, $\rho_1,\ldots, \rho_k$, and the corresponding pairs of canonical variables. We then define inductively the $(k+1)$-th CCC by seeking a pair of unit vectors $(\mathbf a_{k+1},\mathbf b_{k+1})$ that maximizes $\rho(\mathbf a,\mathbf b)$ subject to the constraint that $(\mathbf a_{k+1}^\top \wt\bx, \mathbf b_{k+1}^\top \wt\by )$ is uncorrelated with the first $k$ pairs of canonical variables. Then, $\rho_{k+1}:=\rho(\mathbf a_{k+1},\mathbf b_{k+1})$ is the $(k+1)$-th CCC.

There is a well-know representation of CCCs in terms of the eigenvalues of the population canonical correlation (PCC) matrix defined using the population covariance and cross-covariance matrices:
$$\wt\Sigma_{xx}:= \Cov(\wt{\mathbf x},\wt{\mathbf x}),\quad \wt\Sigma_{yy}:= \Cov(\wt{\mathbf y},\wt{\mathbf y}), \quad \wt\Sigma_{xy}=\wt\Sigma_{yx}^\top:= \Cov(\wt{\mathbf x},\wt{\mathbf y}), $$
where for two random vectors $\mathbf u$ and $\mathbf v$, we define $\Cov(\mathbf u,\mathbf v):=\E \left[(\mathbf u-\E\mathbf u)(\mathbf v-\E\mathbf v)^\top\right]$. It is known that $\rho_i$ is the square root of the $i$-th largest eigenvalue, say $t_i $, of the PCC matrix 
$$\wt{\bm\Sigma}:=\wt\Sigma_{xx}^{-1/2}\wt\Sigma_{xy}\wt\Sigma_{yy}^{-1}\wt\Sigma_{yx}\wt\Sigma_{xx}^{-1/2}.$$ 
%With slight abuse of notations, we shall refer to $t_i$ as the \emph{population canonical correlation coefficients}. 
Suppose we observe $n$ independent samples of $( \wt{\mathbf x}, \wt{\mathbf y})$. Then, we can study the population CCCs through their sample counterparts. More precisely, we form data matrices $\wt{\cal X}$ and $\wt{\cal Y}$ as
\be\label{data mat}\wt{\cal X}:= n^{-1/2}\begin{pmatrix}\wt{\bx}_1,\wt\bx_2, \cdots ,\wt\bx_n\end{pmatrix},\quad \wt{\cal Y}:= n^{-1/2}\begin{pmatrix}\wt{\by}_1,\wt\by_2, \cdots ,\wt\by_n\end{pmatrix} ,\ee
where $(\wt{\bx}_i,\wt \by_i)$ are i.i.d. copies of $( \wt{\mathbf x}, \wt{\mathbf y})$ and $n^{-1/2}$ is a convenient scaling, so that the sample covariance and cross-covariance matrices can be written concisely as 
$$\wt S_{xx}:=\frac{1}{n}\sum_{i=1}^n \wt\bx_i  \wt\bx_i^\top= \wt{\cal X}\wt{\cal X}^\top  , \ \ \wt S_{yy}:=\frac{1}{n}\sum_{i=1}^n \wt\by_i  \wt\by_i^\top= \wt{\cal Y}\wt{\cal Y}^\top , \ \ \wt S_{xy}=\wt S_{yx}^\top:=\frac{1}{n}\sum_{i=1}^n \wt\bx_i  \wt\by_i^\top = \wt{\cal X}\wt{\cal Y}^\top  .$$
The squares of the sample CCCs, $\wt\lambda_1 \ge \wt\lambda_2\ge \cdots \ge \wt\lambda_{p\wedge q}\ge 0$, are then defined as the eigenvalues of the {\it sample canonical correlation} (SCC) matrix  
$$\cal C_{\wt{\cal X}\wt{\cal Y}}:=\wt S_{xx}^{-1/2}\wt S_{xy}\wt S_{yy}^{-1}\wt S_{yx}\wt S_{xx}^{-1/2}.$$
If $n\to \infty$ while $p,$ $q$ and $r$ are fixed, the SCC matrix converges to the PCC matrix almost surely by the law of large numbers, and hence  sample CCCs can be used as consistent estimators of population CCCs. 
However, many modern applications, such as statistical learning, wireless communications, medical imaging, financial economics and population genetics, are seeing a rapidly increasing demand in analyzing high-dimensional data, where $p$ and $q$ are comparable to $n$ when $n$ is large. In the high-dimensional setting, the behavior of the SCC matrix can deviate greatly from the PCC matrix due to the so-called ``curse of dimensionality$"$.

There have been several works on the theoretical analysis of high-dimensional CCA. We mention some of them that are most related to this paper. 

First, we consider the null case where $\wt\bx$ and $\wt\by$ are independent random vectors.  When $\wt\bx$ and $\wt\by$ are independent Gaussian vectors, the eigenvalues of the SCC matrix have the same joint distribution as those of a double Wishart matrix \cite{CCA_TW}. In particular, the joint distribution of the eigenvalues of double Wishart matrices has been studied in the context of the Jacobi ensemble and F-type matrices \cite{CCA_TW2,CCA_TW}, where the largest few eigenvalues of the SCC matrix are shown to satisfy the Tracy-Widom law asymptotically. For generally distributed random vectors $\wt{\bx}$ and $\wt\by$, the Tracy-Widom fluctuation of the largest eigenvalues of the SCC matrix is proved in \cite{CCA2} under the assumption that the entries of $\wt\bx$ and $\wt\by$ have finite moments up to any order. The moment assumption is later relaxed to the finite fourth moment assumption in \cite{PartIII}. In the Gaussian case, it is shown in \cite{Wachter} that, almost surely, the empirical spectral distribution (ESD) of the SCC matrix converges weakly to a deterministic probability distribution (cf. \eqref{LSD}). In the general non-Gaussian case, both the convergence and the linear spectral statistics of the ESD of the SCC matrix have been proved \cite{CCA_ESD,CCA_CLT}. 

Next, we consider the case where $\wt\bx$ and $\wt\by$ have finite rank correlations. If $\wt\bx$ and $\wt\by$ are random Gaussian vectors, then the asymptotic distributions of sample CCCs have been derived when one of $p$ and $q$ is fixed as $n\to \infty$ \cite{Fujikoshi2017}. If $p$ and $q$ are both proportional to $n$, the asymptotic distributions of sample CCCs have been established under the Gaussian assumption in \cite{CCA}. Under certain sparsity assumptions, the theory of high-dimensional sparse CCA and it applications have been discussed in \cite{gao2015,gao2017}. In \cite{ODA2019}, the authors derived asymptotic null and non-null distributions of several test statistics for tests of redundancy in high-dimensional CCA. In \cite{johnstone2020}, the authors studied the asymptotic behaviors of the likelihood ratio processes of CCA under the null hypothesis of no spikes and the alternative hypothesis of a single spike.

In this paper, we consider the following signal-plus-noise model for $\wt{ \mathbf x} \in \R^p$ and $ \wt \by\in \mathbb R^q$:
$$
\wt{ \mathbf x} =A \mathbf z+ \mathbf C_1^{1/2} \mathbf x , \quad \wt \by=B \mathbf z + \mathbf C_2^{1/2} \mathbf y .
$$
Here, $\bz\in \R^r$ is a rank-$r$ signal vector with i.i.d. entries of mean zero and variance one and independent of $\bx$ and $\by$, and $A$ and $B$ are $p \times r$ and $q\times r$ deterministic factor loading matrices, respectively. $ \mathbf x\in \R^p$ and $\by\in \R^q$ are two independent noise vectors with i.i.d. entries of mean zero and variance one, and $\mathbf C_1$ and $\mathbf C_2$ are $p \times p$ and $q\times q$ deterministic population covariance matrices. Then, we can write the data matrices in \eqref{data mat} as
%Moreover, suppose that the entries of $\mathbf x\in \R^p$, $\mathbf y\in \R^q$ and $\mathbf z\in \R^r$ are real independent random variables with zero mean  and unit variance. For $n$ independent samples $( \wh \bx_i, \wh \by_i)$, $1\le i \le n$, we can arrange them into the following data matrix with a conventional scaling $n^{-1/2}$:
\be\label{data model1} \wt{\mathcal X}: = AZ+\mathbf C^{1/2}_1 X , \quad  \wt{\mathcal Y}: =B Z+ \mathbf C^{1/2}_2 Y ,\ee
where $X$, $Y$ and $Z$ are respectively $p\times n$, $q\times n$ and $r\times n$ matrices with i.i.d. entries of mean zero and variance $n^{-1}$ and they are independent of each other. We consider the high-dimensional setting with a low-rank signal, that is, ${p}/{n}\to c_1$ and ${q}/{n}\to c_2$ as $n\to \infty$ for some constants $c_1\in (0,1)$ and $c_2\in (0,1-c_1)$, and $r$ is a fixed integer that does not depend on $n$. 

For the model \eqref{data model1}, the PCC matrix is given by 
$$ \wt{\bm\Sigma}=(\bC_1 + AA^\top)^{-1/2}AB^\top (\bC_2 + BB^\top)^{-1}BA^\top (\bC_1 + AA^\top)^{-1/2},$$
which is of rank at most $r$. We order the nontrivial eigenvalues of $ \wt{\bm\Sigma}$ as $t_1\ge t_2\ge \cdots \ge t_r \ge 0$. 
Under the Gaussian assumption, that is, $X$, $Y$ and $Z$ are independent random matrices with i.i.d. Gaussian entries, Bao et al. \cite{CCA} proved that for any $1\le i \le r$, $\wt\lambda_i$ exhibits very different behaviors depending on whether $t_i$ is below or above the threshold $t_c$, where % defined as
\be\label{tc}t_c :=  \sqrt{\frac{c_1 c_2}{(1-c_1)(1-c_2)}}.\ee
More precisely, if $t_i < t_c$, then the corresponding sample CCC $\wt\lambda_i$ sticks to the right edge $\lambda_+$ of the bulk eigenvalue spectrum (cf. \eqref{lambdapm}) of the SCC matrix, and $n^{2/3}( \wt\lambda_i - \lambda_+)$ converges weakly to the type-1 Tracy-Widom distribution. On the other hand, if $t_i>t_c$, then it gives rise to an outlier $\wt\lambda_i$ that %is detached from the bulk eigenvalue spectrum, and 
lies around a fixed location $\theta_i \in (\lambda_+,1)$ determined by $t_i$, $c_1$ and $c_2$. Furthermore, $n^{1/2}(\wt\lambda_i - \theta_i)$ converges weakly to a centered Gaussian. Such an abrupt change of the behavior of $\wt\lambda_i$ when $t_i$ crosses the threshold $t_c$ is generally called a \emph{BBP transition}, which dates back to the seminal work of Baik, Ben Arous and P\'{e}ch\'{e} \cite{BBP} on spiked sample covariance matrices. The BBP transition has been observed in many random matrix ensembles with finite rank perturbations. Without attempting to be comprehensive, we mention the references \cite{capitaine2009,capitaine2012,FP_2007,KY,KY_AOP,Peche_2006} on deformed Wigner matrices,  \cite{bai2008_spike, BBP, Baik2006, principal,FP_2009, IJ2,DP_spike} on spiked sample covariance matrices, \cite{DY2, yang2018,YLDW2020} on spiked separable covariance matrices, and \cite{belinschi2017,benaych-georges2011,BENAYCHGEORGES2011,wang2017} on several other types of deformed random matrix ensembles. In our setting, the SCC matrix $\cal C_{\wt{\cal X}\wt{\cal Y}}$ can be regarded as a finite rank perturbation of the SCC matrix in the null case with $r=0$.

A natural question is whether the results in \cite{CCA} hold universally, that is, whether $\wt\lambda_i$ satisfies the same properties if we only assume certain moment conditions on the entries of $X$, $Y$ and $Z$. In fact, the proof in \cite{CCA} depends crucially on the rotational invariance of multivariate Gaussian distributions under orthogonal transforms, and it is hard (if possible) to be extended to the data matrices with generally distributed entries. In this paper, we answer the above question definitely, and show the universality of the results in \cite{CCA}. Moreover, we highlight the following improvements over the previous results. % in \cite{CCA}. 
\begin{itemize}
	\item Theorem \ref{main_thm1} shows that the following results hold assuming only a finite fourth moment condition (actually we require a slightly weaker condition \eqref{tail_cond}): for $1\le i \le r$, $n^{2/3}( \wt\lambda_i - \lambda_+)$ converges weakly to the Tracy-Widom law if $t_i<t_c$, while $\wt\lambda_i  \to \theta_i$ in probability if $t_i>t_c$. 
	
	\item We obtain quantitative versions of all the results under general moment assumptions: Theorem \ref{main_thm}  provides almost sharp convergence rates for the sample CCCs; Theorem \ref{main_thm0.5} provides an almost sharp eigenvalue sticking estimate, which shows that the eigenvalues of the SCC matrix stick to those of the null SCC matrix with $r=0$. 
	
	\item Our results hold even when some $t_i$-s are close to the threshold $t_c$ and when there are groups of near-degenerate $t_i$-s---both of these two cases are ruled out in the setting of \cite{CCA}.  
\end{itemize}
%To complete the BBP transition theory, we still need to prove the CLT for $ \wt\lambda_i $ when $t_i>t_c$. Due to length constraint, we include its study into another paper \cite{PartII}, where we show that $n^{1/2}(\wt\lambda_i - \theta_i)$ still converges to a center Gaussian but with a limiting variance that is different from the one in the Gaussian case. 
Instead of using the rotational invariance of multivariate Gaussian distributions, the proofs in this paper are based on a linearization method developed in \cite{PartIII}, which reduces the problem to the study of a $(p+q+2n)\times (p+q+2n)$ random matrix $H$ that is linear in $X$ and $Y$ (cf. \eqref{linearize_block}). Moreover, an optimal local law has been proved for the resolvent $G:=H^{-1}$ in \cite{PartIII}, which is the basis of all the 
proofs in this paper. Our approach is relatively more flexible and allows us to obtain precise convergence rates for the eigenvalues of the SCC matrix $\cal C_{\wt{\cal X}\wt{\cal Y}}$.  %We will also use the sharp local law on the resolvent $G:=H^{-1}$, the eigenvalues rigidity of the SCC matrix, and the Tracy-Widom law proved in \cite{PartIII}.  

Before concluding the introduction, we fix some notations that will be used frequently in the paper. For two quantities $a_n$ and $b_n$ depending on $n$, we use $a_n = \OO(b_n)$ to mean that $|a_n| \le C|b_n|$ for a constant $C>0$, and use $a_n=\oo(b_n)$ to mean that $|a_n| \le c_n |b_n|$ for a positive sequence of numbers $c_n\downarrow 0$ as $n\to \infty$. We will use the notations $a_n \lesssim b_n$ if $a_n = \OO(b_n)$, and $a_n \sim b_n$ if $a_n = \OO(b_n)$ and $b_n = \OO(a_n)$. For a matrix $A$, we use $\|A\|$ to denote its operator norm.  For a vector $\mathbf v$, we use $\|\mathbf v\|$ to denote its Euclidean norm. In this paper, we will abbreviate an identity matrix as $I$ or $1$.

\section{The model and main results}\label{main_result}

\subsection{The model}
We consider two independent families of data matrices $X=(x_{ij})$ and $Y=(y_{ij})$, which are of dimensions $p\times n$ and $q\times n$, respectively. We assume that the entries $x_{ij}$, $1 \le i \le p$, $1\le j \le n$, and $y_{ij}$, $ 1\le i \le q$, $1\le j \le n$, are real independent random variables satisfying 
\begin{equation}\label{assm1}
	\mathbb{E} x_{ij}=\mathbb{E} y_{ij} =0, \ \quad \ \mathbb{E} \vert x_{ij} \vert^2=\mathbb{E} \vert y_{ij} \vert^2  =n^{-1}.
\end{equation}
To be more general, we do not assume that these random variables are identically distributed. 
%For definiteness, in this paper we focus on the real case, i.e. the entries of $X$ and $Y$ are real. However, we remark that our proof can be applied to the complex case after minor modifications.
% if we assume in addition that
%$$\mathbb{E} x_{ij}^2=\mathbb{E} y_{ij}^2  =0. $$
We define the following data model with finite rank correlation:
$$\wt{\mathcal X}: = \mathbf C^{1/2}_1 X + \wt A Z, \quad \wt{\mathcal Y}: = \mathbf C^{1/2}_2 Y + \wt B Z,$$
where $ \mathbf C_1$ and $\mathbf C_2$ are $p\times p$ and $q\times q$ deterministic positive definite symmetric covariance matrices, $\wt A$ and $\wt B$ are $p\times r$ and $q\times r$ deterministic matrices, and $Z=(z_{ij})$ is an $r\times n$ random matrix which gives the nontrivial correlation between \smash{$\wt{\mathcal X}$ and $\wt{\mathcal Y}$}. We assume that $Z$ is independent of $X$ and $Y$, and the entries $z_{ij}$, $1 \le i\le r$, $1\le j \le n$, are independent random variables satisfying
\begin{equation}\label{assmZ}
	\mathbb{E} z_{ij} =0, \ \quad \ \mathbb{E} \vert z_{ij} \vert^2 =n^{-1}.
\end{equation}
In this paper, we study the eigenvalues of the {\it sample canonical correlation} (SCC) matrix 
$$\cal C_{\wt{\cal X}\wt{\cal Y}}= \big(\wt{\cal X}\wt{\cal X}^{\top}\big)^{-1/2}  \wt{\cal X}\wt{\cal Y}^{\top} \big( \wt{\cal Y}\wt{\cal Y}^{\top}\big)^{-1} \wt{\cal Y}\wt{\cal X}^{\top}  \big(\wt{\cal X}\wt{\cal X}^{\top}\big)^{-1/2}.$$
In particular, we are interested in the relations between the eigenvalues of $\cal C_{\wt{\cal X}\wt{\cal Y}}$ and those of the {\it population canonical correlation} (PCC) matrix:
$$\wt{\bm\Sigma}:=\wt\Sigma_{xx}^{-1/2} \wt\Sigma_{xy} \wt\Sigma_{yy}^{-1} \wt\Sigma_{yx}  \wt\Sigma_{xx}^{-1/2} , \quad \wt\Sigma_{xx}:= \mathbf C_1 + \wt A \wt A^{\top}, \ \ \wt\Sigma_{yy}:= \mathbf C_2 + \wt B \wt B^{\top}, \ \ \wt\Sigma_{xy}= \wt\Sigma_{yx}^{\top}:=  \wt A  \wt B^\top .$$
It is well-known that the canonical correlation coefficients are the square roots of the eigenvalues of the PCC matrix. 
Note that $\cal C_{\wt{\cal X}\wt{\cal Y}}$ is similar to %the matrix
$$\cal C'_{\wt{\cal X}\wt{\cal Y}}:= \wt{\cal X}\wt{\cal Y}^{\top} \big( \wt{\cal Y}\wt{\cal Y}^{\top}\big)^{-1} \wt{\cal Y}\wt{\cal X}^{\top} \big(\wt{\cal X}\wt{\cal X}^{\top}\big)^{-1}.$$
%,\quad \cal C'_{{\cal X}{\cal Y}}:=  \big( {\cal X}{\cal Y}^{\top}\big)\big({\cal Y}{\cal Y}^{\top}\big)^{-1}\big({\cal Y}{\cal X}^{\top}\big) \big({\cal X}{\cal X}^{\top}\big)^{-1} .$$
%Note that $\cal C'_{\wt{\cal X}\wt{\cal Y}}$ and $\cal C_{\wt{\cal X}\wt{\cal Y}}$ are similar matrices. 
Under the non-singular transformation 
$\wt{\mathcal X} \to  {\mathcal X}:=\mathbf C_1^{-1/2}\wt{\mathcal X}$ and $ \wt{\mathcal Y} \to  {\mathcal Y}:=\mathbf C_2^{-1/2}\wt{\mathcal Y},$  
%$$\wt{\cal X}\wt{\cal X}^{\top} = \mathbf C_1^{1/2} {\cal X}{\cal X}^{\top}\mathbf C_1^{1/2},\quad \wt{\cal Y}\wt{\cal Y}^{\top} = \mathbf C_2^{1/2} {\cal Y}{\cal Y}^{\top}\mathbf C_2^{1/2},\quad \wt{\cal X}\wt{\cal Y}^{\top} = \mathbf C_1^{1/2} {\cal X}{\cal Y}^{\top}\mathbf C_2^{1/2},$$
we see that 
$$  \cal C'_{\wt{\cal X}\wt{\cal Y}}=\bC_1^{1/2}\cal C'_{{\cal X}{\cal Y}} \bC_1^{-1/2},$$
which shows that $\cal C_{\wt{\cal X}\wt{\cal Y}}$ and $\cal C_{{\cal X}{\cal Y}}$ have the same eigenvalues. A similar argument also shows that the eigenvalues of the PCC matrix are unchanged under the same non-singular transformation.
%Note that the eigenvalues of both SCC and PCC matrices are unchanged under the non-singular transformations \smash{$\wt{\mathcal X} \to  {\mathcal X}:=\mathbf C_1^{-1/2}\wt{\mathcal X}$ and $\wt{\mathcal Y} \to  {\mathcal Y}:=\mathbf C_2^{-1/2}\wt{\mathcal Y}$}. 
Hence, without loss of generality, we only need to consider a simpler model
%Thus it suffices to consider the data matrices
\be\label{assm data}
{\mathcal X}: =  X + A Z, \quad {\mathcal Y}: = Y +  B Z, \quad \text{where } \ A:=\mathbf C^{-1/2}_1 \wt A, \quad B:=\mathbf C^{-1/2}_2 \wt B.
\ee
%By adding some singular vectors with zero singular values, it suffices to 
We assume that $A$ and $B$ have the following singular value decompositions: 
\be\label{assm AB}
A= \sum_{i=1}^r a_i \bu_i^a (\bv_i^a)^{\top} , \quad B= \sum_{i=1}^r b_i \bu_i^b (\bv_i^b)^{\top},
\ee
where $\{a_i\}$ and $\{b_i\}$ are the singular values, $\{\bu_i^a\}$ and $\{\bu_i^b\}$ are the left singular vectors, and  $\{\bv_i^a\}$ and $\{\bv_i^b\}$ are the right singular vectors. We assume that for some constant $C>0$,
\be\label{assm evalue}
0\le a_r \le \cdots \le a_2 \le a_1 \le C, \quad 0\le b_r \le \cdots \le b_2 \le b_1 \le C.
\ee
%and
%\be\label{assm evector}
%\begin{split}
%\{\bu_i^a\} \ \text{ and } \ \{\bu_i^b\} \ \text{ span the same linear space}, \\ \{\bv_i^a\} \ \text{ and } \ \{\bv_i^b\} \ \text{ span the same linear space}.
%\end{split}
%\ee
%Note that by the definition of singular value decompositions, each of the set $\{\bu_i^a\}$, $\{\bv_i^a\}$, $\{\bu_i^b\}$ and $\{\bv_i^b\}$ consists of orthonormal vectors. 
In this paper, we consider the high-dimensional setting, that is, 
\begin{equation*}
	c_1(n) := {p}/{n} \to \hat c_1 \in (0,1), \quad  c_2(n) :=  {q}/{n} \to \hat c_2 \in (0,1-\hat c_1). % \quad \text{with} \quad \hat c_1 + \hat c_2 \in (0,1).  \label{assm2}
\end{equation*}
For simplicity of notations, we will always abbreviate $c_1(n)\equiv c_1$ and $c_2(n)\equiv c_2$ for the rest of the paper. Without loss of generality, we assume that $c_1\ge c_2$. 

We now summarize the main assumptions for future reference. For our purpose, we relax the assumptions \eqref{assm1} and \eqref{assmZ} a little bit. The reader can refer to the explanation above Corollary \ref{main_cor} for the reason of this extension.

%{\color{red}On the other hand, compared with ESD, much less has been known about the convergence rate of the VESD of $Q_{1,2}$. %{\color{red}To the best of our knowledge, there is only one paper \cite{XYZ2013} studying on this topic.} In that paper, 
%The best result so far was obtained in \cite{XYZ2013}, where the authors proved that if $d_N>1$, $\Sigma=I$, and the entries of $X$ are $i.i.d.$ centered random variables, then $\|\mathbb E F^{(M)}_{Q_1,\mathbf u}-F^{(M)}_{1c,\mathbf u}\|=O(n^{-1/2})$ under the finite 10th moment assumption, and $\|F^{(M)}_{Q_1,\mathbf u}-F^{(M)}_{1c,\mathbf u}\|=O(n^{-1/4+\epsilon})$ almost surely under the finite 8th moment assumption. However, we find that both of these bounds are far away from being optimal, and can be improved with a different method. This is the main purpose of this paper. We first state our main assumptions.}

\begin{assumption}\label{main_assm}
	Fix a small constant $\tau>0$.  
	\begin{itemize}
		\item[(i)] $X=(x_{ij})$ and $Y=(Y_{ij})$ are two real independent $p\times n$ and $q\times n$ random matrices. Their entries are independent random variables that satisfy the following moment conditions: 
		\begin{align}
			\max_{i,j}\left|\mathbb{E} x_{ij}\right|   \le n^{-2-\tau}, \quad & \max_{i,j}\left|\mathbb{E} y_{ij}\right|   \le n^{-2-\tau},\label{entry_assm0} \\
			\max_{i,j}\left|\mathbb{E} | x_{ij} |^2  - n^{-1}\right|   \le n^{-2-\tau}, \quad & \max_{i,j}\left|\mathbb{E} | y_{ij} |^2  - n^{-1}\right|   \le n^{-2-\tau}. \label{entry_assm1}
		\end{align}
		%Note that (\ref{entry_assm0}) and (\ref{entry_assm1}) are slightly more general than (\ref{assm1}).
		
		\item[(ii)] $Z=(z_{ij})$ is a real $r\times n$ random matrix that is independent of $X$ and $Y$, and its entries are independent random variables that satisfy the following moment conditions: 
		\begin{align}
			\max_{i,j}\left|\mathbb{E} z_{ij}\right|   \le n^{-1-\tau}, \quad & \max_{i,j}\left|\mathbb{E} | z_{ij} |^2  - n^{-1}\right|   \le n^{-1-\tau}. \label{entry_assmz}
		\end{align} 
		
		\item[(iii)] We assume that
		\begin{equation}
			r \le \tau^{-1},\quad \tau \le c_2 \le  c_1, \quad c_1+c_2\le 1-\tau.  \label{assm20}
		\end{equation}
		
		\item[(iv)]  We consider the data model in \eqref{assm data}, where $A$ and $B$ satisfy \eqref{assm AB} and \eqref{assm evalue}.
	\end{itemize}
\end{assumption}

In this paper, we study the SCC matrix 
$$\cal C_{{\cal X}{\cal Y}}:= \big({\cal X}{\cal X}^{\top}\big)^{-1/2}  {\cal X}{\cal Y}^{\top} \big({\cal Y}{\cal Y}^{\top}\big)^{-1} {\cal Y}{\cal X}^{\top} \big({\cal X}{\cal X}^{\top}\big)^{-1/2},$$
the null SCC matrix $\cal C_{XY}:= S_{xx}^{-1/2} S_{xy} S_{yy}^{-1}S_{yx}S_{xx}^{-1/2} ,$ %\left(XY^{\top}\right)\left(YY^{\top}\right)^{-1}\left(YX^{\top}\right)  \left(XX^{\top}\right)^{-1/2},$$
where
\be\label{def Sxy}S_{xx} := {X}{ X}^{\top}, \quad S_{yy} := {Y}{ Y}^{\top}, \quad S_{xy} = S^{\top}_{yx}:=XY^{\top},\ee
and the PCC matrix $\bm\Sigma_{\cal X\cal Y}:= \Sigma_{xx}^{-1/2} \Sigma_{xy} \Sigma_{yy}^{-1} \Sigma_{yx} \Sigma_{xx}^{-1/2}, $
where
$$ \Sigma_{xx}= I_p +  A A^{\top}, \quad \Sigma_{yy}= I_q +  BB^{\top} , \quad \Sigma_{xy}= \Sigma_{yx}^{\top}=AB^{\top}.$$
Moreover, we will also consider the following matrices:
\begin{align*}
	\cal C_{{\cal Y}{\cal X}}:= \big({\cal Y}{\cal Y}^{\top}\big)^{-1/2}  {\cal Y}{\cal X}^{\top} &\big({\cal X}{\cal X}^{\top}\big)^{-1}{\cal X}{\cal Y}^{\top} \big({\cal Y}{\cal Y}^{\top}\big)^{-1/2},\\
\cal C_{YX}:= S_{yy}^{-1/2} S_{yx} S_{xx}^{-1}S_{xy}S_{yy}^{-1/2},\quad & \quad \bm\Sigma_{\cal Y\cal X}:= \Sigma_{yy}^{-1/2} \Sigma_{yx} \Sigma_{xx}^{-1} \Sigma_{xy} \Sigma_{yy}^{-1/2} .
\end{align*}
Finally, we define another null SCC matrix $\cal C^b_{\cal Y X}$ as
\be\label{Cbxy}\cal C^b_{{\cal Y}X}:= (S^b_{yy})^{-1/2} S^b_{yx}S_{xx}^{-1} S^b_{xy}(S^b_{yy})^{-1/2},\ee
where $S^b_{yy}:=\cal Y \cal Y^{\top}$ and $S^b_{xy}=(S^b_{yx})^{\top}:= X \cal Y ^\top.$ The matrix $\cal C^b_{X\cal Y }$ can be defined in the obvious way.

%Let $X=(x_{ij})$ and $Y=(y_{ij})$ be $p\times n $ and $q\times n$ data matrices, where the entries $x_{ij}$ and $y_{ij}$ are real independent random variables satisfying
%\begin{equation}\label{eq_12moment} %\label{assm1}
%\mathbb{E} x_{ij} =0, \quad \mathbb{E} \vert x_{ij} \vert^2  = n^{-1}, \quad \mathbb{E} y_{ij} =0,  \quad  \mathbb{E} \vert y_{ij} \vert^2  = n^{-1}.
%\end{equation} 
%Moreover, we assume that the random variables $x_{ij}$ and $y_{ij}$ have finite fourth moments and bounded support:
%\begin{equation}\label{eq_4moment} %\label{eqn:subgaus}
%n^2 \max_{i,j} \mathbb E|x_{ij}|^4  \le C_0, \quad n^2 \max_{i,j} \mathbb E|y_{ij}|^4  \le C_0, \quad \max_{i,j} |x_{ij}| \le \phi, \quad \max_{i,j} |y_{ij}| \le \phi,  %\var \left(h_{xy}\right)^{1/2}
%\end{equation}
%where $C_0>0$ is some constant and $\phi\equiv \phi_n \le n^{-c}$ for some constant $c_0>0$. We assume that the dimension statisfies
%$$\frac pn = a_n \to a \in (0,1), \quad \frac qn = b_{n(\Gamma)} \to b \in (0,1), \quad \text{as} \ \ n\to \infty, \quad \text{s.t.} \quad a+b \in (0,1) .$$
%WLOG, we may assume that $p\ge q$. For simplicity, we usually omit the subscript of $a_n$ and $b_{n(\Gamma)}$ from the notations, bearing in mind that all the $a$ and $b$ appearing in the proof refer to $a_n$ and $b_{n(\Gamma)}$. 

%\section{The sample CCA matrix}\label{sec_defcca}

\subsection{Preliminaries}\label{sec mainresults0}

We denote the eigenvalues of $\cal C_{YX}$ by $ \lambda_1 \ge \cdots \ge \lambda_q\ge 0$, while $\cal C_{ X Y}$ shares the same eigenvalues with $\cal C_{ Y X}$, except that it has $p-q$ more trivial zero eigenvalues $ \lambda_ {q+1} = \cdots = \lambda_{p}=0$. We denote the ESD of $\cal C_{YX}$ by 
$$ F_n(x):= \frac1q\sum_{i=1}^q \mathbf 1_{\lambda_i \le x}.$$
%If $X$ and $Y$ are both Gaussian matrices, then 
It is known \cite{Wachter,CCA_ESD} that, almost surely, $F_n$ converges weakly to a deterministic probability distribution $F(x)$ with density 
\be\label{LSD}
f(x)= \frac{1}{2\pi c_2} \frac{\sqrt{(\lambda_+ - x)(x-\lambda_-)}}{x(1-x)}\mathbf 1_{\lambda_- \le x \le \lambda_+},
\ee
where
\be\label{lambdapm}\lambda_\pm:= \left( \sqrt{c_1(1-c_2)} \pm \sqrt{c_2(1-c_1)}\right)^2.\ee
For the model \eqref{assm data}, we denote the eigenvalues of $\cal C_{\cal Y\cal X}$ by $\wt \lambda_1 \ge \wt\lambda_2 \ge \cdots \ge \wt\lambda_q\ge 0$, while $\cal C_{ \cal X \cal Y}$ has $p-q$ more trivial zero eigenvalues $\wt \lambda_ {q+1} = \cdots = \wt \lambda_{p}=0$. We denote the eigenvalues of ${\bm\Sigma}_{\cal X\cal Y}$ by 
\be\label{t1r} t_1 \ge  \cdots \ge t_r \ge t_{r+1} = \cdots = t_q=0.
\ee
Suppose the entries of $X$, $Y$ and $Z$ are i.i.d.\,Gaussian. Then, it was proved in \cite{CCA} that, if $t_i>t_c$ (recall \eqref{tc}), $\wt\lambda_i - \theta_i \to 0$ almost surely, where 
\be\label{gammai}
\theta_i : = t_i\big( 1-c_1+c_1t_i^{-1}\big) \big( 1-c_2+c_2t_i^{-1}\big) ;
\ee
if $t_i\le t_c$, $\wt\lambda_i - \lambda_+\to 0$ almost surely and $n^{2/3}( \wt\lambda_i - \lambda_+)$ converges weakly to the Tracy-Widom law. Note that for $t_i>t_c$, we have $\theta_i >\lambda_+$, so $\wt\lambda_i$ is an outlier that is detached from the support $[\lambda_-,\lambda_+]$ of the limiting distribution $F(x)$. In Section \ref{sec mainresults}, we will state our main results showing that the above BBP transition also holds without the Gaussian assumption. To state and explain the main results, we need to introduce more notations, assumptions and a preliminary result regarding the asymptotic behaviors of the eigenvalues of $\cal C_{YX}$.

In this paper, we will frequently use the following notion of stochastic domination. It was first introduced in \cite{Average_fluc} and subsequently used in many works on random matrix theory. %such as \cite{isotropic,principal,local_circular,Delocal,Semicircle,Anisotropic}. 
It simplifies the presentation of the results and their proofs by systematizing statements of the form ``$\xi$ is bounded by $\zeta$ with high probability up to a small power of $n$".

\begin{definition}[Stochastic domination and high probability event]\label{stoch_domination}
	%\begin{itemize}
	%\item[(i)] 
	(i) Let
	\[\xi=\left(\xi^{(n)}(u):n\in\bbN, u\in U^{(n)}\right),\hskip 10pt \zeta=\left(\zeta^{(n)}(u):n\in\bbN, u\in U^{(n)}\right)\]
	be two families of nonnegative random variables, where $U^{(n)}$ is an $n$-dependent parameter set. We say $\xi$ is stochastically dominated by $\zeta$, uniformly in $u$, if for any small constant $\epsilon>0$ and large constant $D>0$, 
	\[\sup_{u\in U^{(n)}}\bbP\left[\xi^{(n)}(u)>n^\epsilon\zeta^{(n)}(u)\right]\le n^{-D}\]
	for large enough $n\ge n_0(\epsilon, D)$, and we shall use the notation $\xi\prec\zeta$. Throughout this paper, the stochastic domination will always be uniform in all parameters that are not explicitly fixed. %(such as matrix indices, and the spectral parameter $z$). 
	%Note that $n_0(\epsilon, D)$ may depend on quantities that are explicitly constant, such as $\tau$ in Assumption \ref{main_assm}. 
	If  $\xi$ is complex and we have $|\xi|\prec\zeta$, then we will also write $\xi \prec \zeta$ or $\xi=\OO_\prec(\zeta)$.
	%\item[(ii)] 
	
	\vspace{5pt}
	
	\noindent (ii) We extend the definition of $\OO_\prec(\cdot)$ to matrices in the sense of operator norm as follows. Let $A$ be a family of matrices and $\zeta$ be a family of nonnegative random variables. Then, $A=\OO_\prec(\zeta)$ means that $\|A\|\prec \zeta$. 
	%$\left|\left\langle\mathbf v, A\mathbf w\right\rangle\right|\prec\zeta \| \mathbf v\|_2 \|\mathbf w\|_2 $ uniformly in any deterministic vectors $\mathbf v$ and $\mathbf w$. Here and throughout the following, whenever we say ``uniformly in any deterministic vectors", we mean that ``uniformly in any deterministic vectors belonging to a set of cardinality $n^{\OO(1)}$".
	
	\vspace{5pt}
	
	\noindent (iii) We say an event $\Xi$ holds with high probability if for any constant $D>0$, $\mathbb P(\Xi)\ge 1- n^{-D}$ for large enough $n$. Moreover, we say $\Xi$ holds with high probability on an event $\Omega$, if for any constant $D>0$, $\mathbb P(\Omega\setminus \Xi)\le n^{-D}$ for large enough $n$.
	%\end{itemize}
\end{definition}

The following lemma collects basic properties of stochastic domination $\prec$, which will be used tacitly throughout this paper.

\begin{lemma}[Lemma 3.2 in \cite{isotropic}]\label{lem_stodomin}
	Let $\xi$ and $\zeta$ be two families of nonnegative random variables, and let $C>0$ be an arbitrary constant.
	\begin{enumerate}
		\item Suppose that $\xi (u,v)\prec \zeta(u,v)$ uniformly in $u\in U$ and $v\in V$. If $|V|\le n^C$, then $\sum_{v\in V} \xi(u,v) \prec \sum_{v\in V} \zeta(u,v)$ uniformly in $u$.
		
		\item If $\xi_1 (u)\prec \zeta_1(u)$ and $\xi_2 (u)\prec \zeta_2(u)$ uniformly in $u\in U$, then $\xi_1(u)\xi_2(u) \prec \zeta_1(u)\zeta_2(u)$ uniformly in $u$.
		
		\item Suppose that $\Psi(u)\ge n^{-C}$ is deterministic and $\xi(u)$ satisfies $\mathbb E|\xi(u)|^2 \le n^C$ for all $u$. If $\xi(u)\prec \Psi(u)$ uniformly in $u$, then we have $\mathbb E\xi(u) \prec \Psi(u)$ uniformly in $u$.
	\end{enumerate}
\end{lemma}

%For the data matrices $X$ and $Y$, 
Now, we introduce a bounded support condition for random matrices considered in this paper.

\begin{definition}[Bounded support condition] \label{defn_support}
	We say a random matrix $X$ satisfies the {\it{bounded support condition}} with $\phi_n$ if
	\begin{equation}
		\max_{i,j}\vert x_{ij}\vert \prec \phi_n. \label{eq_support}
	\end{equation}
	Whenever (\ref{eq_support}) holds, we say that $X$ has support $\phi_n$. In this paper, $ \phi_n$ is always a deterministic parameter satisfying that $ n^{-{1}/{2}} \leq \phi_n \leq n^{- c_\phi} $ for some small constant $c_\phi>0$.
	%Moreover, if the entries of $X$ satisfy (\ref{size_condition}), then $X$ trivially satisfies the bounded support condition with $q=n^{-\phi}$.
\end{definition}

In this paper, we also consider the case where $|t_i -t_c|=\oo(1)$, i.e., the spike $t_i$ is very close to the BBP transition threshold. Suppose that $X$ and $Y$ have bounded support $\phi_n$ and $Z$ has bounded support $\psi_n$. Then, we make the following assumption. 

\begin{assumption}\label{ass spike} 
We assume that for some integer $0\le r_+ \leq r$, the following statement holds: %$\{i \in \N: r_+ <i \le r\}$ contains all the indices such that 
\begin{equation}\label{eq_spiket}
t_i - t_c \ge  n^{-1/3} + \psi_n + \phi_n \quad \text{if and only if} \quad  1\le  i \le r_+.
\end{equation}
The lower bound is chosen for definiteness, and it can be replaced with any $n$-dependent parameter that is of the same order.
\end{assumption}
\begin{remark} There is some freedom in choosing the two parameters $\phi_n$ and $\psi_n$. In principle, one can choose any $\phi_n,\psi_n\ll 1$ so that the following conditions hold: %the assumptions of our main results below, Theorems \ref{main_thm} and \ref{main_thm0.5}, hold. 
$$\max_{i,j}\vert x_{ij}\vert \prec \phi_n,\quad \max_{i,j}\vert y_{ij}\vert \prec \phi_n,\quad \max_{i,j}\vert z_{ij}\vert \prec \psi_n.$$
However, since smaller $\phi_n$ and $\psi_n$ lead to weaker assumptions and stronger results, it is better to choose them as small as possible. In particular, under certain moment conditions on the entries of $X$, $Y$ and $Z$ as in \eqref{condition_4e}, the best choice of $\phi_n$ and $\psi_n$ is given in \eqref{phipsi}, which comes from a standard truncation argument.
\end{remark}

%Before stating our main results on the eigenvalues of the SCC matrix $\cal C_{\cal Y\cal X}$, we describe 
%{\cor The asymptotic behaviors of the eigenvalues of $\cal C_{ Y X}$ have been well-understood in the literature e.g., \cite{PartIII}. %We denote its eigenvalues by $\lambda_1^b \ge \lambda_2^b \ge \cdots \ge \lambda_q^b $. } 

We define the quantiles of the density \eqref{LSD}, which give the classical locations of $\lambda_i$-s. 

\begin{definition} %[Classical locations of eigenvalues]
	The classical location $\gamma_j$ of the $j$-th eigenvalue is defined as
	\begin{equation}\label{gammaj}
		\gamma_j:=\sup_{x}\left\{\int_{x}^{+\infty} f(t)\dd t > \frac{j-1}{q}\right\},
	\end{equation}
	where $f$ is defined in \eqref{LSD}. Note that we have $\gamma_1 = \lambda_+$ and $\lambda_+ - \gamma_j \sim (j/n)^{2/3}$ for $j>1$.
\end{definition}

The following eigenvalue rigidity and edge universality results for $\cal C_{Y X }$ have been proved in \cite{PartIII}. 

\begin{lemma}[Theorem 2.5 of \cite{PartIII}]\label{lem null0}
	%Suppose the assumptions of Theorem \ref{main_thm} hold. Moreover, suppose that $n^{-{1}/{2}} \leq \psi_n \leq n^{- c_\psi} $ for some (small) constant $c_\psi>0$. 
Suppose Assumption \ref{main_assm} (i) and (iii) hold. Suppose $X$ and $Y$ have bounded support $\phi_n$ with $ n^{-{1}/{2}} \leq \phi_n \leq n^{- c_\phi} $ for some constant $c_\phi>0$. Assume that 
	\be\label{conditionA3} 
	\begin{split}
	\max_{i,j}\mathbb{E} | x_{ij} |^3 \lesssim n^{-3/2}, \ \  \max_{i,j}\mathbb{E} | y_{ij} |^3\lesssim n^{-3/2},\ \ \max_{i,j}\mathbb{E} | x_{ij} |^4  \prec n^{-2}, \ \  \max_{i,j}\mathbb{E} | y_{ij} |^4  \prec n^{-2}.
	\end{split}
	\ee
	Then, the eigenvalues of the null SCC matrix $\cal C_{Y X}$ satisfy the following rigidity estimate: for any constant $\delta>0$ and all $1 \le i \le  (1-\delta)q$, 
	\be\label{rigidity}
	|\lambda_i - \gamma_i | \prec i^{-1/3} n^{-2/3}. %\left[i \wedge (q+1-i)\right]^{-1/3} n^{-2/3}.
	\ee
	Moreover, we have that for any fixed $k \in \N$,
	\begin{equation}\label{joint TW0}
		\begin{split}
			\lim_{n\to \infty}\mathbb{P}&\left[ \left(n^{{2}/{3}}\frac{\lambda_{i} - \lambda_+}{c_{TW}} \leq s_i\right)_{1\le i \le k} \right] = \lim_{n\to \infty} \mathbb{P}^{GOE}\left[\left(n^{{2}/{3}}(\lambda_i - 2) \leq s_i\right)_{1\le i \le k} \right] 
		\end{split}
	\end{equation}
	for all $(s_1 , s_2, \ldots, s_k) \in \mathbb R^k$, where
	$$c_{TW}:= \left[ \frac{\lambda_+^2 (1-\lambda_+)^2}{\sqrt{c_1c_2(1-c_1)(1-c_2)}}\right]^{1/3},$$
	and $\mathbb P^{GOE}$ stands for the law of GOE (Gaussian orthogonal ensemble), which is an $n\times n$ symmetric matrix with independent (up to symmetry) Gaussian entries of mean zero and variance $n^{-1}$. 
\end{lemma}

%\begin{remark}
	Taking $k=1$ in \eqref{joint TW0}, we obtain that $n^{{2}/{3}}(\lambda_{1} - \lambda_+)/{c_{TW}} \Rightarrow F_1,$
	where $F_1$ is the famous type-1 Tracy-Widom distribution derived in \cite{TW1,TW}. Moreover, the joint distribution of the largest  $k$ eigenvalues of GOE can be written in terms of the Airy kernel for any $k$ \cite{Forr}. Hence, \eqref{joint TW0} gives a complete description of the finite-dimensional correlation functions of the edge eigenvalues of $\cal C_{Y X }$.  
	%Let $H^{GOE}$ be an $N\times N$ random matrix belonging to the Gaussian orthogonal ensemble. 
	%\begin{remark}
	%The above results were proved in \cite{CCA2} without the $BZ$ term and under the assumption that all the moments of $\sqrt{n}x_{ij}$ and $\sqrt{n}y_{ij}$ are finite. On the other hand, as one will see below, our assumption allows us to relax the moment assumption by using a standard cut-off argument.  
%\end{remark}

%Otherwise, the eigenvalue will stick to the right edge of the spectrum. 
%Moreover, the limiting distributions for the outliers and extreme non-outlier eigenvalues were also identified in \cite{CCA} under the Gaussian assumption: if $t_i >t_c$, $\sqrt{n}(\wt\lambda_i - \theta_i)$ converges to a normal distribution, while for any fixed $i> r_+$ with $t_i< t_c$, $n^{2/3}(\wt\lambda_i - \lambda_+)$ converges to a Tracy-Widom distribution. 

\subsection{The main results}\label{sec mainresults}

With the above preparations, we are now ready to state our main results on the eigenvalues of the SCC matrix $\cal C_{\cal X\cal Y}$. %--- on the convergence rate of the VESD. 
%Our main result is stated as the following theorem. 
%For a reason that will be clear later (when we prove Corollary \ref{main_cor}), we consider slightly more general random matrices $X=(x_{ij})$. More specifically, we define the following conditions for the entries of $X$: there exist constants $C_0,c_0>0$ such that for all $1\le i \le M$ and $1\le j \le N$,
%\begin{align}
%\left|\mathbb{E} x_{ij}\right|  & \le C_0 n^{-2-c_0},\label{entry_assm0} \\
%\left|\mathbb{E} | x_{ij} |^2  - n^{-1}\right|  & \le C_0n^{-2-c_0}, \label{entry_assm1}\\
%\mathbb{E} | x_{ij} |^4 & \le C_0 n^{-2}, \label{conditionA3} \\
%\left| \mathbb Ex_{ij}^2\right| &\le C_0 n^{-2-c_0}, \ \ \text{if $x_{ij}$ is complex.} \label{entry_assmex}
%\end{align}
We define the following quantities, which characterize the distances from $t_i$-s to the BBP transition threshold:  
\be\label{al+}\Delta_i:= |t_i - t_c|,\quad \al_+: =\min_{1\le i \le r}  |t_{i} - t_c| .\ee
We first describe the convergence of the outlier eigenvalues and the extreme non-outlier eigenvalues.

%\begin{remark}
%A $t_i$ that does not satisfy \eqref{eq_spiket} will give an outlier that lies within an $\OO(n^{-2/3})$ neighborhood of the edge $\lambda_+$. It is essentially indistinguishable from the extremal eigenvalue of $\cal C_{XY}$, which fluctuates around $\lambda_+$ with typical fluctuation $n^{-2/3}$. Hence in \eqref{eq_spiket}, we simply choose the supercritical $t_i$'s that give rise to outliers of the spectrum.
%\end{remark}

\begin{theorem}\label{main_thm}
	%Let $X=(x_{ij})$ be an $M\times N$ random matrix whose entries are independent random variables satisfying (\ref{entry_assm0}), (\ref{entry_assm1}), (\ref{conditionA3}) and (\ref{entry_assmex}). 
	%Suppose Assumptions \ref{main_assm} and \ref{ass spike} hold. Suppose $X$ and $Y$ have bounded support $\phi_n$ such that $ n^{-{1}/{2}} \leq \phi_n \leq n^{- c_\phi} $ for some (small) constant $c_\phi>0$. Assume that 
	%\be\label{conditionA3} 
	%\max_{i,j}\mathbb{E} | x_{ij} |^3  =\OO(n^{-3/2}), \quad \max_{i,j}\mathbb{E} | y_{ij} |^3=\OO(n^{-3/2}),\quad \max_{i,j}\mathbb{E} | x_{ij} |^4  \prec n^{-2}, \quad \max_{i,j}\mathbb{E} | y_{ij} |^4  \prec n^{-2}.
	%\ee
	%Suppose $Z$  is {\it{approximately isometric}} up to $\psi_n$ with $\psi_n=\oo(1)$. 
Suppose Assumptions \ref{main_assm} and Assumption \ref{ass spike} hold. Suppose $X$ and $Y$ have bounded support $\phi_n$ and $Z$ has bounded support $\psi_n$ with $ n^{-{1}/{2}} \leq \phi_n \leq n^{- c_\phi} $ and $n^{-{1}/{2}} \leq \psi_n \leq n^{- c_\psi} $ for some constants $c_\phi, c_\psi>0$. Assume that \eqref{conditionA3} holds. 
	Then, for any $1\le i \le r_+$, we have that
	\be\label{boundout}
	|\wt\lambda_i -\theta_i | \prec (\psi_n  +  \phi_n) \Delta_i + n^{-1/2}\Delta_i^{1/2}.
	\ee
	Let $\varpi\in \N$ be a fixed (large) integer. For any $r_+ +1 \le i \le \varpi$ and small constant $\e>0$, we have that
	\be\label{boundedge}
	-n^{-2/3+\e} < \wt\lambda_i - \lambda_+  \le n^\e (\psi_n^2 + \phi_n^2 + n^{-2/3}) \quad \text{with high probability.}	 
	\ee
	%; for any fixed $i>r$, we have
	%\be\label{boundedge2}
	%|\wt\lambda_i - \lambda_+ | \prec  n^{-2/3}. 
	%\ee
\end{theorem}

\begin{remark}
This theorem gives precise large deviation bounds on the locations of the outliers and largest few extreme non-outlier eigenvalues.  Consider a small support case with \smash{$\phi_n+\psi_n\le n^{-1/3}$} (which holds with probability $1-\oo(1)$ if we assume the existence of $12$-th moment, see \eqref{phipsi} below). Then, \eqref{boundout} and \eqref{boundedge} show that the fluctuation of the $i$-th eigenvalue changes from the order \smash{$(\psi_n  +  \phi_n) \Delta_i + n^{-1/2}\Delta_i^{1/2}$} to $n^{-2/3}$ when $\Delta_i$ crosses the scale $n^{-1/3}$. This implies the occurrence of the BBP transition. 
\end{remark}
%\begin{remark}
%In \eqref{boundout} and \eqref{boundedge}, the $\psi_n$ terms come into the estimates through a different way from the $\phi_n$ term. As a result, the factors before them (one is $C$ while one is $n^\e$) are different. In particular, the above theorem still holds if we relax the assumption on $\psi_n$ and only assume $\psi_n=\oo(1)$.
%\end{remark}

For the non-outlier eigenvalues of $\cal C_{\cal X\cal Y}$, they stick to the corresponding eigenvalues of $\cal C_{XY}$ as given by the following theorem.

\begin{theorem}\label{main_thm0.5}
	Suppose the assumptions of Theorem \ref{main_thm} hold. Assume that $\al_+ \ge n^{\e_0}(\psi_n+  \phi_n)$ for some constant $\e_0>0$. Then, we have the eigenvalue sticking estimates: 	 
%		\be\label{boundstick}
%	|\wt\lambda_{i+ r_+} - \lambda_i^b | \prec n^{-1}\al_+^{-1}
%	\ee
%	for all $i\le (1-\delta)q$, where $\delta>0$ is any small constant. 
	\be\label{boundstickextra}
	|\wt\lambda_{i+ r_+} - \lambda_i | \prec n^{-1}\al_+^{-1}
	\ee
	for all $i\le (1-\delta)q$, where $\delta>0$ is any small constant. 	
\end{theorem}

\begin{remark}
	This theorem establishes a large deviation bound on the non-outlier eigenvalues of $\cal C_{\cal X\cal Y}$ with respect to the eigenvalues of $\cal C_{X{Y}}$. Combining it with Lemma \ref{lem null0} above, we immediately obtain the asymptotic behaviors of the non-outlier eigenvalues of $\cal C_{\cal X\cal Y}$. In particular, when $\alpha_+ \gg n^{-1/3}$, the right-hand side of (\ref{boundstickextra}) is much smaller than $n^{-2/3}$. Together with \eqref{joint TW0} for $\lambda_i$, \eqref{boundstickextra} implies that the largest non-outlier eigenvalue of $\cal C_{\cal X\cal Y}$ also converges to the Tracy-Widom law as long as $t_i$-s are away from the transition threshold $t_c$ at least by $\al_+\gg n^{-1/3}$. 
\iffalse	
	Notice that applying \eqref{boundstick} to $\cal C^b_{X{\cal Y}}$ and $\cal C_{XY}$ also gives that the eigenvalue $\lambda_i^b$ sticks to $\lambda_i$ for $1\le i \le (1-\delta)q$. Thus we obtain the following eigenvalue sticking estimate
	\be\label{boundstickextra}
	|\wt\lambda_{i+ r_+} - \lambda_i | \prec n^{-1}\al_+^{-1}.
	\ee
	The reason why we state \eqref{boundstick} instead of \eqref{boundstickextra} in Theorem \ref{main_thm0.5} will be explained in equation \eqref{interlacing_eqb}. % of the supplement \cite{MY_aapsuppl}.
	\fi
\end{remark}

In many settings, people usually assume certain moment conditions on the entries of $X$, $Y$ and $Z$ instead of the bounded support condition. By using Markov's inequality and a standard truncation argument, we can derive a bounded support condition from the moment assumptions. Then, with Theorems \ref{main_thm} and \ref{main_thm0.5}, we can easily obtain the following corollary. Since we did not assume the entries of $X$, $Y$ and $Z$ are identically distributed, the means and variances of the truncated entries may be different. This is why we have assumed the slightly more general mean and variance conditions \eqref{entry_assm0}--\eqref{entry_assmz}. %{\cor For the reader's convenience, we will give its proof in Section \ref{appd_cor213} of the supplement \cite{MY_aapsuppl}.}

\begin{corollary}\label{main_cor}
	Assume that $X=(x_{ij})$, $Y=(Y_{ij})$ and $Z=(z_{ij})$ are respectively $p\times n$, $q\times n$ and $r\times n$ matrices, whose entries are real independent random variables satisfying \eqref{assm1}, \eqref{assmZ} and  
	\be\label{condition_4e} 
	\max_{i,j}\mathbb{E}  |\sqrt{n} x_{ij} | ^{a}  \le C, \quad \max_{i,j}\mathbb{E}|\sqrt{n} y_{ij} |^{a}  \le C,\quad \max_{i,j}\mathbb{E}  |\sqrt{n} z_{ij} | ^{b}  \le C ,
	\ee 
	for some constants $a>4$, $b>2$ and $C>1$. If Assumption \ref{main_assm} (iii)--(iv) and Assumption \ref{ass spike} hold with
	\be\label{phipsi}\phi_n = n^{-1/2 + 2/a } ,\quad \psi_n= n^{-1/2+1/b } ,\ee
	%and that the entries $z_{ij}$ satisfy that
	%\be\label{condition_2e} 
	%\max_{i,j}\mathbb{E}  |\sqrt{n} z_{ij} | ^{2+\tau}  \le C.
	%\ee 
	%are i.i.d. random variables satisfying \eqref{assmZ}. 
	then for any $1\le i \le r_+$ and small constant $\e>0$,
	\be\label{boundinp000}
	\lim_{n\to \infty}\P\left(|\wt\lambda_i - \theta_i| \le n^\e\left[ \left(\psi_n +  \phi_n\right) \Delta_i + n^{-1/2}\Delta_i^{1/2}\right]\right) =1 . % \ \ \text{with probability} \ 1-\oo(1),
	\ee
	Moreover, assume that the eigenvalues of $\bm\Sigma_{\cal X\cal Y}$ satisfy that %converge as $n\to \infty$ with %such that 
	\be\label{well-sep000} 
	\al_+ \ge n^{\e_0}(\psi_n+  \phi_n)+ n^{-1/3+\e_0}%\lim_n t_{r_+} > t_c > \lim_{n} t_{r_+ + 1}.
	\ee
	for a constant $\e_0>0$. Then, for any fixed $k\in \N$ and all $(s_1 , s_2, \ldots, s_k) \in \mathbb R^k$, 
	\be\label{boundinTW}
	\begin{split}
		& \lim_{n\to \infty}\mathbb{P} \left[ \left(n^{{2}/{3}}\frac{\wt\lambda_{i+r_+ } - \lambda_+}{c_{TW}} \leq s_i\right)_{1\le i \le k} \right]  = \lim_{n\to \infty} \mathbb{P}^{GOE}\left[\left(n^{{2}/{3}}(\lambda_i - 2) \leq s_i\right)_{1\le i \le k} \right].
	\end{split}
	\ee
	%Moreover, \eqref{eq_TW} holds if $\al_+ \ge n^{-1/3+\e}$ for some constant $\e>0$, and the convergence in Theorem \ref{main_thm2} holds if \eqref{assm wellsep} holds.
\end{corollary}

If the entries of $X$, $Y$ and $Z$ are identically distributed, then we can obtain the following result under a weaker tail condition than \eqref{condition_4e}. We believe this tail condition is sharp.
%Then together with \eqref{boundstick}, we obtain the following universality result for the largest non-outlier eigenvalue of $\cal C_{\cal X\cal Y}$ 

\begin{theorem}\label{main_thm1}
	Suppose Assumption \ref{main_assm} (iii)--(iv) and Assumption \ref{ass spike} hold. Assume that $x_{ij}=n^{-1/2} \wh x_{ij}$, $y_{ij}= n^{-1/2} \wh y_{ij}$ and $z_{ij}= n^{-1/2} \wh z_{ij}$, where $\{\wh x_{ij}\}$, $\{\wh y_{ij}\}$ and $\{\wh z_{ij}\}$ are three independent families of i.i.d. random variables with mean zero and variance one. Moreover, suppose the following tail condition holds:
	\begin{equation}
		\lim_{t \rightarrow \infty } t^4 \left[\mathbb{P}\left( \vert \wh x_{11} \vert \geq t\right)+ \mathbb{P}\left( \vert \wh y_{11} \vert \geq t\right)\right]=0 . \label{tail_cond}
	\end{equation}
	We assume that the eigenvalues of $\bm\Sigma_{\cal X\cal Y}$ converge as $n\to \infty$ with %such that 
	\be\label{well-sep} \lim_n t_{r_+} > t_c > \lim_{n} t_{r_+ + 1}.\ee
	%the eigenvalues of $\mathbf \Sigma_{\cal X\cal Y}$ converge as $n\to \infty$ such that \eqref{well-sep} holds. 
	Then, both \eqref{boundinTW} and the following convergence in probability hold:
	\be\label{boundinp}
	\lim_{n\to\infty}\P\left(\wt\lambda_i - \theta_i \le \e\right)=1 \ \ \text{for any constant $\e>0$}.
	\ee
	%both \eqref{boundinp} and .
	%we have 
	%\be\label{eq_TW}
	%\lim_{n\to \infty}\mathbb P \left( n^{2/3}\frac{\wt\lambda_{r_+ + 1} - \lambda_+}{c_{TW}} \le s\right) = F_{1}(s), \quad \text{for all } \ \ s\in \mathbb R, 
	%\ee
	%where $F_1$ is the Type-1 Tracy-Widom distribution,
\end{theorem}
\begin{remark}
If $\wh x_{11}$ and $\wh y_{11}$ have finite fourth moments, then the tail condition \eqref{tail_cond} holds. Hence, \eqref{tail_cond} is strictly weaker than \eqref{condition_4e}, and it gives a weaker result \eqref{boundinp} without an explicit convergence rate for \smash{$\wt\lambda_i - \theta_i$}. Note that Theorem \ref{main_thm1} cannot be derived directly from Corollary \ref{main_cor}: when $a=4$ and $b=2$, we have $\phi_n=\psi_n=1$ in \eqref{phipsi} and the result \eqref{boundinp000} becomes a trivial statement. We also remark that \eqref{well-sep} means for large enough $n$, there are exactly $r_+$ outliers and $t_i$-s are all away from the BBP transition threshold by a small constant, i.e., $\al_+\gtrsim 1$. This is consistent with Assumption \ref{ass spike} (up to a small constant) with $\phi_n=\psi_n=1$.
\end{remark}

Finally, we mention that for an outlier eigenvalue, \smash{$n^{1/2}(\wt\lambda_i - \theta_i)$} actually converges to a normal distribution, which has been proved in \cite{CCA} for the Gaussian case and for well-separated outliers, i.e. every pair $t_i$ and $t_j$ are either exactly degenerate or separated from each other by a distance of order 1. The proof for the general distribution case with near-degenerate outliers is quite involved, and, considering the length of this paper, we include it into another paper \cite{PartII}. 

%\vspace{5pt}

The rest of this paper is organized as follows.  %In Section \ref{main_result}, we define the model and state the main results, Theorem \ref{main_thm}, Theorem \ref{main_thm0.5} and Theorem \ref{main_thm1}. 
In Section \ref{sec_maintools}, we introduce the linearization method and collect some basic tools that will be used in the proof. Then, we will give the proof of Theorem \ref{main_thm} in Section \ref{sec mainthm}. The proofs of Theorem \ref{main_thm0.5}, Corollary \ref{main_cor} and Theorem \ref{main_thm1} will be presented in Sections \ref{sec pfstick}--\ref{sec mainthm3}. 
%Sections A--C. % of the supplement \cite{MY_aapsuppl}. 
%Our proofs utilize a result on the eigenvalues of the null SCC matrix,  Lemma \ref{lem null}, which will be proved in Section \ref{sec pflemma} of the supplement \cite{MY_aapsuppl}. %the supplement \cite{MY_aapsuppl}.

\section{Linearization method and local laws}\label{sec_maintools}

%\subsection{Notations}

The self-adjoint linearization method has been proved to be useful in studying the local laws of random matrices of Gram type \cite{Alt_Gram, AEK_Gram, DY_TW2, DY_TW1,Anisotropic, XYY_circular,yang2018}. We now introduce a generalization of this method, which was introduced in \cite{PartIII} to prove Lemma \ref{lem null0}. For the discussion below, we assume that $\cal X\cal X^{\top}$, $\cal Y\cal Y^{\top}$, $XX^{\top}$ and $YY^{\top}$ are all non-singular almost surely. (This is trivially true if, say, the entries of $X$, $Y$ and $Z$ have continuous densities.) Then, given $\lambda > 0$, it is an eigenvalue of $\cal C_{\cal X\cal Y}$ if and only if the following equation holds:
\be\label{eq det0}\det\left[ \cal X\cal Y^{\top}\big( \cal Y\cal Y^{\top}\big)^{-1}\cal Y\cal X^{\top} - \lambda \cal X\cal X^{\top}\right] = 0 .\ee
Using the Schur complement formula, we can check that equation \eqref{eq det0} is equivalent to
$$\det \begin{pmatrix} \lambda \cal X\cal X^{\top} &  \lambda^{1/2}\cal X\cal Y^{\top} \\  \lambda^{1/2}\cal Y\cal X^{\top} &  \lambda \cal Y\cal Y^{\top} \end{pmatrix} = 0 .$$ 
By the Schur complement formula again, the above equation is equivalent to 
\be\label{deteq}\det \begin{bmatrix} 0 & \begin{pmatrix} \cal X & 0\\ 0 & \cal Y\end{pmatrix}\\ \begin{pmatrix} \cal X^{\top} & 0\\ 0 & \cal Y^{\top}\end{pmatrix}  & \begin{pmatrix}  \lambda  I_n & \lambda^{1/2}I_n\\ \lambda^{1/2} I_n &  \lambda I_n\end{pmatrix}^{-1}\end{bmatrix} = 0 \quad \text{if \ \ $\lambda \notin \{0, 1\}$.} \ee

%if $\lambda\ne 1$.

Inspired by equation \eqref{deteq}, we define the following $(p+q+2n)\times (p+q+2n)$ symmetric block matrix %, which is a linear function of $X$:
\begin{equation}\label{linearize_block}
	H(\lambda) : = \begin{bmatrix} 0 & \begin{pmatrix}  X & 0\\ 0 &  Y\end{pmatrix}\\ \begin{pmatrix} X^{\top} & 0\\ 0 &  Y^{\top}\end{pmatrix}  & \begin{pmatrix}  \lambda  I_n & \lambda^{1/2}I_n\\ \lambda^{1/2} I_n &  \lambda I_n\end{pmatrix}^{-1}\end{bmatrix} .
\end{equation}
In general, we can extend the argument $\lambda$ to $z\in \C_+:=\{z\in \C: \im z>0\}$ and call it $H(z)$, where we take $z^{1/2}$ to be the branch with positive imaginary part. Then, using \eqref{assm data} and \eqref{assm AB}, we can write equation \eqref{deteq} as
\begin{align}\label{detereq temp}
	\det \left[ H(\lambda) + \begin{pmatrix} \bU & 0 \\ 0 & \bE\end{pmatrix}\begin{pmatrix} 0 & \cal D\\ \cal D  & 0\end{pmatrix}\begin{pmatrix} \bU^{\top} & 0 \\ 0 & \bE^{\top}\end{pmatrix} \right]=0,
\end{align} 
where $\cal D$ is a $2r\times 2r$ matrix with
\be\label{defncalD}\cal D:=\begin{pmatrix} \Sigma_a & 0 \\ 0 & \Sigma_b\end{pmatrix}, \quad \Sigma_a:=\diag \left( a_1, \cdots, a_r\right), \quad \Sigma_b:=\diag \left(b_1, \cdots, b_r\right),\ee
and $\bU$ and $\bE$ are $(p+q)\times 2r$ and $2n\times 2r$ matrices, respectively:
\be\label{defnUE} 
\begin{split}
	&\bU := \begin{bmatrix} \begin{pmatrix}\mathbf u_1^a, \cdots, \mathbf u_r^a \end{pmatrix} & 0 \\ 0 & \begin{pmatrix}\mathbf u_1^b, \cdots, \mathbf u_r^b \end{pmatrix} \end{bmatrix}, \quad
 \bE:= \begin{bmatrix} \begin{pmatrix}Z^{\top}\mathbf v_1^a, \cdots,Z^{\top} \mathbf v_r^a \end{pmatrix} & 0 \\ 0 & \begin{pmatrix}Z^{\top}\mathbf v_1^b, \cdots, Z^{\top}\mathbf v_r^b \end{pmatrix} \end{bmatrix}.
\end{split}
\ee
If $\lambda$ is not an eigenvalue of $\cal C_{XY}$, then $H(\lambda)$ is non-singular by the Schur complement formula and \eqref{detereq temp} is equivalent to 
\begin{align}\label{detereq temp2}
	\det \left[ 1+\begin{pmatrix} 0 & \cal D\\ \cal D  & 0\end{pmatrix}\begin{pmatrix} \bU^{\top} & 0 \\ 0 & \bE^{\top}\end{pmatrix} \frac1{H(\lambda)} \begin{pmatrix} \bU & 0 \\ 0 & \bE\end{pmatrix}\right]=0,
\end{align} 
where we used the identity $\det(1+ M_1M_2)= \det(1+ M_2M_1)$ for any matrices $M_{1}$ and $M_{2}$ of conformable dimensions. 
% Then we can write 
%$$\begin{pmatrix} 0 & \begin{pmatrix} TY & 0\\ 0 &T^{\top} X\end{pmatrix}\\ \begin{pmatrix} X^{\top}T^{\top} & 0\\ 0 & \cal Y^{\top} T\end{pmatrix}  & 0\end{pmatrix} = \begin{pmatrix} \bE & 0 \\ 0 & \bU\end{pmatrix}\begin{pmatrix} 0 & \cal D\\ \cal D  & 0\end{pmatrix}\begin{pmatrix} \bE^{\top} & 0 \\ 0 & \bU^{\top}\end{pmatrix},$$
%where
%$$\bU:=\begin{pmatrix} Y_1 & 0\\ 0 & X_1\end{pmatrix}.$$
Inspired by the above discussion, we define the resolvent (or Green's function) 
\begin{equation}\label{eqn_defG}
	G(z):= \left[H(z)\right]^{-1} , \quad z\in \mathbb C_+ \cup \R,
\end{equation}
whenever the inverse exists. Note that although $H(\lambda)$ is not well-defined for $\lambda=1$, we can still define $ G(1)=\lim_{z\to 1} G(z)$ using the Schur complement, see \eqref{GL1} and \eqref{GR1} below.  In order to study the eigenvalues of $\cal C_{\cal X\cal Y}$, we need to obtain some estimates on the $4r\times 4r$ matrix 
$$\begin{pmatrix} \bU^{\top} & 0 \\ 0 & \bE^{\top}\end{pmatrix} G(\lambda) \begin{pmatrix} \bU & 0 \\ 0 & \bE\end{pmatrix}.$$
This is provided by the {\it anisotropic local law} of $ G(z)$, which is one of the main results in \cite{PartIII}. We will state it in Theorem \ref{thm_local} below. %are the main results of this section. 

For the proof of Theorem \ref{main_thm0.5}, we will also use a different representation of \eqref{detereq temp2}: if $\lambda$ is not an eigenvalue of $\cal C^b_{X \cal Y}$, then $\lambda$ is an eigenvalue of $\cal C_{
	\cal X\cal Y}$ if and only if
\begin{align}\label{detereq temp3}
	\det \left[ 1+\begin{pmatrix} 0 & \cal D_a\\ \cal D_a  & 0\end{pmatrix}\begin{pmatrix} \bU_a^{\top} & 0 \\ 0 & \bE_a^{\top}\end{pmatrix} G^b(\lambda) \begin{pmatrix} \bU_a & 0 \\ 0 & \bE_a\end{pmatrix}\right]=0,
\end{align}
where  
\begin{equation}\label{eqn_defGb}
	G^b(z):= \big[H^b(z)\big]^{-1} , \quad  H^b(z) : = \begin{bmatrix} 0 & \begin{pmatrix}  X & 0\\ 0 &  \cal Y\end{pmatrix}\\ \begin{pmatrix} X^{\top} & 0\\ 0 &  \cal Y^{\top}\end{pmatrix}  & \begin{pmatrix}  z I_n & z ^{1/2}I_n\\ z^{1/2} I_n &  z I_n\end{pmatrix}^{-1}\end{bmatrix} ,
\end{equation}
%and
$$\cal D_a:=\begin{pmatrix} \Sigma_a & 0 \\ 0 & 0\end{pmatrix}, \quad \bU_a: = \begin{pmatrix} \begin{pmatrix}\mathbf u_1^a, \cdots, \mathbf u_r^a \end{pmatrix} & 0 \\ 0 & 0 \end{pmatrix}, \quad \bE_a:= \begin{pmatrix} \begin{pmatrix}Z^{\top}\mathbf v_1^a, \cdots,Z^{\top} \mathbf v_r^a \end{pmatrix} & 0 \\ 0 & 0\end{pmatrix}.$$

For simplicity of notations, we introduce the following index sets for linearized matrices.
%First we introduce some notations.
\begin{definition}[Index sets]\label{def_index}
	We define the index sets
	$$\cal I_1:=\llbracket 1,p\rrbracket, \quad \cal I_2:=\llbracket p+1,p+q\rrbracket,\quad \cal I_3:=\llbracket p+q+1,p+q+n\rrbracket, \quad \cal I_4:=\llbracket p+q+n+1,p+q+2n\rrbracket. $$
	We will consistently use latin letters $i,j\in\sI_{1}\cup \sI_2$ and greek letters $\mu,\nu\in\sI_{3}\cup \sI_4$. Moreover, we will use notations $\fa,\fb\in \cal I:=\cup_{i=1}^4 \cal I_i$. 
	% We label the indices of the random matrices according to
	% $$X= (x_{i\mu}:i\in \mathcal I_1, \mu \in \mathcal I_3), \quad Y= (y_{j\nu}:j\in \mathcal I_2, \nu \in \mathcal I_4).$$ 
	% \quad A= (A_{ij}:i\in \mathcal I_1, j \in \mathcal I_2),\quad B= (B_{ij}:i\in \mathcal I_1, j \in \mathcal I_2).$$
	%For $i\in \cal I_1$, $j\in \cal I_2$, $\mu \in \cal I_3$ and  $\nu \in \cal I_4$, we denote $\overline i:= i+p$, $\overline j:= j-p$, $\overline \mu : = \mu +n $ and $\overline \nu : = \nu - n $. 
\end{definition}

Next, we define several other types of resolvents that will be used in the proof. 

\begin{definition}[Resolvents]\label{resol_not}
	%For $z = E+ \ii \eta \in \mathbb C_+,$ we define the resolvents $G(z)$. 
We denote the $(\cal I_1\cup \cal I_2)\times (\cal I_1\cup \cal I_2)$ block of $ G(z)$ by $ \cal G_L(z)$, the $(\cal I_1\cup \cal I_2)\times (\cal I_3\cup \cal I_4)$ block by $ \cal G_{LR}(z)$, the $(\cal I_3\cup \cal I_4)\times (\cal I_1\cup \cal I_2)$ block by $ \cal G_{RL}(z)$, and the $(\cal I_3\cup \cal I_4)\times (\cal I_3\cup \cal I_4)$ block by $ \cal G_R(z)$. We denote the $\cal I_\al \times \cal I_\al$ block of $ G(z)$ by $ \cal G_\al(z)$ for $\al=1,2,3,4$.  Then, we define the partial traces 
	$$ m_\al(z) :=\frac1n\tr  \cal G_{\al}(z) = \frac{1}{n}\sum_{\fa \in \cal I_\al}  G_{\fa\fa}(z) ,\quad \al=1,2,3,4. $$
	Recalling the notations in \eqref{def Sxy}, we define $\cal H:=S_{xx}^{-1/2}S_{xy}S_{yy}^{-1/2}$ and
	\be\label{Rxy}
	R_1(z):%=(\cal C_{XY}-z)^{-1}
	=(\cal H\cal H^{\top}-z)^{-1}, \quad  R_2(z):%=(\cal C_{YX}-z)^{-1}
	=(\cal H^{\top}\cal H-z)^{-1}, \quad  m(z):= q^{-1}\tr  R_2(z).
	\ee
	Note that we have $R_1\cal H = \cal HR_2$, $\cal H^{\top} R_1 = R_2 \cal H^{\top} $, and 
	\be\label{R12} \tr  R_1 = \tr  R_2 - \frac{p-q}{z}= q  m(z) - \frac{p-q}{z},\ee
	since $\cal C_{XY}=\cal H\cal H^{\top}$ has $p-q$ more zero eigenvalues than $\cal C_{YX}=\cal H^{\top}\cal H$. Moreover, we define 
	$$R(z):= \begin{pmatrix} -z & -z^{1/2}\cal H \\ -z^{1/2}\cal H^\top & - z\end{pmatrix}^{-1}.$$
	Finally, we can define ${\cal G}^b_L(z)$, ${\cal G}^b_R(z)$, $ m^b_\al(z)$, $\cal H^b$, $R^b$ etc.\;in the obvious way by replacing $Y$ with $\cal Y$. 
	%for $\wt {\mathcal Q}_{1,2}$ as
	%\begin{equation}\label{def_green}
	%\mathcal G_1(X,z):=\left(\wt{\mathcal Q}_1(X) -z\right)^{-1} , \ \ \ \mathcal G_2 (X,z):=\left(\wt{\mathcal Q}_2(X)-z\right)^{-1} .
	%\end{equation}
	% We denote the ESD $\rho^{(n)}$ of $\wt {\mathcal Q}_{1}$ and its Stieltjes transform as
	%\be\label{defn_m}
	%\rho\equiv \rho^{(n)} := \frac{1}{n} \sum_{i=1}^n \delta_{\lambda_i(\wt{\mathcal Q}_1)},\quad m(z)\equiv m^{(n)}(z):=\int \frac{1}{x-z}\rho_{1}^{(n)}(dx)=\frac{1}{n} \mathrm{Tr} \, \mathcal G_1(z).
	%\ee
	%We also introduce the following quantities:
	%$$m_1(z)\equiv m_1^{(n)}(z):= \frac{1}{N}\sum_{i=1}^n\sigma_i (\mathcal G_1)_{ii}(z) ,\quad m_2(z)\equiv m_2^{(N)}(x):=\frac{1}{N}\sum_{\mu=1}^N \wt\sigma_\mu (\mathcal G_2)_{\mu\mu}(z). $$
\end{definition}

Using the Schur complement formula, we can check that
$$ R(z):=\begin{pmatrix} R_1  & - z^{-1/2} R_1\cal H  \\ - z^{-1/2} \cal H^{\top} R_1  &  R_2 \end{pmatrix}.$$
Let $\cal H = \sum_{k = 1}^{q} \sqrt {\lambda_k} \xi_k \zeta _{k}^{\top} $ be a singular value decomposition of $\cal H$, where
$\sqrt{\lambda_1}\ge  \cdots \ge \sqrt{\lambda_{q}} \ge 0 = \sqrt{\lambda_{q+1}} = \cdots =\sqrt{ \lambda_{p}}$ are the singular values, $\{\xi_{k}\}_{k=1}^{p}$ are the left-singular vectors, and $\{\zeta_{k}\}_{k=1}^{q}$ are the right-singular vectors. Then, we have the following eigendecomposition of $R(z)$:
\begin{equation}
	R\left( z \right) = \sum\limits_{k = 1}^q \frac{1}{\lambda_k-z}\left( {\begin{array}{*{20}c}
			{{\xi _k \xi _k^{\top}  }} & {-z^{-\frac12}\sqrt {\lambda _k } \xi _k \zeta _{ k}^{\top}}  \\
			{-z^{-\frac12} \sqrt {\lambda _k } \zeta _{k} \xi _k^{\top}  } & {\zeta _k\zeta _k^{\top} }  \\
	\end{array}} \right) - \frac1z \left( {\begin{array}{*{20}c}
			{\sum_{k=q+1}^p{\xi _k \xi _k^{\top}  }} & 0  \\
			{0 } & {0}  \\
	\end{array}} \right). \label{spectral1}
\end{equation}
On the other hand, applying the Schur complement formula to $G(z)$, we get that 
\be\label{GL1}
\begin{split}
	\cal G_L %& = \begin{pmatrix} S_{xx}^{-1/2}R_1S_{xx}^{-1/2} & - z^{-1/2}S_{xx}^{-1/2}R_1\cal HS_{yy}^{-1/2} \\ - z^{-1/2}S_{yy}^{-1/2}\cal H^{\top} R_1S_{xx}^{-1/2} & S_{yy}^{-1/2}R_2S_{yy}^{-1/2}\end{pmatrix} \\
	&=  \begin{pmatrix} S_{xx}^{-1/2} & 0 \\ 0 & S_{yy}^{-1/2} \end{pmatrix}R(z) \begin{pmatrix} S_{xx}^{-1/2} & 0 \\ 0 & S_{yy}^{-1/2} \end{pmatrix}. % \\
	%&= \begin{pmatrix} S_{xx}^{-1/2}R_1S_{xx}^{-1/2} & - z^{-1/2}S_{xx}^{-1/2}\cal H R_2 S_{yy}^{-1/2} \\ - z^{-1/2}S_{yy}^{-1/2}R_2\cal H ^{\top} S_{xx}^{-1/2} & S_{yy}^{-1/2}R_2S_{yy}^{-1/2}\end{pmatrix} 
\end{split}
\ee
%where 
%$$\cal G_1= S_{xx}^{-1/2}R_1S_{xx}^{-1/2} = \left(S_{xy}S_{yy}^{-1}S_{yx} - z S_{xx}\right)^{-1}, \quad \cal G_2 = S_{yy}^{-1/2}R_2S_{yy}^{-1/2}= \left(S_{yx}S_{xx}^{-1}S_{xy}  - z S_{yy}\right)^{-1}.$$
Moreover, the other blocks take the forms
\begin{align}
&\cal G_R =   \begin{pmatrix}  z  I_n & z^{\frac12}I_n\\ z^{\frac12}I_n &  z  I_n\end{pmatrix} +   \begin{pmatrix}  z  I_n & z^{\frac12}I_n\\ z^{\frac12}I_n &  z  I_n\end{pmatrix}  \begin{pmatrix} X^{\top} & 0 \\ 0 & Y^{\top} \end{pmatrix} \cal G_L \begin{pmatrix} X & 0 \\ 0 &  Y \end{pmatrix} \begin{pmatrix}  z  I_n & z^{\frac12}I_n\\ z^{\frac12}I_n &  z  I_n\end{pmatrix}  ,\label{GR1}\\
%\ee
%and
%\be
%\begin{split}
& {\cal G}_{LR}(z)= -\cal G_L(z) \begin{pmatrix} X & 0 \\ 0 &  Y \end{pmatrix} \begin{pmatrix}  z  I_n & z^{\frac12}I_n\\ z^{\frac12}I_n &  z  I_n\end{pmatrix}  , \quad  {\cal G}_{RL}(z)= {\cal G}_{LR}(z)^\top . \label{GLR1} %-  \begin{pmatrix}  z  I_n & z^{\frac12}I_n\\ z^{\frac12}I_n &  z  I_n\end{pmatrix}  \begin{pmatrix} X^{\top} & 0 \\ 0 & Y^{\top} \end{pmatrix} {\cal G}_L(z).
%\end{split}
\end{align}
%Denoting
%$$S_x:= XX^{\top}, \quad S_y:= YY^{\top}, \quad A:=S_x^{-1/2}(XY^{\top})S_y^{-1/2}, \quad R_1:=(AA^{\top}-z)^{-1}, \quad R_2:=(A^{\top}A-z)^{-1}.$$
%Then we can write 
%$$ \cal G_1= S_x^{-1/2}R_1S_x^{-1/2},\quad  \cal G_2= S_y^{-1/2}R_2S_y^{-1/2}, \quad \cal G_L = \begin{pmatrix} S_x^{-1/2}R_1S_x^{-1/2} & - z^{-1/2}S_x^{-1/2}R_1AS_y^{-1/2} \\ - z^{-1/2}S_y^{-1/2}A^{\top} R_1S_x^{-1/2} & S_y^{-1/2}R_2S_y^{-1/2}\end{pmatrix} .$$
Expanding the product in \eqref{GR1} using \eqref{GL1} and calculating partial traces, we can check that
\begin{align}
m_3 (z)&= z+\frac1n\left( -2 z p - z^{2}\tr R_1 + z \tr  R_2\right)= c_2 z(1-z)m(z) + (1-c_1-c_2)z,\label{m3m}\\
%and 
%\be
%\begin{split}
	m_4(z) &= z+\frac1n\left(  - 2z q - z^{2}\tr  R_2+ z  \tr R_1 \right) =  c_2 z(1-z)m(z) - (c_1-c_2)+ (1- 2c_2) z,\label{m4m}
%\end{split}
\end{align}
where we also used \eqref{R12} in the derivations. In particular, we have the identity
\be\label{m34}
m_3(z) - m_4 (z)= (1-z)(c_1-c_2) . 
\ee
We remark that all the above identities also hold for $G^b$, ${\cal G}^b_L(z)$, ${\cal G}^b_R(z)$, $ m^b_\al(z)$  etc.\;with some obvious changes of notations.

Since $S_{xx}$ and $S_{yy}$ are standard sample covariance matrices, it is well-known that their eigenvalues are all inside the supports of Marchenko-Pastur laws, $[(1-\sqrt{c_1})^2 , (1+\sqrt{c_1})^2]$ and $[(1-\sqrt{c_2})^2 , (1+\sqrt{c_2})^2]$, with probability $1-\oo(1)$ \cite{No_outside}. %Hence both $S_{xx}^{-1}$ and $S_{yy}^{-1}$ are well-behaved under the assumption \eqref{assm20}. 
In our proof, we will need some slightly stronger estimates on the extreme eigenvalues of $S_{xx}$ and $S_{yy}$, denoted by $\lambda_1 (S_{xx}) \ge \lambda_p (S_{xx})$ and $\lambda_1 (S_{yy}) \ge  \lambda_q (S_{yy})$, which are given by the following lemma. 

\begin{lemma}\label{SxxSyy}
	Suppose Assumption \ref{main_assm} holds. Suppose $X$ and $Y$ have bounded support $\phi_n$ and $Z$ has bounded support $\psi_n$ with $ n^{-{1}/{2}} \leq \phi_n \leq n^{- c_\phi} $ and $n^{-{1}/{2}} \leq \psi_n \leq n^{- c_\psi} $ for some constants $c_\phi, c_\psi>0$. 	
	Then, for any constant $\e>0$, we have that with high probability,
	\begin{align}
	&\label{op rough1} (1-\sqrt{c_1})^2 - \e \le  \lambda_p(S_{xx})   \le  \lambda_1(S_{xx}) \le (1+\sqrt{c_1})^2 + \e ,
	\\
	%and
	&\label{op rough2} (1-\sqrt{c_2})^2 - \e \le  \lambda_q(S_{yy})  \le  \lambda_1(S_{yy}) \le (1+\sqrt{c_2})^2 + \e .
	\end{align}
	Moreover, there exists a constant $c>0$ such that with high probability,
	\be\label{op rough2.5} c \le  \lambda_q(S^b_{yy})  \le  \lambda_1(S^b_{yy}) \le c^{-1} ,
	\ee
	where $ \lambda_1(S^b_{yy})$ and $ \lambda_q(S^b_{yy})$ are respectively the largest and smallest eigenvalues of  $S^b_{yy}$.
\end{lemma}
\begin{proof}
	The estimates \eqref{op rough1} and \eqref{op rough2} have been proved in Lemma 3.3 of \cite{PartIII}. To get \eqref{op rough2.5}, we write
	$$S_{yy}^b = \begin{pmatrix} I_q, B\end{pmatrix} WW^{\top}\begin{pmatrix} I_q\\ B^{\top}\end{pmatrix} , \quad W:=\begin{pmatrix} Y\\ Z\end{pmatrix}.$$
	Since $r/n\to 0$, the estimate \eqref{op rough2} applied to $WW^{\top}$ gives  that with high probability,
	$$(1-\sqrt{c_2})^2 - \e \le  \lambda_{q+r}(WW^{\top})   \le  \lambda_1(WW^{\top}) \le (1+\sqrt{c_2})^2 + \e.$$
	Then, using that for any unit vector $\bv\in \R^{q}$, $\|\mathbf v\|  \sim \|\mathbf u\|$ for $\bu:=\begin{pmatrix} I_q\\ B^{\top}\end{pmatrix}\bv$, we conclude \eqref{op rough2.5}.
\end{proof}

%We remark that all the above resolvent identities or estimates also hold for $G^b(z)$, but we do not present them in order to simplify the presentation.
%
%\vspace{5pt}

Let $m_{\al c}$ be the asymptotic limits of $m_\al$ for $\al=1,2,3,4$. In \cite{PartIII}, we have obtained that %as $n\to \infty$:
\begin{align}
	&m_{1c}(z) %= \frac{ - (1-2a) z - (a- b)+\sqrt{(z-d_-)(z-d_+)} }{2(1-a)(1-z^2)}
	= \frac{ - z +c_1+c_2+\sqrt{(z-\lambda_-)(z-\lambda_+)} }{2(1-c_1)z(1-z)} - \frac{c_1}{(1-c_1)z}, \label{m1c}\\
	&m_{2c}(z) = \frac{ -z +c_1 + c_2+ \sqrt{(z-\lambda_-)(z-\lambda_+)}}{2(1-c_2)z(1-z)} -\frac{c_2}{(1-c_2)z}, \label{m2c}\\
	&m_{3c}(z) %&= \frac{1}{2}\left[ (1-2a) z + a- b + \sqrt{z^2- 2(a+b-2ab)z +(a-b)^2}\right] \\
	= \frac{1}{2}\left[ (1-2c_1) z + c_1 - c_2 + \sqrt{(z-\lambda_-)(z-\lambda_+)}\right] , \label{m3c}\\
	&m_{4c}(z)= \frac{1}{2}\left[ (1-2c_2) z +  c_2 - c_1 + \sqrt{(z-\lambda_-)(z-\lambda_+)}\right] ,\label{m4c}
\end{align}
where $\lambda_{\pm}$ are defined in \eqref{lambdapm}. It is easy to see that when $z\to 1$, both $m_{1c}(z)$ and $m_{2c}(z)$ have finite limits, and without loss of generality, we still denote them by $m_{1c}(1)$ and $m_{2c}(1)$. With \eqref{m3m}, we can easily obtain  the asymptotic limit of $m(z)$ as
\begin{align}\label{mc}
	&m_c(z)= \frac{m_{3c}(z) + (c_1+c_2 - 1)z}{c_2z(1-z)} = \frac{1-c_2}{c_2} m_{2c}(z).
\end{align}
Through direct calculations, we can check that $m_{\al c}$-s satisfy the following equations:
\begin{align}
	& m_{1c}= - \frac{c_1}{m_{3c}} , \quad m_{2c} = -\frac{ c_2}{m_{4c}}, \quad  m_{3c}(z) - m_{4c} (z)= (1-z)(c_1-c_2) .\label{selfm12}
	%& m_{3c}(z) %=\frac{(z-1) \left[1-(z-1)m_{2c}(z)\right]}{\left[1-(z-1)m_{1c}(z)\right]\left[1-(z-1)m_{2c}(z)\right] -  z^{-1}}
	%=\frac{ 1-(z-1)m_{2c}(z)}{z^{-1} - [m_{1c}(z)+m_{2c}(z)] + (z-1)m_{1c}(z)m_{2c}(z)}, \label{selfm3}  \\
	%& m_{4c}(z) = \frac{(z-1) \left[1-(z-1)m_{1c}(z)\right]}{\left[1-(z-1)m_{1c}(z)\right]\left[1-(z-1)m_{2c}(z)\right] -  z^{-1}},\label{selfm4}\\
	%& m_{3c}^2(z) + \left[ (2c_1 -1)z - c_1+c_2\right]m_{3c}(z) + c_1(c_1-1)z(z-1) =  0. \label{selfm32}
\end{align}
Finally, we introduce the function
\be
\begin{split}\label{hz}
	h(z):&=  \frac{z^{-1/2}m_{3c}(z)}{1+(1-z)m_{2c}(z)} =  \frac{z^{-1/2}m_{4c}(z)}{1+(1-z)m_{1c}(z)}	\\
	&= \frac{z^{1/2}}{2} \left[ - z + (2-c_1-c_2) + \sqrt{(z-\lambda_-)(z-\lambda_+)}\right].
\end{split}
\ee
Now, with the functions $m_{\al c}$ and $h$, we can define the matrix limit of $G(z)$ as
\be \label{defn_pi}
\Pi(z) := \begin{bmatrix} \begin{pmatrix} c_1^{-1}m_{1c}(z)I_p & 0\\ 0 & c_2^{-1}m_{2c}(z)I_q\end{pmatrix} & 0 \\ 0  & \begin{pmatrix}  m_{3c}(z)I_n  & h(z)I_n\\  h(z)I_n &  m_{4c}(z)  I_n\end{pmatrix}\end{bmatrix} .\ee

Given $z=E+\ii \eta$, we define its distance (along the real axis) to the two edges as
\begin{equation}
	\kappa \equiv  \kappa_E := \min\left\{ \vert E -\lambda_-\vert,\vert E -\lambda_+\vert\right\}  .\label{KAPPA}
\end{equation}
We have the following lemma, which can be proved through direct calculations using \eqref{m1c}--\eqref{m4c}. 

\begin{lemma}\label{lem_mbehavior}
	 If \eqref{assm20} holds, then the following estimates hold for any constants $c, C>0$..
	\begin{itemize}
		\item[(1)] For $z\in \C_+ \cap \{z: c\le |z| \le C\}$, we have
		%\begin{equation}\label{absmc}
		%\quad \left|z^{-1} - [m_{1c}(z)+m_{2c}(z)] + (z-1)m_{1c}(z)m_{2c}(z) \right|\sim 1,
		% \ee 
		% and
		\be\label{Immc} 
		\vert m_{c}(z) \vert \sim 1,  \quad 0\le \im m_{c}(z) \sim \begin{cases}
			{\eta}/{\sqrt{\kappa+\eta}}, &   \ E \notin [\lambda_-,\lambda_+] \\
			\sqrt{\kappa+\eta}, &  \ E \in [\lambda_-,\lambda_+]\\
		\end{cases}.
	\end{equation}
	
	%\begin{equation}
	%\rho_{1,2c}(x) \sim \sqrt{\lambda_+-x}, \quad \ \ \text{ for } x \in \left[\lambda_+ - 2\wt c,\lambda_+ \right];\label{SQUAREROOT}
	%\end{equation}
	
	\item[(2)] For $z,z_1, z_2  \in \C_+ \cap \{z: c\le |z| \le C\} \cap \{\re z> \lambda_+\}$, we have
	\begin{align}
		& |m_{c}(z) - m_{c}(\lambda_+)| \sim |z-\lambda_+|^{1/2} , \quad |m_{c}'(z) | \sim |z-\lambda_+|^{-1/2},\label{eq_mcomplex}\\
	%\end{equation}
	%\item[(3)] For $z_1, z_2 \in \C_+ \cap \{z: c\le |z| \le C\} \cap \{\re z> \lambda_+\}$, we have
	%and
	%\begin{equation}
		& |m_{c}(z_1) - m_{c}(z_2)|   \sim  \frac{|z_1-z_2|}{\max_{i=1,2}|z_i-\lambda_+|^{1/2}}.\label{eq_mdiff}
	\end{align}
	
	%\begin{equation}\label{Immc}
	%\vert m_{1,2c}(z) \vert \sim 1,  \quad  \im m_{1,2c}(z) \sim \begin{cases}
	%    {\eta}/{\sqrt{\kappa+\eta}}, & \text{ if } E\geq \lambda_+ \\
	%    \sqrt{\kappa+\eta}, & \text{ if } E \le \lambda_+\\
	%  \end{cases};
	%\end{equation}
	%%for $z = E+\ii \eta\in S(\wt c,C_0,\omega)$;
	%\item[(3)] there exists constant $\tau'>0$ such that
	%\begin{equation}\label{Piii}
	%\min_{\mu\in \mathcal I_2} \vert 1 + m_{1c}(z)\wt \sigma_\mu \vert \ge \tau', \quad \min_{i\in \mathcal I_1} \vert 1 + m_{2c}(z)\sigma_i  \vert \ge \tau',
	%\end{equation}
	%for any $z \in S(\wt c,C_0,-\infty)$.
\end{itemize}
The above estimates also hold for $m_{\al c}$, $\al=1,2,3,4$. Finally, $h(z)$ also satisfies \eqref{eq_mcomplex}, \eqref{eq_mdiff} and the first estimate in \eqref{Immc}. 
\end{lemma}
%$$ \sqrt{z^2- 2(a+b-2ab)z +(a-b)^2}$$
%$$\left( \sqrt{c_1(1-c_2)} \pm \sqrt{c_2(1-c_1)}\right)^2$$

%We define the deterministic limit $\Pi$ of the resolvent $G$ in (\ref{eqn_defG}) as
%\begin{equation}\label{defn_pi}
%\Pi (z): = \left( {\begin{array}{*{20}c}
%   { -\left(1+m_{2c}(z)\Sigma \right)^{-1} } & 0  \\
%   0 & { - z^{-1} (1+m_{1c}(z)\wt \Sigma )^{-1} }  \\
%\end{array}} \right) .
%\end{equation}
%Note that we have
%\be\label{mcPi}
%\frac1{nz}\sum_{i\in \mathcal I_1} \Pi_{ii} =m_c. 
%\ee

For simplicity of notations, we introduce the following notion of generalized entries.
\begin{definition}[Generalized entries]\label{defn_gen_entry}
Given $\mathbf v,\mathbf w \in \mathbb C^{\mathcal I}$, $\fa\in \mathcal I$ and $\mathcal I\times \mathcal I$ matrix $\cal A$, we denote
\begin{equation}
	\cal A_{\mathbf{vw}}:=\langle \mathbf v,\cal A\mathbf w\rangle, \quad  \cal A_{\mathbf{v}\fa}:=\langle \mathbf v,\cal A\mathbf e_\fa\rangle, \quad \cal A_{\fa\mathbf{w}}:=\langle \mathbf e_\fa,\cal A\mathbf w\rangle,
\end{equation}
where $\mathbf e_\fa$ is the standard unit vector along the $\fa$-th coordinate axis, and the inner product is defined as $\langle \mathbf v, \mathbf w\rangle:= \bv^* \bw$ with $\bv^*$ being the conjugate transpose of $\bv$. Given a vector $\mathbf v\in \mathbb C^{\mathcal I_\al}$, $\al=1,2,3,4$, we always identify it with its natural embedding in $\C^{\cal I}$. For example, we shall identify $\mathbf v\in \mathbb C^{\mathcal I_1}$ with a vector $\bv' \in \C^{\cal I}$ with $\bv'(i)=\bv(i)$ for $i\in \cal I_1$ and $\bv'(i)=0$ for $i\notin \cal I_1$.
%The exact meanings will be clear from the context.
\end{definition}

%The local laws for $G(z)$ take different forms for $z$ in 
We define the following spectral domains for the local laws of $G(z)$.
\begin{definition}[Spectral domains]
For any constant $\e >0$, we define the following two domains: %following two domains for the spectral parameter $z$:
\begin{align}
	S(\e)&:= \left\{z=E+ \ii \eta: \e \leq E \leq 2, n^{-1+\e} \leq \eta \leq \e^{-1} \right\}, \label{SSET1}\\
%\end{equation}
%and
%\be
S_{out}(\e)&:=S(\epsilon)\cap \{z=E+\ii\eta: E\notin [\lambda_-,\lambda_+], n\eta\sqrt{\kappa + \eta} \ge n^\epsilon\}.\label{SSETOUT}
\end{align}
Correspondingly, we define the following two domains that are away from $z=1$: for any fixed $\wt\e>0$,
$$
\wt S(\e,\wt \e):= \left\{z=E+ \ii \eta: \e \leq E \leq 1-\wt\e, n^{-1+\e} \leq \eta \leq \e^{-1} \right\}, \ \ \wt S_{out}(\e,\wt\e):= \wt S(\e,\wt \e)\cap  S_{out}(\e).
$$
\end{definition}

Now, we are ready to state the main result of this section, i.e., the local laws for $G(z)$. For $z=E+\ii\eta$, we define the control parameter
\begin{equation}\label{eq_defpsi}
\Psi (z):= \sqrt {\frac{\Im \, m_{c}(z)}{{n\eta }} } + \frac{1}{n\eta}.
\end{equation}
%Note that by (\ref{Immc}), we have that for $z\in  \C_+ \cap \{z: c\le |z| \le C\}$,
%\begin{equation}\label{psi12}
%\|\Pi\|=\OO(1), \quad \Psi(z) \gtrsim n^{-1/2} , \quad \Psi^2(z) \lesssim (n\eta)^{-1}, \quad \Psi(z) \sim  \sqrt {\frac{\Im \, m_{\al c}(z)}{{n\eta }} } + \frac{1}{n\eta} \ \ \text{with} \ \ \al=1,2,3,4.
%\end{equation}
%As a convention, we denote
%\begin{equation}
%S(-\infty):= \left\{z=E+ \ii \eta: \frac12 \lambda_- \leq E \leq 2, 0 \leq \eta \leq 1 \right\}. \label{SSETinfty}
%\end{equation}
%In defining these domains, we only include the right edge $\lambda_+$ of the spectrum, since the regime around the left edge $\lambda_-$ is irrelevant to our proof. 
%In particular, we shall denote
%\begin{equation}
%S(c_0,C_0,-\infty):= \left\{z=E+ \ii \eta: \lambda_+ - c_0 \leq E \leq C_0 \lambda_+, 0 \leq \eta \leq 1 \right\}.
%\end{equation}

\begin{theorem} [Theorem 2.13 and Theorem 2.14 of \cite{PartIII}]\label{thm_local} %[Results on covariance matrices with small support]
Suppose the assumptions of Lemma \ref{lem null0}  hold. Then, for any fixed $\wt\e,\e>0$, the following estimates hold. 
\begin{itemize}
	\item[(1)] {\bf Anisotropic local law}: For any $z\in S(\epsilon)$ and deterministic unit vectors $\mathbf u, \mathbf v \in \mathbb C^{\mathcal I}$, we have
	\begin{equation}\label{aniso_law}
		\left|   G_{\mathbf u\mathbf v}(z)   -   \Pi_{\mathbf u\mathbf v} (z)  \right| \prec \phi_n + \Psi(z).
	\end{equation}
	
	\item[(2)] {\bf Averaged local law}: For any $z \in \wt S(\epsilon,\wt \e)$,  we have %(\cor need to check\nc)
	\begin{equation}
		\vert m(z)-m_{c}(z) \vert \prec (n \eta)^{-1}. \label{aver_in}
	\end{equation}
	%\begin{equation}
	%%|m(z)-m_c(z)| + 
	%\vert \wh m_{\al}(z)-m_{\al c}(z) \vert \prec \min \left\{\phi_n,\frac{\phi_n^2}{\sqrt{\kappa+\eta}}\right\} + (n \eta)^{-1}, \quad \al=1,2,3,4. \label{aver_in1} %+ q^2 
	%\end{equation}
	%where $m$ is defined in \eqref{defn_m}. 
	Moreover, outside of the spectrum, we have a stronger estimate for any $z\in \wt S_{out}(\epsilon,\wt \e)$:
	\begin{equation}\label{aver_out0}
		| m(z)-m_{c}(z)|\prec \frac{1}{n(\kappa +\eta)} + \frac{1}{(n\eta)^2\sqrt{\kappa +\eta}}.
	\end{equation}
	 The estimates \eqref{aver_in} and \eqref{aver_out0} also hold for $m_{\al}(z) -m_{\al c}(z)$, $\al=1,2,3,4$.
	%\begin{equation}\label{aver_out1}
	%%|m(z)-m_c(z)| +
	% | \wh m_\al(z)-m_{\al c}(z)|\prec \min \left\{\phi_n,\frac{\phi_n^2}{\sqrt{\kappa+\eta}}\right\} + \frac{1}{n(\kappa +\eta)} + \frac{1}{(n\eta)^2\sqrt{\kappa +\eta}}, \quad \al=1,2, 3,4,
	%\end{equation}
	%uniformly in $z\in S_{out}(\e):=S(\epsilon)\cap \{z=E+\ii\eta: E\notin [\lambda_-,\lambda_+], n\eta\sqrt{\kappa + \eta} \ge n^\epsilon\}$. %\cor The estimates for $\wh m(z)$ can be derived from \eqref{m3m} and \eqref{m4m}. \nc %se two estimates also hold for $(m(z) -m_c(z))$.%where $\kappa$ is defined in \eqref{KAPPA}. 
\end{itemize}
All the above estimates are uniform in the spectral parameter $z$. %and any set of deterministic unit vectors of cardinality $n^{\OO(1)}$. 
\end{theorem}

%The estimates in (\ref{aver_in}) and (\ref{DIAGONAL}) are usually referred to as the {\it{averaged local law}} and {\it{entrywise local law}}, respectively. In fact under different assumptions, they have been proved previously in different forms \cite{BPZ1,KY2}. For completeness, we will give a concise proof in Appendix \ref{appendix1} that fits into our setting. 

%On the other hand, the norm bound on $H$ in (\ref{boundH}) follows from a standard application of the moment method, for example, see \cite[Lemma 4.3]{EKYY1} or \cite{Handbook_DS}. We do not repeat its proof here.

%With Theorem \ref{thm_local} as a key input, we can prove a stronger estimate on $m(z)$ that is independent of $q$. This averaged local law implies the rigidity of eigenvalues for $\mathcal Q_1$. 
%For $\wt\e>0$, we define the following domains
%$$
%\wt S(\e,\wt \e):= \left\{z=E+ \ii \eta: \frac12 \lambda_- \leq E \leq 1-\wt\e, n^{-1+\e} \leq \eta \leq 1 \right\}, \quad \wt S_{out}(\e,\wt\e):= \wt S(\e,\wt \e)\cap  S_{out}(\e).
%$$
%Note that these two domains are away from $z=1$. 
\iffalse
The averaged local law leads to the 
following rigidity of eigenvalues.
\begin{theorem}[Theorem 2.5 of \cite{PartIII}] \label{thm_largerigidity}
Suppose the assumptions of Lemma \ref{lem null0}  hold. For any fixed $\delta>0$, the following rigidity estimate holds for all $1 \le i \le (1-\delta)q$:
\be\label{rigidity}
|\lambda_i - \gamma_i | \prec i^{-1/3} n^{-2/3}.
\ee
\end{theorem}
\fi

The averaged local law implies the rigidity of eigenvalues in \eqref{rigidity}. The anisotropic local law (\ref{aniso_law}) and the rigidity estimate \eqref{rigidity} together give the following delocalization of eigenvectors. %Its proof is standard.

\begin{lemma}[Lemma 3.9 of \cite{PartIII}] \label{lem delocalX}
Suppose \eqref{aniso_law} and \eqref{rigidity} hold. Then, for any small constant $\delta>0$ and deterministic unit vectors $\mathbf u_\al \in \mathbb C^{\mathcal I_\al}$, $\al=1,2,3,4$, the following estimates hold: %and constant $0<c_1<c_0$, we have
\begin{align}
&	\max_{1 \le k \le (1-\delta)q } \left\{ \left|\langle \mathbf u_1,S_{xx}^{-1/2}\xi_k\rangle \right|^2+\left|\langle \mathbf u_2,S_{yy}^{-1/2}\zeta_k\rangle \right|^2\right\} \prec n^{-1},\label{delocal1}\\
%\end{equation}
%and
%\begin{equation}
&	\max_{1 \le k \le (1-\delta)q} \left\{ \left|\langle \mathbf u_3,X^{\top}S_{xx}^{-1/2}\xi_k\rangle \right|^2+\left|\langle \mathbf u_4,Y^{\top}S_{yy}^{-1/2}\zeta_k\rangle \right|^2\right\} \prec n^{-1}.\label{delocal2}
\end{align}
\end{lemma}

Away from the support $[\lambda_-,\lambda_+]$, the anisotropic local law can be strengthened as follows. 
%we have stronger control all the way down to the real axis. Denote the spectral parameter set 

\begin{theorem}[Anisotropic local law outside the bulk spectrum]\label{thm_localout}  
Suppose the assumptions of Lemma \ref{lem null0}  hold. Fix any constant $\epsilon>0$. For any 
\begin{equation}\label{eq_paraout}
	z\in D_{out}(\epsilon):=\left\{z= E+ \ii\eta: \lambda_+ + n^{-2/3+\e}\le E \le 2,  0\le \eta \le 1\right\},
\end{equation}
and  deterministic unit vectors $\bu, \bv \in \mathbb{C}^{\mathcal I}$, the following anisotropic local law holds:
\begin{equation}\label{aniso_outstrong}
	\left| G_{\mathbf u\mathbf v}(z)   -   \Pi_{\mathbf u\mathbf v} (z)   \right|  \prec \phi_n+\sqrt{\frac{\im m_{c}(z)}{n\eta}} \asymp \phi_n + n^{-1/2}(\kappa+\eta)^{-1/4} .
\end{equation}
\end{theorem}
\begin{proof}
The second step of \eqref{aniso_outstrong} follows from \eqref{Immc}. Using \eqref{aniso_law} and $\kappa\ge n^{-2/3+\e}$, we can get that \eqref{aniso_outstrong} holds for $z\in S(\e)\cap D_{out}(\epsilon)$ with $\eta \ge \eta_0:=n^{-1/2}\kappa^{1/4}$. Hence, it remains to prove that for $z\in D_{out}(\epsilon)$ with $0\le \eta \le \eta_0$, we have 
\begin{equation}\label{aniso_outstrong2}
	\left|  G_{\mathbf v\mathbf v}(X,z) - \Pi_{\mathbf v\mathbf v} (z) \right|  \prec \phi_n+n^{-1/2}\kappa^{-1/4},
\end{equation}
for any deterministic unit vector $\bv \in \mathbb{C}^{\cal I}$. Note that \eqref{aniso_outstrong2} implies \eqref{aniso_outstrong} by the polarization identity
\begin{align*}
	\langle \mathbf u, \cal M \mathbf v\rangle &= \frac14  \langle (\mathbf u+\mathbf v), \cal M  (\mathbf u+\mathbf v)\rangle-  \frac14  \langle (\mathbf u-\mathbf v), \cal M  (\mathbf u-\mathbf v)\rangle \\
	&+\frac{\ii}{4} \langle (\ii\mathbf u+\mathbf v), \cal M  (\ii\mathbf u+\mathbf v)\rangle-  \frac{\ii}4  \langle (\ii\mathbf u- \mathbf v), \cal M  (\ii\mathbf u-\mathbf v)\rangle 
\end{align*}
for any $\cal I\times \cal I$ matrix $\cal M$. Now, fix any $z = E+\ii \eta \in D_{out}(\epsilon)$ with $\eta \le \eta_0$. We denote $z_0:= E+ \ii \eta_0$. Since \eqref{aniso_outstrong2} holds at $z_0$, it suffices to prove the following estimates:
\begin{align}
\Pi_{\mathbf v\mathbf v} (z)-\Pi_{\mathbf v\mathbf v}(z_0)  & \prec n^{-1/2}\kappa^{-1/4},\label{prof_m}\\
%\ee
%and 
%\be
G_{\mathbf v\mathbf v} (z)-G_{\mathbf v\mathbf v}(z_0) & \prec n^{-1/2}\kappa^{-1/4}.\label{prof_G}
\end{align}
The estimate \eqref{prof_m} follows immediately from \eqref{eq_mdiff}. It remains to show \eqref{prof_G}.

%it is enough to show that 
%\be\label{prof_m2}
%|m_{1c}(z)-m_{1c}(z_0)|+|m_{2c}(z)-m_{2c}(z_0)| \prec n^{-1/2}\kappa^{-1/4}.
%\ee
%Using \eqref{eq_mdiff}, 
%%\eqref{eq_mcomplexd}, we get that 
%%$$ |m_{1c}'(z)| \lesssim |z-\lambda_+|^{-1} \sim (\kappa+\eta)^{-1/2}.$$
%%where we used \eqref{sqroot4} in the second step. 
%we obtain that
%$$ |m_{1c}(z) - m_{1c}(z_0)| \lesssim \frac{z-z_0}{|z_0-\lambda_+|^{1/2}} \le n^{-1/2}\kappa^{-1/4}.$$
%We can deal with the $m_{2c}$ term in the same way. This finishes the proof of \eqref{prof_m}.
We write $\mathbf v=\begin{pmatrix} \bv_1^{\top} , \bv_2^{\top}, \bv_3^{\top},\bv_4^{\top} \end{pmatrix}^{\top}$, where $\bv_\al\in \C^{\cal I_\al}$, $\al=1,2,3,4$. We claim that
\be\label{claim GL}
\begin{pmatrix} \bv_1^* , \bv_2^*\end{pmatrix} \left[\cal G_L(z) -\cal G_L(z_0)\right] \begin{pmatrix} \bv_1 \\ \bv_2\end{pmatrix} \prec n^{-1/2}\kappa^{-1/4}.
\ee
For simplicity of notations, in the following proof, we will always identify $\bv_\al$, $\al=1,2,3,4,$ with their natural embeddings in $\C^{\cal I}$ (recall Definition \ref{defn_gen_entry}). Using \eqref{spectral1} and \eqref{GL1}, and recalling that with high probability $E-\lambda_k \gtrsim 1$ for $k\ge (1-\delta)q $ by the rigidity estimate \eqref{rigidity}, we obtain that
\be
\begin{split}\label{zz0}
	\left|\langle \bv_1, \left(G(z) - G(z_0)\right) \bv_1\rangle\right| \prec &\sum_{k \le (1-\delta)q}  \frac{\eta_0  |\langle \bv_1,S_{xx}^{-1/2}{\xi}_k\rangle|^2}{\left[(E-\lambda_k)^2 + \eta^2 \right]^{1/2}\left[(E-\lambda_k)^2 + \eta_0^2 \right]^{1/2}} \\
	&+ {\eta_0} \sum_{k>(1-\delta)q}{|\langle \bv_1,S_{xx}^{-1/2}{\xi}_k\rangle|^2}.
\end{split}
\ee
%{\color{blue} dimension does not match. Do you mean the natural embedding with $\vb'_1=(\vb_1, \mathbf{0}).$}
%Here and throughout the rest of this paper, we will always identify vectors $\bv_1$ and $\bv_2$ with their embeddings $\begin{pmatrix} \bv_1 \\ 0\end{pmatrix}$ and $\begin{pmatrix} 0 \\ \bv_2\end{pmatrix}$, respectively.  
By \eqref{rigidity}, we have that for any $k\ge 1$, $E-\lambda_k \gtrsim \kappa \gg \eta_0$ with high probability. Then, using \eqref{delocal1} and \eqref{op rough1}, we can bound \eqref{zz0} by
\begin{equation*}
	\begin{split}
		  \left|\langle \bv_1, \left(G(z) - G(z_0)\right) \bv_1\rangle\right| & \prec  \eta_0 + \frac1q \sum_{k = 1}^{q} \frac{\eta_0 }{(E-\lambda_k)^2 + \eta_0^2}  =\eta_0 + \im m(z_0) \\
		&\prec \eta_0 + \frac{1}{n\kappa} + \frac{1}{(n\eta_0)^2 \sqrt{\kappa}}+ \im m_c(z_0) \\
		&\lesssim   \frac{1}{n\kappa} + \frac{1}{(n\eta_0)^2 \sqrt{\kappa}}+ \frac{\eta_0}{\sqrt{\kappa+\eta_0}}  \lesssim n^{-1/2}\kappa^{-1/4},
	\end{split}
\end{equation*}
where we used the spectral decomposition for $m(z)$ in the second step, \eqref{aver_out0} in the third step, and \eqref{Immc} in the fourth step. Similarly, we have
\begin{equation*}
	\begin{split}
		 \left|\langle \bv_1, \left(G(z) - G(z_0)\right) \bv_2\rangle\right|  
		&\prec  \left|1- \sqrt{\frac{z}{z_0}}\right| \left|\langle \bv_1, G(z_0)\bv_2\rangle\right| + \sum_{k = 1}^{q} \frac{\eta_0|\langle \bv_1 ,S_{xx}^{-1/2}{\xi}_k\rangle | | \langle \bv_2, S_{yy}^{-1/2}{\zeta}_k \rangle|}{|\lambda_k-z||\lambda_k-z_0|}\\
		& \prec \eta_0 + \im m(z_0) \prec  n^{-1/2}\kappa^{-1/4}.
	\end{split}
\end{equation*}
Similar arguments also apply to $\langle \bv_2, \left(G(z) - G(z_0)\right) \bv_1\rangle$ and $|\langle \bv_2, \left(G(z) - G(z_0)\right) \bv_2\rangle$. Hence we conclude \eqref{claim GL}. Finally, using \eqref{claim GL}, \eqref{GR1}, \eqref{GLR1} and Lemma \ref{lem delocalX}, we can get \eqref{prof_G}. We omit the details. 
\end{proof}

The second moment of $ \langle \mathbf u, (G(z)-\Pi(z)) \mathbf v\rangle $ in fact satisfies a stronger bound. It will be used in the proof of Theorem \ref{main_thm1}.

\begin{lemma}\label{thm_largebound}
%Let $X$ be a data matrix satisfying the assumptions in Theorem \ref{thm_local}. %except \eqref{assm_3moment}. Assume that \eqref{aniso_law} holds for $G(X,z)$. 
Suppose the assumptions of Lemma \ref{lem null0} hold. Fix any constant $\e>0$. For any deterministic unit vectors $\mathbf u , \mathbf v  \in \mathbb C^{\mathcal I}$, we have that uniformly in $z\in S(\e)$ (recall \eqref{SSET1}),
\begin{equation}\label{weak_off}
	\mathbb{E} \left| G_{\mathbf u\mathbf v}(z)   -   \Pi_{\mathbf u\mathbf v} (z)   \right| ^2 \prec \Psi^2(z), 
\end{equation}
 and uniformly in $z\in D_{out}(\e)$ (recall \eqref{eq_paraout}),
\begin{equation}\label{weak_offout}
	\mathbb{E} \left| G_{\mathbf u\mathbf v}(z)   -   \Pi_{\mathbf u\mathbf v} (z)   \right| ^2 \prec \frac{1}{n\sqrt{\kappa+\eta}} .
\end{equation}
% we have that
%This estimates also holds for $G(z)$.
\end{lemma}
\begin{proof}
The estimate \eqref{weak_off} has been proved in Lemma 3.10 of \cite{PartIII}. The estimate \eqref{weak_offout} can be proved using almost the same argument, where the only difference is that we replace the local law \eqref{aniso_law} with the stronger one \eqref{aniso_outstrong} in the proof. We omit the details.
\end{proof}

%With all the above preparations, we are ready to give the proofs of the main results in Section \ref{sec mainresults}.  
%Moreover, we assume that $a_i$ and $b_i$ take values in the extended real line, that is we can take (recall \eqref{assm evalue})
%\be\label{ab}
%{0\le a_r \le a_1 \le \infty,}\quad {0\le b_r \le b_1 \le\infty}. 
%\ee
%This is legitimate since only the combinations $\frac{a_i}{1+a_i^2}$, $\frac{b_i}{1+b_i^2}$,  $\frac{a_i}{\sqrt{1+a_i^2}}$ or $\frac{b_i}{\sqrt{1+b_i^2}}$ appear in our study of CCA matrices.

\section{Proof of Theorem \ref{main_thm} }\label{sec mainthm}

In this section, we prove Theorem \ref{main_thm} using the local laws, Theorems \ref{thm_local} and \ref{thm_localout}, and the eigenvalue rigidity \eqref{rigidity}. During the proof, in order to avoid some non-generic events, we assume that
\be\label{abscont}
\text{the entries $x_{ij}$, $y_{ij}$ and $z_{ij}$ have continuous densities}.
\ee
It can be achieved by adding a small perturbation to $X$, $Y$ and $Z$. For example, we can add to each matrix a small Gaussian matrix:
$$X\to X + \delta e^{-n}X_G, \quad Y\to Y + \delta e^{-n}Y_G, \quad Z\to Z + \delta e^{-n}Z_G.$$
These Gaussian components are negligible for our results and can be easily removed by taking $\delta\to 0$. Under \eqref{abscont}, the matrices $\cal X\cal X^{\top}$, $\cal Y\cal Y^{\top}$, $XX^{\top}$ and $YY^{\top}$ are all non-singular almost surely. Moreover, almost surely, $\lambda=1$ is not in the spectrum of $\cal C_{XY}$ or $\cal C_{\cal X\cal Y}$. 
By \eqref{detereq temp2}, $0<\lambda<1$ is an eigenvalue of $\cal C_{\cal X\cal Y}$ if and only if 
\begin{align} \label{masterx}
\det \left[ 1+\begin{pmatrix} 0 & \cal D\\ \cal D  & 0\end{pmatrix}\begin{pmatrix} \bU^{\top} & 0 \\ 0 & \bE^{\top}\end{pmatrix} G(\lambda) \begin{pmatrix} \bU & 0 \\ 0 & \bE\end{pmatrix}\right]=0.
\end{align} 

Using a standard large deviation estimate (e.g., Lemma 3.8 of \cite{EKYY1}), we can derive the following approximate isometry condition for $Z$:
\begin{equation}
\| ZZ^{\top} - I_{r}\| \prec \psi_n . \label{eq_iso}
\end{equation}
Now, for any $\lambda\in D_{out}(\e)$, using Theorem \ref{thm_localout} and \eqref{eq_iso}, we can write \eqref{masterx} as
\be\label{masterx2}
\begin{split}
&0 =\det \left[ 1+\begin{pmatrix} 0 & \cal D\\ \cal D  & 0\end{pmatrix}\left(\Pi_r(\lambda) + \cal E_{4r}\right)\right] \\
& = \det \left[ \begin{pmatrix} I_{2r} & \cal D \begin{pmatrix}  m_{3c}(\lambda)I_r & h(\lambda)\cal M_r \\ h(\lambda)\cal M_r^{\top}&  m_{4c}(\lambda)I_r\end{pmatrix}   \\ \cal D\begin{pmatrix} c_1^{-1}m_{1c}(\lambda)I_r & 0 \\ 0 & c_2^{-1}m_{2c}(\lambda)I_r\end{pmatrix}  & I_{2r} \end{pmatrix} +\begin{pmatrix} 0 & \cal D\\ \cal D  & 0\end{pmatrix} \cal E_{4r}  \right].
\end{split}
\ee
Here, $\cal E_{4r}$ is a $4r\times 4r$ random matrix satisfying
\be\label{e4r}  \|\cal E_{4r}\| \prec \psi_n + \phi_n+ n^{-1/2}\kappa_{\lambda}^{-1/4} ,\quad \text{with}\quad \kappa_\lambda: = \min\left\{ \vert \lambda -\lambda_-\vert,\vert \lambda -\lambda_+\vert\right\} ,
\ee
$\cal M_r$ is an $r\times r$ orthogonal matrix with entries
$$(\cal M_r)_{ij}:= (\bv_i^a)^{\top} \bv_j^b, \quad 1\le i, j \le r,$$
and $\Pi_r(\lambda)$ is defined as
$$\Pi_r(\lambda):= \begin{bmatrix}\begin{pmatrix} c_1^{-1}m_{1c}(\lambda)I_r & 0 \\ 0 & c_2^{-1}m_{2c}(\lambda)I_r\end{pmatrix} & 0\\0 & \begin{pmatrix}  m_{3c}(\lambda)I_r & h(\lambda)\cal M_r \\ h(\lambda)\cal M_r^{\top} &  m_{4c}(\lambda)I_r\end{pmatrix} \end{bmatrix}.$$
Applying the Schur complement formula and using \eqref{selfm12}, we obtain that \eqref{masterx2} is equivalent to
\be\label{masterx3}
\begin{split}
&\  \det \left[\begin{pmatrix} I_{r} + \Sigma_a^2 & h(\lambda)m_{3c}^{-1}(\lambda)\Sigma_a^2 \cal M_r \\ h(\lambda)m_{4c}^{-1}(\lambda)\Sigma_b^2 \cal M_r^{\top} & I_r + \Sigma_b^2 \end{pmatrix} + \cal D \cal E_{2r} \right]=0\\
\Leftrightarrow &\ \det \left(\frac{m_{3c}(\lambda)m_{4c}(\lambda)}{h^2(\lambda)}I_r -  \frac{\Sigma_a }{(I_r+\Sigma_a^2)^{1/2}}\cal M_r \frac{\Sigma_b^2}{I_r+\Sigma_b^2}\cal M_r^{\top} \frac{\Sigma_a}{(I_r +\Sigma_a^2)^{1/2}}+ \cal E_r\right) =0,
\end{split}
\ee
where $\cal E_{2r}$ and  $\cal E_r$ are $2r\times 2r$ and $r\times r$ random matrices, both of which satisfy the same bound as in \eqref{e4r}. 
%$$  \|\cal E_r\| +  \|\cal E_{2r}\| \le C\psi_n + \OO_\prec \left(\phi_n+\sqrt{\frac{\im m_{c}(z)}{n\eta}}\right).$$
Note that the matrix 
$$\frac{\Sigma_a }{(I_r+\Sigma_a^2)^{1/2}}\cal M_r \frac{\Sigma_b^2}{I_r+\Sigma_b^2}\cal M_r^{\top} \frac{\Sigma_a}{(I_r+\Sigma_a^2)^{1/2}} $$ 
is the PCC matrix $(1+AA^{\top})^{-1/2} AB^{\top} (1+BB^\top)^{-1} BA^{\top}(1+AA^{\top})^{-1/2}$ in the basis of $\bu_i^a$, $1\le i\le r$. Thus, its eigenvalues are exactly the squares of the population CCCs, $t_1, t_2, \cdots, t_r$ (recall \eqref{t1r}). After a change of basis, \eqref{masterx3} reduces to 
\be\label{masterx4}
\det \left(\frac{m_{3c}(\lambda)m_{4c}(\lambda)}{h^2(\lambda)}I_r - \diag \left(t_1, \cdots, t_r \right) + \cal E'_r(\lambda)\right) =0 ,
\ee
where $\cal E_r'$ also satisfies the same bound as in \eqref{e4r}.

Next, we show that if $\cal E'_r=0$, then solving equation \eqref{masterx4} gives the classical locations $\theta_i$ defined in \eqref{gammai}. Using \eqref{m3c}, \eqref{m4c} and \eqref{hz}, we can calculate that 
\begin{align*}
f_c(z):&=\frac{m_{3c}(z)m_{4c}(z)}{h^2(z)} = z [1+(1-z)m_{1c}(z)][1+(1-z)m_{2c}(z)] \\
&= \frac{z- (c_1+c_2-2c_1c_2) + \sqrt{(z-\lambda_-)(z-\lambda_+)}}{2(1-c_1)(1-c_2)}.
\end{align*}
We can find the inverse function of $f_c(z)$ for $ z\notin [\lambda_-,\lambda_+]$ as
$$g_c(\xi) := \xi \big( 1- c_1 + c_1\xi^{-1}\big)\big( 1- c_2 + c_2\xi^{-1}\big).$$
Note that $f_c(\lambda)$ is monotonically increasing in $\lambda$ for $\lambda>\lambda_+$, so the function $f_c(\lambda) - t_i=0$ has a solution in $(\lambda_+,\infty)$ if and only if (recall \eqref{tc})
\be\label{t>tc} 
f_c(\lambda_+)< t_i \quad \Leftrightarrow \quad t_c < t_i .
\ee
If \eqref{t>tc} holds, the classical location of the outlier corresponding to $t_i$ is $\theta_i = g_c(t_i)$, which gives \eqref{gammai}.

\vspace{5pt}

With direct calculations, one can verify the following simple estimates on $f_c$ and $g_c$. 

\begin{lemma} \label{lem_complexderivative}
Fix a large constant $C>0$. Let $z, z_1, z_2\in \mathbb D:=\{z\in \C: \lambda_+ < \re z < C,0<\im z\le C\}$ and $\xi,\xi_1, \xi_2 \in f_c( \mathbb D)$. The following estimates hold:
\begin{align}
	|f_{c}(z) - f_{c}(\lambda_+)| \sim |z-\lambda_+|^{1/2}, \quad & |f_{c}'(z) | \sim |z-\lambda_+|^{-1/2}, \label{eq_mcomplex0}\\
	|g_{c}(\xi) - \lambda_+| \sim |\xi-t_c|^2, \quad  & |g_{c}'(\xi) | \sim |\xi-t_c| ,\label{eq_gcomplex0}\\
%and
%\begin{equation}
%	\begin{split}
	  |f_{c}(z_1) - f_{c}(z_2)| \sim \frac{|z_1-z_2|}{\max_{i=1,2}|z_i-\lambda_+|^{1/2}}, \quad & |g_{c}(\xi_1) - g_{c}(\xi_2)| \sim |\xi_1-\xi_2|\cdot\max_{i=1,2}|\xi_i-t_c|.\label{eq_mdiff0}
%	\end{split}
%\end{equation}
\end{align}
The estimate \eqref{eq_mcomplex0} also holds for $z$ with $\lambda_{-}+c \le \re z \le \lambda_+$ and $ 0<\im z\le c^{-1}$ for any small constant $c>0$.
\end{lemma}

For the proof of Theorem \ref{main_thm}, we record the following eigenvalue interlacing result:
\begin{equation}\label{interlacing_eq0}
\wt\lambda_i \in [\lambda_{i + 2r}, \lambda_{i-2r}],
\end{equation}
where we adopt the convention that $\lambda_{i}=1$ if $i<1$ and $\lambda_i = 0$ if $i>q$. For the reader's convenience, we briefly describe why \eqref{interlacing_eq0} holds. We first consider a 1-dimensional perturbation:
$$X_1: =  X + \bu_1 \bv_1^{\top}, \quad \bu_1\in \R^p, \quad \bv_1\in \R^n.$$
Then, it is easy to see that $\cal P_X:=X^{\top}(XX^{\top})^{-1}X$ is a projection onto the subspace $\cal W$ spanned by the rows of $X$. Similarly, $\cal P_{X_1}:=X_1^{\top}(X_1X_1^{\top})^{-1}X_1$ is a projection onto the subspace $\cal W_1$ spanned by the rows of $X_1$. Moreover, $\cal W$ and $\cal W_1$ differ at most by a 1-dimensional subspace. Hence, by Cauchy interlacing, we have
$$\lambda_i \left(\cal P_{X_1} \cal P_Y \cal P_{X_1}\right) \in [\lambda_{i + 1}(\cal P_{X} \cal P_Y \cal P_{X}), \lambda_{i-1}(\cal P_{X} \cal P_Y \cal P_{X})], \quad \text{where}\quad \cal P_Y:=Y^{\top} \frac{1}{YY^{\top}}Y.$$
Notice that $\cal P_{X} \cal P_Y \cal P_{X}$  (resp. $\cal P_{X_1} \cal P_Y \cal P_{X_1}$) has the same nonzero eigenvalues as $\cal C_{XY}$ (resp. $\cal C_{X_1 Y}$): if $\bu$ is an eigenvector of $\cal C_{XY}$ with eigenvalue $\lambda$, then $ X^{\top}(XX^{\top})^{-1/2}\bu$ is an eigenvector of $\cal P_{X} \cal P_Y \cal P_{X}$ with the same eigenvalue. Thus, we get 
$$\lambda_i \left(\cal C_{X_1Y}  \right) \in [\lambda_{i + 1}(\cal C_{XY}  ), \lambda_{i-1}(\cal C_{XY}  )].$$
Repeating this estimate $r$ times for the rank-$r$ perturbation $\cal X$, we get
$$\lambda_i \left(\cal C^a_{\cal X Y}  \right) \in [\lambda_{i + r}(\cal C_{XY}  ), \lambda_{i-r}(\cal C_{XY}  )],$$
where $\cal C^a_{\cal X Y}$ is defined by replacing $X$ with $\cal X$ in $\cal C_{XY} $. Obviously, the same argument works for the rank-$r$ perturbation of $Y$, which leads to \eqref{interlacing_eq0}.

%\vspace{5pt}

With \eqref{masterx4} and \eqref{interlacing_eq0}, the rest of the proof for Theorem \ref{main_thm}  is similar to those in \cite[Section 4]{principal} and \cite[Section 6]{KY}, but these references have only considered cases with small support $\phi_n \prec n^{-1/2}$. We need to adapt their proofs to our setting with larger $\phi_n $ and $\psi_n$. %Hence we will only describes the necessary changes one needs to make to the arguments in \cite{principal,DY2,KY}, without giving all the details.

\begin{proof}[Proof of Theorem \ref{main_thm}]
%First the estimate \eqref{boundedge2} follows directly from the interlacing \eqref{interlacing_eq0} combined with the rigidity result in Theorem \ref{thm_largerigidity}. It remains to prove \eqref{boundout} and \eqref{boundedge}.
For simplicity of presentation, in this proof we abbreviate $\phi_n + \psi_n$ as $\phi_n$ because these two factors always appear together. By Theorems \ref{thm_local} and \ref{thm_localout} and equations \eqref{rigidity} and \eqref{eq_iso}, for any fixed $\epsilon>0$, we can choose a high-probability event $\Xi$ on which the following estimates hold: %for some constant $C_0>0$, %satisfying the following conditions. 
\begin{equation}\label{aniso_lawev}
  \left\|{ \begin{pmatrix} \bU^{\top} & 0 \\ 0 & \bE^{\top}\end{pmatrix}  G(z) \begin{pmatrix} \bU & 0 \\ 0 & \bE\end{pmatrix}}-\Pi_r(z) \right\| \le  n^{\e/2}\left(\phi_n + \Psi(z)\right),\ \ \text{for }z\in S(\e);
\end{equation}
\begin{equation} \label{eq_bound1ev}
 \left\|{ \begin{pmatrix} \bU^{\top} & 0 \\ 0 & \bE^{\top}\end{pmatrix}  G(z) \begin{pmatrix} \bU & 0 \\ 0 & \bE\end{pmatrix}}-\Pi_r(z) \right\| \le n^{\e/2}\left(\phi_n + n^{-1/2}\kappa^{-1/4} \right),\ \ \text{ for } z\in D_{out}(\epsilon);%x\in [\lambda_+ + n^{-2/3+\e}, \varsigma_2 \lambda_+]; %\ \kappa \equiv \kappa(x):=x-\lambda_+. 
\end{equation}
for a fixed large integer $\varpi \in \N$,
\begin{equation} \label{eq_bound2ev}
 \left|\lambda_i - \lambda_+\right| \leq n^{-2/3+\epsilon}, \quad  \text{ for }1\le i \le \varpi + 2r.
\end{equation}
We remark that the randomness of $X$ and $Y$ only comes into play to ensure that $\Xi$ holds with high probability. The rest of the proof will be entirely deterministic once restricted to $\Xi$. 
In the following proof, we assume that $\e$ is a sufficiently small constant.

We now define the index sets
\begin{equation} \label{eq_otau}
	\mathcal{O}_{\e}:=\left\{i: t_i  - t_c  \geq  n^\e\phi_n+  n^{-{1}/{3}+\e} \right\}.
\end{equation}
%for some constant $C_\e>0$. 
Since the constant $\e$ is arbitrary, in order to prove \eqref{boundout} and \eqref{boundedge}, it suffices to show that there exists a constant $C>0$ such that
\begin{equation}\label{eq_spikepf}
	\mathbf 1(\Xi)\left|\wt\lambda_{i}-\theta_i  \right| \le C n^{2\e}\left( \phi_n  \Delta_i +  n^{-1/2 }\Delta_i^{1/2}\right) 
\end{equation}
for all $i\in \cal O_{4\e}$, and 
\begin{equation}\label{eq_nonspikepf}
	- n^{-2/3+\e}\le \mathbf 1(\Xi)\left(\wt \lambda_{i}-\lambda_{+}\right) \le  C n^{8\e}\phi_n^2+ Cn^{-2/3+ 12\e} 
\end{equation}
for all $i\in \{1,\cdots, \varpi\}\setminus \cal O_{4\e}$. For the rest of the proof, we assume that $\Xi$ holds. %Note that \eqref{boundout} and the upper bound in \eqref{boundedge} follow from \eqref{eq_spikepf} and \eqref{eq_nonspikepf} directly. On the other hand, the lower bound in \eqref{boundedge} follows from the interlacing \eqref{interlacing_eq0} and the rigidity estimate \eqref{rigidity}.

\vspace{10pt}

%
%For $1 \leq i \leq r_++s^+,$
\noindent{\bf Step 1:} Our first step is to prove that on $\Xi$, there are no eigenvalues outside the neighborhoods of $\theta_i$'s. 
%Denote $\widetilde{D}:=\{\tsig_i^a\} \cup \{\tsig_i^b\}.$ 
For $1\le i \le r_+,$ we define the permissible intervals %let the permissible regions of the spiked eigenvalues be  
\begin{equation}\label{rIi}
	\rI_i\equiv \rI_i(\mathbf t):=\left[\theta_i - n^\e\left( \phi_n \Delta_i +  n^{-1/2}\Delta_i^{1/2}\right) , \theta_i+n^\e\left( \phi_n \Delta_i +  n^{-1/2}\Delta_i^{1/2}\right)\right],
\end{equation} 
where $\mathbf t$ denotes the vector $\mathbf t:=(t_1, t_2 ,\cdots, t_r)$. We then define 
\begin{equation}\label{I0}
	\rI\equiv \rI( \mathbf t):=\rI_0  \cup  \bigcup_{i \in \mathcal{O}_{\epsilon}}\rI_i( \mathbf t)    ,\quad   \rI_0 :=\left[0, \lambda_+ +  n^{2\e}\phi_n^2 + n^{-2/3+3\epsilon}\right].
\end{equation}
%and
%\begin{equation*}
%\rI(\widetilde{D}):=\rI_0 \cup \bigcup_{i \in \mathcal{O}_{\epsilon}} \rI_i(\widetilde{D}). 
%\end{equation*}
We claim the following result.
\begin{lemma}\label{lem_gapI}
	%There exists constants $C_1, C_\e>0$ such that 
	The complement of $\rI(\bf t)$ contains no eigenvalues of $\cal C_{\cal X\cal Y}$.
\end{lemma}
\begin{proof}%[Proof of Lemma \ref{lem_gapI}] 
	The main idea is similar to the ones for \cite[Proposition 6.5]{KY} and \cite[Lemma S.4.2]{DY2}. It suffices to show that for any $1\le i \le r$, if $x\notin \rI(\mathbf t)$, then
	\be\label{skip1} | f_c(x) - t_i | \ge c \left(  n^\e\phi_n +  n^{-1/2+\epsilon}\kappa_x^{-1/4}\right) \ee
	for some constant $c>0$. Thus, \eqref{masterx4} cannot hold on $\Xi$ by \eqref{eq_bound1ev}.

	For $x\notin \rI_0$, with \eqref{eq_mdiff0}, we get that
	$$ f_c(x)- t_c=f_c(x)- f_c(\lambda_+) \ge  c \kappa_x^{1/2} \ge c'\left(  n^\e \phi_n +  n^{-1/2+\epsilon}\kappa_x^{-1/4}\right),$$
	for some constants $c,c'>0$. This concludes \eqref{skip1} for $ i \ge r_+$ by using $ t_i \le t_c + n^{-1/3} +  \phi_n $. 
	
	For the case $1\le i \le r_+$, we take any $x\notin \rI_0 \cup \rI_i(\mathbf t)$.  We first assume that there exists a constant $\wt c>0$ such that $\theta_i \notin [x-\wt c\kappa_x, x+ \wt c\kappa_x]$. Since $f_{c}$ is monotonically increasing on $(\lambda_+,+\infty)$, we have that
	$$|f_{c}(x)-t_i| =|f_{c}(x)-f_{c}(\theta_i)| \ge |f_{c}(x) - f_{c}(x\pm \wt c\kappa_x)| \ge  c \kappa_x^{1/2} \ge c'\left( n^\e \phi_n +  n^{-1/2+\epsilon}\kappa_x^{-1/4}\right),$$
	for some constants $c,c'>0$, where we used \eqref{eq_mdiff0} in the third step. On the other hand, suppose $\theta_i \in [x-\wt c\kappa_x, x+ \wt c\kappa_x]$, in which case we have that $\theta_i-\lambda_+ \sim \kappa_x$. By \eqref{eq_gcomplex0}, we have
	$\kappa_x \sim \theta_i - \lambda_+ \sim \Delta_i^2 $. Then, using \eqref{eq_mdiff0} and the definition of $\rI_i(\mathbf t)$, we get that for $x\notin \rI_i(\mathbf t) $,
	\be\nonumber
	\begin{split}
		|f_{c}(x)-t_i| = |f_c(x)- f_{c}(\theta_i)| &\ge c \Delta_i^{-1}\left( n^\e\phi_n \Delta_i +  n^{-1/2+\epsilon}\Delta_i^{1/2}\right) \ge c' \left(  n^\e\phi_n +  n^{-1/2+\epsilon}\kappa_x^{-1/4}\right) ,
	\end{split}
	\ee
	for some constants $c,c'>0$. This concludes \eqref{skip1} and hence Lemma \ref{lem_gapI}.
\end{proof}

%Together with  (\ref{eq_mderivative}), for some constant $C>0,$  we have that  
%\begin{equation}\label{eq_compare1}
%\left|\frac{1}{1+m_c(x) \sigma_1}-\frac{1+d_1}{d_1} \right| \geq C n^{-1/2+\epsilon} (-(\tsig_1)^{-1}-m_c(\lambda_+))^{-1/2}. 
%\end{equation}
%By the assumption of (\ref{eq_multione}), we find that 
%\begin{equation} \label{eq_compare2}
%\frac{1+d_i}{d_i}-\frac{1}{1+m_c(x) \sigma_i} \sim 1, \  2 \leq i \leq r_++s^+. 
%\end{equation}
%As a consequence, by (\ref{eq_determinant}), (\ref{eq_compare1}) and (\ref{eq_compare2}), we conclude that for some constant $C>0,$ 
%\begin{equation*}
%\det(\mathcal{D}^{-1}+x \bu^* \Pi(x) \bu) \geq C n^{-1/2+\epsilon} (-(\tsig_1)^{-1}-m_c(\lambda_+))^{-1/2}.
%\end{equation*}
%Next, we conclude from (\ref{eq_derivativebound}) that under Assumption \ref{ass:spike}
%\begin{equation*}
%\kappa \sim (-(\tsig_1)^{-1}-m_c(\lambda_+))^2. 
%\end{equation*}
%Hence, on the event $\Xi,$ by 
%By (\ref{eq_bound1ev}) and invoking Lemma \ref{eq_lemmadeter}, we have 
%\begin{equation*}
%\det(\mathcal{D}^{-1}+x \mathbf{U}^* G(x) \mathbf{U})=\det(\mathcal{D}^{-1}+x \Pi(x)) \times \{1+O_{\prec}(n^{-\epsilon/2})\}+O_{\prec}(n^{-1}\kappa^{-1/2}). 
%\end{equation*}
%This implies that on the event $\Xi,$ we have 
%\begin{equation}\label{eq_contro1}
%\det(\mathcal{D}^{-1}+x\bu^* G(x) \bu) \geq C n^{-1/2+\epsilon_1} \kappa^{-1/4}.
%\end{equation}
%This shows that all the eigenvalues of $\ctQ_1$ belong to $\rI(\widetilde{D})$ when $\Xi$ happens. 

%\vspace{10pt}

\noindent{\bf Step 2:} Before giving the general proof, for heuristics, we consider an easy case where 
%will show that each $\rI_i(\mathbf t)$, $i \in  \mathcal{O}_\epsilon$, contains the right number of eigenvalues of $\cal C_{\cal X\cal Y}$, under a {\it special configuration} such that
% see \eqref{eq_multione} below. For simplicity, we relabel the indices in $\mathcal{O}_\epsilon^{(a)}\cup \mathcal{O}_\epsilon^{(b)}$ as $\wt\sigma_1, \cdots, \wt\sigma_{r_\e}$, and call them $\e$-{\it{spikes}}. Moreover, we assume that they correspond to classical locations of outliers as $x_1,\cdots, x_{r_\e}$ (some of them are determined by $\theta_1$, while others are given by $\theta_2$), such that  
%\be
%x_1 \ge x_2 \ge \cdots \ge x_{r_\e}.
%\ee
%The corresponding permissible intervals $\rI_i^{(a)}$ and $\rI_\mu^{(b)}$ are relabelled as $\rI_i$, $1\le i \le r_\e$. In this step, we consider a special configuration $\mathbf x\equiv \mathbf x(0):=(x_1, x_2 , \cdots , x_{r_\e})$ of 
the $t_i$'s are {\it independent of $n$} and satisfy that
\begin{equation}\label{eq_multione}
	t_1>t_2 > \cdots > t_{r_+}>\lambda_+ . %\quad r_\e:=|\cal O_\e|.
\end{equation}
We claim that each $\mathbf{\rI}_i(\mathbf t)$, $1\le  i \le r_+$, contains precisely one eigenvalue of $\cal C_{\cal X\cal Y}$. Fix any $1\le i \le r_+$, we choose a small $n$-independent positively oriented closed contour ${\bm\Gamma} \subset \mathbb{C}/[0, \lambda_+]$ that encloses $\theta_i$ but no other points of the set $\{\theta_i: 1\le i \le r_+\}.$ Define two functions
\begin{equation}\label{deff12}
\begin{split}	&f_1(z):=\det\left(f_c(z)I_r- \diag \left(t_1, \cdots, t_r \right) \right), \\ & f_2(z):=\det\left(f_c(z)I_r- \diag \left(t_1, \cdots, t_r \right) + \cal E_r'(z)\right),
	\end{split}
\end{equation}  
where $\cal E_r'$ is defined in \eqref{masterx4}. 
The functions $f_1,f_2$ are holomorphic on and inside $\bm \Gamma$ when $n$ is sufficiently large, because ${\bm\Gamma}$ does not enclose any pole of $G(z)$ by \eqref{eq_bound2ev}. Moreover, by the construction of ${\bm\Gamma},$ the function $f_1$ has precisely one zero inside ${\bm\Gamma}$ at $\theta_i.$ By \eqref{eq_bound1ev}, we have  
\begin{equation*}
	\min_{z \in {\bm\Gamma}}|f_1(z)| \gtrsim 1, \quad \max_{z \in {\bm\Gamma}}|f_1(z)-f_2(z)|=\oo(1).
\end{equation*}
The claim then follows from Rouch{\' e}'s theorem. %This completes our proof for configuration of $\{\tsig_i\}$ satisfying (\ref{eq_multione}).

\vspace{5pt}

\noindent{\bf Step 3:} In order to extend the argument in Step 2 to an arbitrary $n$-dependent configuration $\mathbf t$, we need to deal with the case where some of the intervals $\rI_i$ and $ \rI_j$, $i\ne j$, have non-empty overlaps. For any constant $\e>0$, we denote $r_\e:=|\cal O_\e|.$ In this step, we prove  the following claim for the first $r_{4\e}$ eigenvalues.
\begin{claim}\label{lem_cont} 
	On event $\Xi$, the estimate \eqref{eq_spikepf} holds for $i\in \cal O_{4\e}$. 
\end{claim}
\begin{proof}
	Let $\mathcal{B}$ denote the finest partition of $\{1,\cdots, r_{+}\}$ in the sense that $i$ and $j$ belong to the same block of $\cal B$ whenever $\rI_i  \cap \rI_j \neq  \emptyset$. We now fix any $1\le i \le r_{4\e}$ and denote by $B_i$ the block of $\mathcal{B}$ that contains $i$. %Note that elements of $B_i$ are sequences of consecutive integers. 
	Our first task is to estimate $\theta_{j-1}-\theta_j$ for $j,j-1\in B_i$. We claim that there exists a constant $C_1>0$ such that
	\be\label{deltajj-1} 
	\theta_{j-1}-\theta_j \le C_1 \left( n^\e\phi_n \Delta_j +  n^{-1/2+\epsilon}\Delta_j^{1/2}\right), \quad \text{if }\  j \in B_i \text{ and } j-1\in B_i.
	\ee
	First, we assume that $j\in \cal O_{3\e}$. We pick any $x\in \rI_j\cap \rI_{j-1}$ such that $\theta_j \le x \le \theta_{j-1}$. Then, using \eqref{eq_mcomplex0} and \eqref{eq_mdiff0}, we obtain that 
	\begin{align*}
		|f_c(x) - t_j |=|f_c(x) - f_c(\theta_j) | \le C \left(n^\e\phi_n +  n^{-1/2+\epsilon}\Delta_j^{-1/2} \right) \ll \Delta_j,
	\end{align*}
	since $\Delta_j \ge  n^{3\e}\phi_n+  n^{-{1}/{3}+3\e}$ for $j\in \cal O_{3\e}$. Thus, we get that $|f_c(x)-t_c| = (1+\oo(1))\Delta_j$. Similarly, we can show that $ |f_c(x)-t_c| = (1+\oo(1))\Delta_{j-1} $. This gives \eqref{deltajj-1}  due to the choice of $x$ and the definition of $\rI_j$ and $\rI_{j-1}$. In addition, we also get that 
	\be\label{deltajj-12} 
	\Delta_j =(1+\oo(1)) \Delta_{j-1} , \quad    \text{if }\  j \in B_i \ \text{ and }\ j-1\in B_i.
	\ee
	It remains to verify that $j\in \cal O_{3\e}$ for all $j\in B_i$. Let $j_0$ be the smallest integer such that $\theta_{j_0} \notin B_i$. Since $|B_i|\le r$, by \eqref{deltajj-1} we have that %for $j<j_0$, 
	$$ \theta_{j_0 - 1 } > \theta_i - C \left( n^\e\phi_n \Delta_i +  n^{-1/2+\epsilon}\Delta_i^{1/2}\right) $$
	for some constant $C>0$. Then, using $i\in \cal O_{4\e}$, $j_0 \notin \cal O_{3\e}$ and \eqref{eq_gcomplex0}, we can check that
	$$\theta_{j_0 -1 } - \theta_{j_0} \gg  \left( n^\e\phi_n \Delta_{j_0-1} +  n^{-1/2+\epsilon}\Delta_{j_0-1}^{1/2}\right) +  \left( n^\e\phi_n \Delta_{j_0}  +  n^{-1/2+\epsilon}\Delta_{j_0}^{1/2}\right),$$
	which contradicts the definition of $B_i$. This concludes \eqref{deltajj-1}.

	Now, with \eqref{deltajj-1}, \eqref{deltajj-12} and $|B_i|\le r$, we obtain that
	\begin{equation}\label{diamBi}
		d_i:=\text{diam} \Big( \bigcup_{j \in B_i} \rI_j \Big) \leq C_r\left(  n^\e\phi_n \Delta_{i} +  n^{-1/2+\epsilon}\Delta_{i}^{1/2}\right), % \quad {i_0}:= \max_j \{j\in B_i\},
	\end{equation}
	for some constant $C_r>0$ depending on $r$ and $C_1$ only.
	%for some constant $C>0$. 
	On the other hand, by \eqref{eq_gcomplex0} we have that 
	\begin{align*}
		\theta_i-\lambda_+- d_i &\ge c \Delta_i^2 - C_r\left(  n^\e\phi_n \Delta_{i} +  n^{-1/2+\epsilon}\Delta_{i}^{1/2}\right) \gg n^{2\e}\phi_n^2 + n^{-2/3+3\epsilon} ,
	\end{align*}
	where we used $\Delta_i \ge n^{4\e}\phi_n+  n^{-{1}/{3}+4\e}$ for $i\in \mathcal O_{4\e}$ in the second step. Hence, there is a gap between the right edge of $\rI_0$ and the left edge of $ \bigcup_{j \in B_i} \rI_j$.

	Let $x_i$ and $y_i$ be the left and right end points of the interval $\bigcup_{j \in B_i} \rI_j $. Then, we pick the contour
	$${\bm\Gamma}_i:= \{z= x_i + \ii \eta: -d_i \le \eta \le d_i\}\cup  \{z= y_i + \ii \eta: -d_i \le \eta \le d_i\}\cup  \{z= E \pm \ii d_i: x_i \le E \le y_i\} ,$$
	which lies in the half plane on the right of $\rI_0$, and only includes $\theta_j$'s with $j\in B_i$ but no other points of the set $\{\theta_i: 1\le i \le r_+\}$. We again consider the functions $f_1$ and $f_2$ in \eqref{deff12}. We know that $f_1(z)$ has exactly $|B_i|$ eigenvalues at $\theta_j$, $j\in B_i$. Moreover, with the arguments in Lemma \ref{lem_gapI}, one can show that 
	$$\left\|\cal E(z)\right\| =\oo(1) \quad \text{for} \ \ z\in {\bm\Gamma}_i, \quad \text{where}\quad \cal E(z):=\left[f_c(z)I_r- \diag \left(t_1, \cdots, t_r \right)\right]^{-1} \cal E_r'(z).$$
	%where $c>0$ is a constant depending only on $c_{1,2}$ (or the $\tau$ in \eqref{assm20}). 
	Thus, we have 
	$$ \left|f_2(z) - f_1(z)\right| =|f_1(z)| \left|\det\left( 1 + \cal E(z)\right) - 1\right| < |f_1(z)|\quad \text{for} \ \ z\in {\bm\Gamma}_i.$$
	By Rouch{\' e}'s theorem, $f_2(z)$ has exactly $|B_i|$ eigenvalues in $\bigcup_{j \in B_i} \rI_j$. Together with Lemma \ref{lem_gapI} and a simple eigenvalue counting argument, we get that $\wt\lambda_i \in \bigcup_{j \in B_i} \rI_j$, and hence
	$$|\wt\lambda_i - \theta_i| \le d_i , \quad i \in \cal O_{4\e}.$$
	This concludes Claim \ref{lem_cont} by \eqref{diamBi}. 
\end{proof}

\noindent{\bf Step 4:} Finally, we consider the eigenvalues $\wt\lambda_i$ with $i \notin \mathcal O_{4\e}$. 
%First, we fix a configuration $\mathbf t(0)$ satisfying (\ref{eq_multione}). By Step 2, (\ref{eq_bound2ev}) and \eqref{interlacing_eq0}, we have
%\begin{equation}\label{eq_upper}
%\wt\lambda_i(0) \in \rI_0,\quad \text{and}\quad \wt\lambda_i (0) \geq \lambda_+ - n^{-2/3+\epsilon}.
%\end{equation}
%The above two estimates give that 
%$$|\wt\lambda_i (0)-\lambda_+|\le C_1^2 \psi_n^2 + n^{2\e}\phi_n^2 + n^{-2/3+3\epsilon}.$$
%Next we employ a similar continuity argument as in Step 3.  For $s \in [0,1],$
First, by (\ref{eq_bound2ev}) and \eqref{interlacing_eq0}, we have that 
\begin{equation}\label{iterlacing_t}
	\wt\lambda_i \geq \lambda_+-n^{-2/3+\epsilon}, \quad i \leq \varpi,
\end{equation}
which verifies the lower bound in \eqref{eq_nonspikepf}.
%Similar to the proof of Claim \ref{lem_cont}, 
%if $\rI_0$ is disjoint from the other $\rI_j$'s, then by the continuity of $\wt\lambda_i(s)$ and Lemma \ref{lem_gapI}, we can conclude that $\wt\lambda_i(s) \in \rI_0$ for all $s \in [0,1]$. Otherwise, 
For the upper bound, we consider the intervals in \eqref{rIi} and
$$\wh\rI_0 :=\left[0, \lambda_+ +  \wt C_1 \left(n^{8\e}\phi_n^2 + n^{-2/3+12\epsilon}\right)\right],$$
for a constant $\wt C_1>0$. Then, we define a partition $\mathcal B$ as in Step 3, where $B_0$ is the block of $\mathcal B$ that contains $i$. With the same arguments as in the proof of Claim \ref{lem_cont}, we can prove that 
\be\label{extra out1}\wh\rI_0 \cup \Big( \bigcup_{j\in B_0} \rI_{j}\Big)\subset \left[0, \lambda_+ +C_2\left(  n^{8\e}\phi_n^2 + n^{-2/3+12\epsilon}\right)\right]\ee
for a large enough constant $C_2>0$. Moreover, for any $j\notin B_0$, we have that $j\in \cal O_{4\e}$ by \eqref{eq_gcomplex0} as long as $\wt C_1$ is chosen large enough. Thus, using Lemma \ref{lem_gapI}, the result of Step 3 and a simple eigenvalue counting argument, we get that
$$ \wt\lambda_i \in \wh\rI_0  \cup \Big( \bigcup_{j\in B_0} \rI_{j}\Big), \quad i\notin \cal O_{4\e}.$$
%Together with (\ref{eq_upper}), \eqref{iterlacing_t} and the continuity of eigenvalues along the path, we obtain that
%\begin{equation}\nonumber
%\left|\wt\lambda_i(s) - \lambda_+\right| \le C_3\left( \psi_n^2 + n^{2\e}\phi_n^2 + n^{-2/3+3\epsilon}\right), \quad r_{\e}< i \le r  ,
%\end{equation}
%for all $s \in [0,1]$. Obviously, we can apply the same arguments to $r_{4\e}< i \le r+s$ by replacing $\rI_0(1)$ with $[0, \lambda_+ + \wt C_1^2 \psi_n^2 + n^{8\e}\phi_n^2 + n^{-2/3+12\epsilon}]$, 
This concludes the upper bound in \eqref{eq_nonspikepf} by \eqref{extra out1}, and hence completes the proof of Theorem \ref{main_thm}.
\end{proof} 

\appendix

\section{Proof of Theorem \ref{main_thm0.5}}\label{sec pfstick}
%For the rest of this section, we prove the eigenvalue sticking result, i.e. Theorem \ref{main_thm0.5}.

To conclude Theorem \ref{main_thm0.5}, we claim that it suffices to prove the following eigenvalue sticking estimate: 
	\be\label{boundstick}
	|\wt\lambda_{i+ r_+} - \lambda_i^b | \prec n^{-1}\al_+^{-1},
	\ee
where $\lambda_1^b \ge \lambda_2^b \ge \cdots \ge \lambda_q^b $ denote the eigenvalues of \smash{$\cal C^b_{X{\cal Y}}$}.  
In fact, this estimate shows that the non-outlier eigenvalues of $\cal C_{\cal X\cal Y}$ with non-trivial $A$ and $B$ stick to those of \smash{$\cal C^b_{X{\cal Y}}$} with $A=0$ and the same $B$. On the other hand, notice that \smash{$\cal C^b_{X{\cal Y}}$} has no outlier eigenvalues, because its PCC matrix is zero. Hence, as a special case of \eqref{boundstick}, we also know that $\lambda_i^b$ stick to the eigenvalues $\lambda_i$ of $\cal C_{XY}$ with $A=0$ and $B=0$. Together with \eqref{boundstick}, it gives the eigenvalue sticking estimate \eqref{boundstickextra}.

%{\cob Before stating our main results on the eigenvalues of the SCC matrix $\cal C_{\cal Y\cal X}$, we describe the behaviors of the eigenvalues of the null SCC matrix $\cal C^b_{\cal Y X}$ (recall \eqref{Cbxy}). We denote its eigenvalues by $\lambda_1^b \ge \lambda_2^b \ge \cdots \ge \lambda_q^b $. Then we define the quantiles of the density \eqref{LSD}, which give the classical locations of the eigenvalues. }}
	
In the proof of \eqref{boundstick}, we will need to use the following eigenvalue rigidity estimate for $\lambda_i^b$. When $B=0$, it reduces to \eqref{rigidity} in Lemma \ref{lem null0}. 
%i.e. there is no $Z$ term, then the same results have been proved in \cite{PartIII} under the same assumptions. %Moreover, the largest eigenvalue of $\cal C^b_{Y\cal X }$ satisfies the Tracy-Widom distribution. 
\begin{lemma}\label{lem null}
	Suppose the assumptions of Theorem \ref{main_thm} hold. Then, we have the following eigenvalue rigidity estimate: for any constant $\delta>0$ and all $1 \le i \le  (1-\delta)q$, 
	\be\label{rigidityb}
	|\lambda^b_i - \gamma_i | \prec i^{-1/3} n^{-2/3}. %\left[i \wedge (q+1-i)\right]^{-1/3} n^{-2/3}.
	\ee
\end{lemma}

Another tool for the proof of \eqref{boundstick} is the anisotropic local law for $G^b(z)$, which can be derived easily from the local law, Theorem \ref{thm_local}, for $G(z)$ by using  the approximate isometry condition \eqref{eq_iso} and the following Woodbury matrix identity: for $\mathcal{A},S,\mathcal{B},T$ of conformable dimensions,  
\begin{equation}\label{woodbury}
(\mathcal A+S\mathcal BT)^{-1}=\mathcal A^{-1}-\mathcal A^{-1}S(\mathcal B^{-1}+T\mathcal A^{-1}S)^{-1}T\mathcal A^{-1}.
\end{equation}
We define
%$$\Pi^b (Z,z) :=\Pi(z) - \Pi(z) \begin{pmatrix} \bU_b & 0 \\ 0 & \bE_b\end{pmatrix} \begin{pmatrix} I_{2r} &   \begin{pmatrix}  0 & 0 \\ 0 &  m_{4c}(z)\Sigma_b\end{pmatrix}   \\ \begin{pmatrix} 0 & 0 \\ 0 & c_2^{-1}m_{2c}(z)\Sigma_b\end{pmatrix}  & I_{2r} \end{pmatrix}^{-1}  \begin{pmatrix} 0 & \cal D_b\\ \cal D_b  & 0\end{pmatrix}\begin{pmatrix} \bU_b^{\top} & 0 \\ 0 & \bE_b^{\top}\end{pmatrix}\Pi(z),$$
\begin{align*}
&\Pi^b (z) :=\Pi(z) - \Pi(z) \begin{pmatrix} \bU_b & 0 \\ 0 & \bE_b\end{pmatrix} \begin{bmatrix} \begin{pmatrix}  0 & 0 \\ 0 &  c_2m_{2c}^{-1}(z)\Sigma_b\cal M_b\end{pmatrix}   &   \begin{pmatrix}  0 & 0 \\ 0 &  \cal M_b\end{pmatrix}   \\ \begin{pmatrix} 0 & 0 \\ 0 & \cal M_b\end{pmatrix}  & \begin{pmatrix}  0 & 0 \\ 0 &  m_{4c}^{-1}(z)\Sigma_b\cal M_b\end{pmatrix}  \end{bmatrix}\begin{pmatrix} \bU_b^{\top} & 0 \\ 0 & \bE_b^{\top}\end{pmatrix}\Pi(z),
\end{align*}
where
\be\label{def mb}\cal M_b:=\frac{\Sigma_b}{1+\Sigma_b^2} , \quad  \bU_b: = \begin{pmatrix} 0 & 0 \\ 0 & \begin{pmatrix}\mathbf u_1^b, \cdots, \mathbf u_r^b \end{pmatrix} \end{pmatrix}, \quad \bE_b:= \begin{pmatrix} 0 & 0 \\ 0 & \begin{pmatrix}Z^{\top}\mathbf v_1^b, \cdots, Z^{\top}\mathbf v_r^b \end{pmatrix} \end{pmatrix}.\ee

\begin{lemma} [Anisotropic local law for $G^b$]\label{thm_localb} %[Results on covariance matrices with small support]
Suppose the assumptions of Theorem \ref{main_thm}  hold. Fix any constant $\e>0$ and unit vectors $\mathbf u, \mathbf v \in \mathbb C^{\mathcal I}$ that are independent of $X$ and $Y$. We have that uniformly for $z\in S(\epsilon)$ (recall \eqref{SSET1}),
\begin{equation}\label{aniso_lawb}
	\left|  G^b_{\mathbf u\mathbf v}(z)   -  \Pi^b_{\mathbf u\mathbf v} (z)  \right| \prec \psi_n + \phi_n +  \Psi(z),
\end{equation}
and uniformly for $z\in D_{out}(\epsilon)$ (recall \eqref{eq_paraout}), 
\begin{equation}\label{aniso_outstrongb}
	\left|  G^b_{ \mathbf u \mathbf v}(z)  - \Pi^b_{ \mathbf u \mathbf v} (z)  \right|  \prec \psi_n  +  \phi_n + n^{-1/2}(\kappa+\eta)^{-1/4} .
\end{equation}
%Moreover we have the following bound
%\begin{equation}\label{weak_offb}
%\mathbb{E} \vert \langle \mathbf u ,  G^b(z)\mathbf v\rangle - \langle \mathbf u , \Pi^b(z)\mathbf v\rangle  \vert^2 \prec \Psi^2(z).
%\end{equation}
%for any deterministic unit vectors $\mathbf u , \mathbf v  \in \mathbb C^{\mathcal I}$. 
Moreover, \eqref{aniso_lawb} and \eqref{rigidityb} together imply that for any constant $\delta>0$,
\begin{align}
&	\max_{1\le k \le (1-\delta)q} \left\{ \left|\langle \mathbf u_1,S_{xx}^{-1/2}\xi_k^b\rangle \right|^2+\left|\langle \mathbf u_2,(S^b_{yy})^{-1/2}\zeta_k^b\rangle \right|^2\right\} \prec n^{-1},\label{delocal1b}\\
%\end{equation}
%and
%\begin{equation}
&	\max_{1\le k \le (1-\delta)q} \left\{ \left|\langle \mathbf u_3,X^{\top}S_{xx}^{-1/2}\xi^b_k\rangle \right|^2+\left|\langle \mathbf u_4,\cal Y^{\top}(S^b_{yy})^{-1/2}\zeta^b_k\rangle \right|^2\right\} \prec n^{-1},\label{delocal2b}
\end{align}
where $\{\xi_{k}^b\}_{k=1}^{p}$ are $\{\zeta_{k}^b\}_{k=1}^{q}$ are the left and right singular vectors of $\cal H^b$ (recall Definition \ref{resol_not}), respectively, and $\mathbf u_\al \in \mathbb C^{\mathcal I_\al}$ are unit vectors  independent of $X$ and $Y$. 
\end{lemma}
\begin{proof}%[Proof of Theorem \ref{thm_localb}]
Using \eqref{woodbury}, we can write $G^b(z)$ in \eqref{eqn_defGb} as
\be\label{woodbury2}
\begin{split}
	&G^b=  G - G\begin{pmatrix} \bU_b & 0 \\ 0 & \bE_b\end{pmatrix} \left[  \begin{pmatrix} 0 & \cal D_b^{-1}\\ \cal D_b^{-1}  & 0\end{pmatrix} + \begin{pmatrix} \bU_b^{\top} & 0 \\ 0 & \bE_b^{\top}\end{pmatrix}G\begin{pmatrix} \bU_b & 0 \\ 0 & \bE_b\end{pmatrix}  \right]^{-1}  \begin{pmatrix} \bU_b^{\top} & 0 \\ 0 & \bE_b^{\top}\end{pmatrix}G,
\end{split}
\ee
where $\cal D_b:=\begin{pmatrix} 0 & 0 \\ 0 & \Sigma_b\end{pmatrix}.$ Since $\cal D_b^{-1}$ is not well-defined, the above expression should be understood as
\begin{align*}
	&\left[  \begin{pmatrix} 0 & \cal D_b^{-1}\\ \cal D_b^{-1}  & 0\end{pmatrix} +  \begin{pmatrix} \bU_b^{\top} & 0 \\ 0 & \bE_b^{\top}\end{pmatrix}G\begin{pmatrix} \bU_b & 0 \\ 0 & \bE_b\end{pmatrix} \right]^{-1}\\
	 \equiv &\left[ 1 +  \begin{pmatrix} 0 & \cal D_b\\ \cal D_b  & 0\end{pmatrix}  \begin{pmatrix} \bU_b^{\top} & 0 \\ 0 & \bE_b^{\top}\end{pmatrix}G\begin{pmatrix} \bU_b & 0 \\ 0 & \bE_b\end{pmatrix}  \right]^{-1} \begin{pmatrix} 0 & \cal D_b\\ \cal D_b  & 0\end{pmatrix}.
\end{align*}
%for any $\cal I \times \cal I$ matrix $\cal M$. 
Combining \eqref{woodbury2} with Theorem \ref{thm_local}, Theorem \ref{thm_localout} and \eqref{eq_iso}, we can conclude \eqref{aniso_lawb} and \eqref{aniso_outstrongb}. The estimates \eqref{delocal1b} and \eqref{delocal2b} follow from \eqref{aniso_lawb} and \eqref{rigidityb} as in the proof of Lemma \ref{lem delocalX}, where the details can be found in the proof of Lemma 3.9 in \cite{PartIII}.
\end{proof}

As in Section \ref{sec mainthm}, from equation \eqref{detereq temp3}, we can derive a similar equation as \eqref{masterx4}. More precisely, suppose $\lambda$ is not an eigenvalue of $\cal C^b_{X\cal Y}$ and the following local law holds for $G^b(\lambda)$: 
%\begin{equation}\label{aniso_lawb222}
%\left| \langle \mathbf u, G^b(\lambda) \mathbf v\rangle - \langle \mathbf u, \Pi^b (\lambda)\mathbf v\rangle \right| \prec \Phi_n.
%\end{equation}
$$\begin{pmatrix} \bU_a^{\top} & 0 \\ 0 & \bE_a^{\top}\end{pmatrix} G^b(\lambda) \begin{pmatrix} \bU_a & 0 \\ 0 & \bE_a\end{pmatrix} - \Pi^b_r(\lambda) =\OO \left( \Phi_n\right) \quad \text{with high probability,}$$
where $\Phi_n$ is a deterministic parameter satisfying $0<\Phi_n\le n^{-c}$ for a small constant $c>0$, and
$$\Pi_r^b(\lambda):= \begin{bmatrix}\begin{pmatrix} c_1^{-1}m_{1c}(\lambda)I_r & 0 \\ 0 & 0\end{pmatrix} & 0\\0 & \begin{pmatrix}  m_{3c}(\lambda)I_r - \frac{h^2(\lambda)}{m_{4c}(\lambda)}\cal M_r \frac{\Sigma_b^2}{1+\Sigma_b^2}\cal M_r^{\top} & 0 \\ 0 &  0\end{pmatrix} \end{bmatrix}. $$
Then, $\lambda$ is an eigenvalue of $\cal C_{\cal X\cal Y}$ if and only if  
\be\label{masterx5}
\det \left(f_c(\lambda)I_r - \diag \left(t_1, \cdots, t_r \right) + \cal E_r(\lambda)\right) =0 ,
\ee
where $\cal E_r$ is an error term satisfying $\|\cal E_r\|\lesssim n^\e \psi_n+\Phi_n$ with high probability for any small constant $\e>0$. 
%More precisely, the bound on $\cal E_r$ is due to the following estimate 
%$$\begin{pmatrix} \bU_a^{\top} & 0 \\ 0 & \bE_a^{\top}\end{pmatrix} G^b(\lambda) \begin{pmatrix} \bU_a & 0 \\ 0 & \bE_a\end{pmatrix} - \Pi^b_r(\lambda) =\OO_\prec\left( \Phi_n\right) $$
%by local law \eqref{aniso_lawb222}, where
%$$\Pi_r^b(\lambda):= \begin{pmatrix}\begin{pmatrix} c_1^{-1}m_{1c}(\lambda)I_r & 0 \\ 0 & 0\end{pmatrix} & 0\\0 & \begin{pmatrix}  m_{3c}(\lambda)I_r - \frac{h^2(\lambda)}{m_{4c}(\lambda)}\cal M_r \frac{\Sigma_b^2}{1+\Sigma_b^2}\cal M_r^{\top} & 0 \\ 0 &  0\end{pmatrix} \end{pmatrix}. $$
Moreover, similar to \eqref{interlacing_eq0}, we have the following eigenvalue interlacing,
\begin{equation}\label{interlacing_eqb}
	\wt\lambda_i \in [\lambda^b_{i + r}, \lambda^b_{i-r}],
\end{equation}
where we adopt the convention that $\lambda^b_{i}=1$ if $i<1$ and $\lambda^b_i = 0$ if $i>q$. {This is the main reason why we need to prove \eqref{boundstick} first instead of proving \eqref{boundstickextra} directly: our proof of \eqref{boundstick} will use crucially the rank-$r$ interlacing in \eqref{interlacing_eqb}, while the rank-$2r$ interlacing in \eqref{interlacing_eq0} is not strong enough to yield \eqref{boundstickextra} directly.
}

%This is the main reason why we use $\cal C^b_{X\cal Y}$ and $G^b(z)$ instead of $\cal C_{X Y}$ and $G(z)$ for the proof of Theorem \ref{main_thm0.5}%as in Section \ref{sec mainthm}
%---the interlacing result \eqref{interlacing_eq0} is not strong enough by using a rank-$2r$ perturbation.

\begin{proof}[Proof of Theorem \ref{main_thm0.5}]
	%\begin{proof}[Proof of Theorem \ref{main_thm0.5}]
	For simplicity, in the following proof, we abbreviate $\phi_n + \psi_n$ as $\phi_n$. As a byproduct of the proof of Lemma \ref{lem null} in Section \ref{sec pflemma}, we obtain an averaged local law for $m^b(z)$ in equation \eqref{aver_inbbbb} below. By \eqref{aver_inbbbb}, \eqref{rigidityb}, Theorem \ref{main_thm}, Lemma \ref{SxxSyy}, \eqref{eq_iso} and Lemma \ref{thm_localb}, for any small constants $\e,\wt\e,\delta>0$ and fixed integer $\varpi \in \N$, we can choose a high-probability event $\Xi$ on which the following estimates hold: 
		\be\label{aver_lawb}
	|m^b(z)-m_c(z)|\le \frac{n^{\e/4}}{n\eta} ,\quad  \text{ for }\ z\in \wt S(\e,\wt\e);
	\ee
	\begin{equation} \label{eq_bound2evb}
	  |\lambda^b_i - \lambda_+ | \leq n^{-2/3+\epsilon/2}, \quad  \text{ for }\ 1\le i \le \varpi ;
	\end{equation}
	\begin{equation}\label{eq_stickingrigi2}
|\lambda_i^b-\gamma_i| \leq i^{-1/3} n^{-2/3+\epsilon/2} , \quad \text{ for }\ 1\le  i \leq (1-\delta) q;
	\end{equation}
	\begin{equation}\label{eq_stickoutlier}
	|\wt\lambda_{i}-\theta_i| \leq n^\e \phi_n \Delta_i + n^{-1/2+\e}\Delta_i^{1/2}, \quad \text{ for } \ 1\le i\le r_+ ;
	\end{equation}
	\begin{equation}\label{eq_stickingrigi}
		-  n^{-2/3 + \e/2} \le  \wt\lambda_{i}-\lambda_+  \leq n^{\e/2}\phi_n^2 + n^{-2/3+\e/2}, \quad \text{ for } \ r_+ +1 \le i \le \varpi;
	\end{equation}
	\be\label{op rough3} 
	c_0\le \min\{ \lambda_p(S_{xx}) , \lambda_q(S^b_{yy})\} \le \max\{\lambda_1(S_{xx}), \lambda_1(S^b_{yy})\} \le c_0^{-1} ;
	\ee
	\begin{equation}
	 \| ZZ^{\top} - I_{r}\| \le n^{\e/20}\phi_n ; \label{eq_isob}
	\end{equation}
	%for a fixed large integer $\varpi > r$;
	%\begin{equation}\label{eq_stickingrigi}
	%\mathbf{1}(\Xi)|\wt\lambda_{i}-\lambda_+| \leq C_1\psi_n^2 +n^{\e/2}\phi_n^2 + n^{-2/3+\e/2} , \quad \text{ for } \ r +1 \le i \le \varpi;
	%\end{equation}
	%for a fixed $\tau>0$,
	\begin{equation}\label{aniso_lawevb}
		 \left\|{ \begin{pmatrix} \bU_a^{\top} & 0 \\ 0 & \bE_a^{\top}\end{pmatrix}  G^b(z) \begin{pmatrix} \bU_a & 0 \\ 0 & \bE_a\end{pmatrix}}-\Pi^b_r(z) \right\| \le   n^{\e/2}\left(\phi_n + \Psi(z)\right),\quad \text{ for }\ z\in S(\e);
	\end{equation}
	\begin{equation} \label{eq_bound1evb}
		\begin{split}
		& \left\|{ \begin{pmatrix} \bU_a^{\top} & 0 \\ 0 & \bE_a^{\top}\end{pmatrix}  G^b(z) \begin{pmatrix} \bU_a & 0 \\ 0 & \bE_a\end{pmatrix}}-\Pi^b_r(z) \right\|  \le  n^{\e/2}\left(\phi_n + n^{-1/2}\kappa^{-1/4} \right),\quad \text{ for }\ z\in D_{out}(\epsilon);
		\end{split}%x\in [\lambda_+ + n^{-2/3+\e}, \varsigma_2 \lambda_+]; %\ \kappa \equiv \kappa(x):=x-\lambda_+. 
	\end{equation}
	%for unit vectors $\bu_\al\in \C^{\cal I_\al}$, $\al=1,2,3,4$, 
	\be\label{delocalav1}
	 \max_{1\le k \le (1-\delta)q} \left\{ \left|\langle \mathbf u_1,S_{xx}^{-1/2}\xi_k^b\rangle \right|^2+\left|\langle \mathbf u_2,(S^b_{yy})^{-1/2}\zeta_k^b\rangle \right|^2\right\} \le n^{-1 +\e/20};
	\ee
	%and
	\begin{equation}\label{delocalav2}
		 \max_{1\le k \le (1-\delta)q} \left\{ \left|\langle \mathbf u_3,X^{\top}S_{xx}^{-1/2}\xi_k^b\rangle \right|^2+\left|\langle \mathbf u_4,\cal Y^{\top}(S^b_{yy})^{-1/2}\zeta_k^b\rangle \right|^2\right\}  \le n^{-1 +\e/20}.
	\end{equation}
	Here, $c_0$ is a small enough constant, and the vectors $\bu_\al \in \C^{\cal I_\al}$, $\al=1,2,3,4$, belong to a set of vectors that is independent of $X$ and $Y$, has cardinality $n^{\OO(1)}$, and includes all the unit vectors that will be used in the proof. Again, the randomness of $X$, $Y$ and $Z$ only comes into play to ensure that $\Xi$ holds with high probability, and the rest of the proof will be entirely deterministic on the event $\Xi$. 
	
	%will follow the strategy as described in the beginning of Section \ref{sec:ev}. 
	%First of all, we construct the permissible region. Fix $\epsilon>0,$ by Theorem \ref{thm_outlier}, Lemma \ref{lem_locallaw} and \ref{thm_largerigidity}, we find that there exists a high-probability event $\Xi_0 \equiv \Xi_0(n, \epsilon)$ such that the following conditions hold. \\
	%(i). For the given $\tau>0,$ we have 
	%\begin{equation}\label{eq_stickingrigi}
	%\mathbf{1}(\Xi_0)|\lambda_{r_++s^++1}-\lambda_+| \leq n^{-2/3+\epsilon}, \ \mathbf{1}(\Xi_0)|\mu_i-\gamma_i| \leq i^{-1/3} n^{-2/3+\epsilon}, \ (i \leq \tau p). 
	%\end{equation}
	%(ii). For $z \in S(\varsigma_1, \varsigma_2)$ in (\ref{eq_paraset}), we have 
	%\begin{equation}\label{aniso_lawev}
	%\mathbf{1}(\Xi_0)  \norm{\bu^* (G(x)-\Pi(x)) \bu} \leq n^{\epsilon} \Psi(z),
	%\end{equation}
	%and 
	%\begin{equation*}
	%\mathbf{1}(\Xi_0)\max_{i,j} |\langle \bv_i, G(z) \bv_j \rangle-\langle \bv_i, \Pi(z) \bv_j \rangle| \leq n^{\epsilon}\Psi(z).
	%\end{equation*}
	%Till the end of the proof, we will fix a realization $\mathcal{Q}_1 \in \Xi_0.$  
	
	\vspace{5pt}
	
	\noindent{\bf Step 1:}  As in the proof of Theorem \ref{main_thm}, we first find a permissible region. For any $i$, we define the set 
	\begin{equation}\label{eq_omega}
		\Omega_i :=
		\Big\{x \in [\lambda^b_{i+r+1}, \lambda_+ + n^{2\e} \phi_n^2 +  n^{-2/3+2\epsilon}]: \ \text{dist} \Big (x, \text{Spec}(\cal C^b_{X\cal Y}) \Big )>n^{-1+\epsilon} \alpha_+^{-1} \Big\},
	\end{equation}
	where $\text{Spec}(\cal C^b_{X\cal Y})$ stands for the eigenvalue spectrum of $\cal C^b_{X\cal Y}$.
	
	\begin{lemma}\label{prop_eigensticking} 
		There exists a constant $C_1>0$ such that for $\alpha_+ \geq C_1  (n^\e\phi_n +   n^{-1/3+\epsilon})$ and $i \leq n^{1-2\epsilon} \alpha_+^3,$ 
		%there exists a constant $c_0>0$ such that 
		the set $\Omega_i$ contains no eigenvalue of $\cal C_{\cal X\cal Y}.$
	\end{lemma}
	
	%Next \footnote{From the this point, we only need it when $\tau$ in Assumption \ref{ass:spike} can be of any order.  When it is constant order, we can just ignore this part. } ,
	\begin{proof}
		In the proof, we always use the following spectral parameter 
		\begin{equation}\label{etax2}
			 z_x= x+\ii \eta_x  ,\quad  \text{with}\quad \eta_x:=n^{-1+ \epsilon} \alpha_+^{-1}.
		\end{equation}
		Suppose $x \in \Omega_i$. We first claim that for any deterministic unit vectors $\bu, \bv \in \Gamma$, %we have
		\begin{align}\label{eq_bound2}
			|G_{\bu \bv}^b(z_x)-G_{\bu \bv}^b(x)| & \le Cn^{\e/20}\im m^b(z_x) + C n^{\e/20}\eta_x, \quad x\in \Omega_i.
		\end{align}
		We use a similar argument as in the proof of Theorem \ref{thm_localout}. To illustrate the idea, for deterministic unit vectors  \be\label{decompose_uv_addadd}
		\mathbf v=\begin{pmatrix} \bv_1^{\top} , \bv_2^{\top}, \bv_3^{\top},\bv_4^{\top} \end{pmatrix}^{\top},\quad \mathbf u=\begin{pmatrix} \bu_1^{\top} , \bu_2^{\top}, \bu_3^{\top},\bu_4^{\top} \end{pmatrix}^{\top}\quad \text{with}\ \ \bu_\al, \bv_\al \in \C^{\cal I_\al}, 
		\ee
		we calculate $G^b_{\bu_1\bv_1}(z_x) - G^b_{\bu_1\bv_1}(x)$ as an example. As in \eqref{zz0}, we have
		\begin{align*}
			&\left|  G^b_{\mathbf u_1\mathbf v_1}(z_x) - G^b_{\mathbf u_1\mathbf v_1}(x) \right| \\
			&\lesssim \sum_{k \le (1-\delta)q} \frac{\eta_x  |\langle \bv_1,S_{xx}^{-1/2}{\xi}^b_k\rangle| |\langle \bu_1,S_{xx}^{-1/2}{\xi}_k^b\rangle| }{|\lambda^b_k - x|\left[(\lambda^b_k-x)^2 + \eta_x^2 \right]^{1/2}}  +  \eta_x\sum_{k>(1-\delta)q}{|\langle \bv_1,S_{xx}^{-1/2}{\xi}^b_k\rangle||\langle \bu_1,S_{xx}^{-1/2}{\xi}^b_k\rangle|} \\
			&\lesssim n^{-1+\e/20}  \sum_{k = 1}^{q} \frac{\eta_x  }{ (\lambda^b_k-x)^2 + \eta_x^2 } + \eta_x \lesssim  n^{\e/20}\im m^b(z_x) +  \eta_x,
		\end{align*}
		where in the second step we used \eqref{op rough3}, \eqref{delocalav1} and $|\lambda_k^b - x| \ge \eta_x$ for $x\in \Omega_i$, and in the last step we used the spectral decomposition of $m^b(z_x)$. The proofs for the rest of the cases $ G^b_{\bu_\al\bv_\beta}(z_x) - G^b_{\bu_\al\bv_\beta}(x)$, $\al,\beta=1,2,3,4$, are similar, so we omit the details. 
		
		%Next going through a similar calculation as at the beginning of this section, we can obtain a similar equation as \eqref{masterx4}: $x\in \Omega_i$ is an eigenvalue of $\cal C_{\cal X\cal Y}$ if and only if 
		Recall that $x\in \Omega_i$ is an eigenvalue of $\cal C_{\cal X\cal Y}$ if and only if \eqref{masterx5} holds, where $\cal E_r$ satisfies the following bound 
		%\be\label{masterx5}
		% \det \left(f_c(x)I_r - \diag \left(t_1, \cdots, t_r \right) + \cal E'_r(x)\right) =0 ,
		%\ee
		by \eqref{eq_bound2}, \eqref{aver_lawb} and \eqref{aniso_lawevb}:
		$$\|\cal E_r(x)\| \le C\left(  n^{\e/20}\im m_c(z_x) + n^{\e/20}\eta_x +n^{\e/2} \phi_n + n^{\e/2}\Psi(z_x) + \frac{n^{3\e/10}}{n\eta_x} \right) $$
		for some constant $C>0$. With \eqref{Immc} and the definition of $\Psi(z_x)$ in \eqref{eq_defpsi}, we can further bound that
		$$\|\cal E_r(x)\| \le C'\left( n^{\e/2} \phi_n +  n^{\e/2}\im m_c(z_x) +\frac{ n^{\e/2}}{n\eta_x} \right) $$
		for some constant $C'>0$. Now, to prove Lemma \ref{prop_eigensticking}, it suffices to show that for any $1\le j \le r$,
		\be\label{diag_big2}
		\left|f_c(x) - t_j \right| > C'\left( n^{\e/2} \phi_n +  n^{\e/2}\im m_c(z_x) +\frac{ n^{\e/2}}{n\eta_x} \right),\quad x\in \Omega_i. 
		\ee
		%as long as $x\in \Omega_i$. 

		Since $i \leq n^{1-2\epsilon} \alpha_+^3$, by \eqref{eq_stickingrigi2} we have that for $x\in \Omega_i$,
		\be\label{kappax2}
		-\left( n^{2\e} \phi_n^2 +  n^{-2/3+2\epsilon}\right) \le \lambda_+-x \lesssim \left( {i}/{n}\right)^{2/3} + i^{-1/3}n^{-2/3+\e/2} 	\lesssim   n^{-4\e/3}\al_+^2 ,  %+ n^{-1+5\e/3} \al_+^{-1}
		\ee
		where we used $\gamma_i\sim (i/n)^{2/3}$ and $\al_+\ge C_1n^{-1/3+\e}$. Then, by \eqref{eq_mcomplex0}, we have
		\begin{align*} 
		|f_{c}(x)-t_c|=|f_{c}(x)-f_{c}(\lambda_+)| \le C n^{-2\e/3}{\al_+} , \quad & x\in \Omega_i \cap \{x: x\le \lambda_+\},\\
		 |f_{c}(x)-t_c|=|f_{c}(x)-f_{c}(\lambda_+)|\le C\left(n^{\e} \phi_n +  n^{-1/3+\epsilon}\right), \quad & x\in \Omega_i \cap \{x: x> \lambda_+\}, 
		 \end{align*}
		for a constant $C>0$ that does not depend on $C_1$. Hence, as long as $C_1$ is large enough, we have
		\be\label{fctc1} |f_{c}(x)-t_c| \le \frac14 \al_+ \quad \Rightarrow \quad \left|f_c(x) - t_j \right|  \ge \frac34\al_+,\ee
		where we used the definition of $\al_+$ in \eqref{al+}.
		%Plugging the above two estimates into \eqref{sticking_dj} and using $|\wt\sigma_j^a + m_{2c}^{-1}(\lambda_+)|\ge \al_+$ by \eqref{alpha+}, we obtain that
		%\be \nonumber %\label{sticking_dj}
		%\left|\frac{d_j^a + 1}{d_j^a} - \frac{1}{1+m_{2c}(x)\sigma_j^a}\right|\gtrsim \al_+ 
		%\ee
		%as long as $c_0$ is sufficiently small. 
		On the other hand, with \eqref{Immc}, \eqref{etax2} and \eqref{kappax2}, we can verify that %for $x\in \Omega_i$ and $x\le \lambda_+$,
		$$ C'\left( n^{\e/2} \phi_n +  n^{\e/2}\im m_c(z_x) +\frac{ n^{\e/2}}{n\eta_x} \right) \le C'' \left( n^{\e/2} \phi_n + n^{\e/2}\sqrt{\kappa_x+\eta_x} + n^{-\e/2}\al_+\right) \ll \al_+$$
		for $x\in \Omega_i \cap \{x: x\le \lambda_+\}$, 
		and %for $x\in \Omega_i$ and $x> \lambda_+$,
		$$C'\left( n^{\e/2} \phi_n +  n^{\e/2}\im m_c(z_x) +\frac{ n^{\e/2}}{n\eta_x} \right)\le C'' \left( n^{\e/2} \phi_n +\frac{ n^{\e/2}\eta_x}{\sqrt{\kappa_x+\eta_x}}  + n^{-\e/2}\al_+\right)  \ll \al_+ $$
		for $x\in \Omega_i \cap \{x: x> \lambda_+\} $. Together with \eqref{fctc1}, these two estimates imply that \eqref{diag_big2} holds. This concludes the proof.
	\end{proof}
	
	\noindent{\bf Step 2:}  In this step, we perform a counting argument for a special case as in the following lemma. We postpone its proof until we finish the proof of Theorem \ref{main_thm0.5}. 
	\begin{lemma}\label{lem_counting}  
		Given $0\le r_+ \le r$, we choose a matrix $A\equiv A(0)$ of rank $r_+$ such that the eigenvalues configuration $\mathbf t\equiv \mathbf t(0):=(t_1, t_2 , \cdots , t_{r })$ of the PCC matrix satisfies that
		\begin{equation}\label{eq_multione2}
			%t_1>t_2 > \cdots > t_{r_+}>t_c >t_{r_+ + 1} = \cdots = t_r =0,\quad   
			(t_{r_+}-t_c)  \wedge \min_{1\le i \le r_+-1} \left(t_{i}-t_{i+1}\right)   \gtrsim 1,\quad t_{r_+ + 1} = \cdots = t_r =0.
		\end{equation}
		Then, for $i \leq n^{1-4\epsilon} \alpha_+^3(0)$, %$\nc \cor$\alpha_+ \geq n^{-1/3+\epsilon}$\nc and \cor$i \leq n^{1-2\epsilon} \alpha_+^3,$\nc 
		we have 
		\begin{equation}\label{eq_counting}
			|\wt\lambda_{i+r_+}-\lambda_i^b | \leq  n^{-1+2\epsilon} \alpha_+^{-1}(0),
		\end{equation}
		where $\al_+(0) $ is defined as in \eqref{al+} for $\mathbf t(0)$. (The meaning of the argument 0 will be clear in Step 3.) 
	\end{lemma}
	%We leave the proof of Lemma \ref{lem_counting} to Appendix \ref{appen:counting}. 
	
	%It is easy to choose the matrix $A(0)$ by fixing the singular vectors, and then choosing the singular values of $A(0)$ properly.
	
	%\vspace{5pt}
	
	%\begin{proof}[Proof of Theorem \ref{thm_eigenvaluesticking}]
	
	\noindent{\bf Step 3:} 
	%For general $\{\tsig_i\}$ without the assumption (\ref{eq_multione}), we shall 
	In this step, we employ a continuity argument as in \cite[Section 6.5]{KY} and \cite[Section S.4.2]{DY2}. We   choose a continuous ($n$-dependent) path $A(s)$ for $0\le s\le 1$, such that $A(1)=A$ is the matrix in Theorem \ref{main_thm0.5}, and $A(0)$ gives an eigenvalue configuration $\mathbf t(0)$ satisfying \eqref{eq_multione2}.
	%first choose a $A(0)$ such that $\mathbf t(0)$ satisfies \eqref{eq_multione2}, and then choose a continuous ($n$-dependent) path of $A(s)$ that connects $A(0)$ and $A(1)=A$. 
	Correspondingly, we have a continuous path of the configuration $\mathbf t(s)$ and the sample eigenvalues \smash{$\{\wt \lambda_i(s)\}_{i=1}^n$}. We can choose $A(s)$ such that 
	$$\inf_{s\in[0,1]}\al_+(s) \gtrsim \al_+ \equiv \al_+(1),$$ 
	where $\al_+(s)$ is defined as in \eqref{al+} for the eigenvalue configuration $\mathbf t(s)$.

	In this step we consider the case where  $\alpha_+ \geq C_1  (n^\e\phi_n +   n^{-1/3+\epsilon} )$ and $i \leq n^{1-4\epsilon} \alpha_+^3$. Without loss of generality, we rename $\al_+:=\inf_{s\in[0,1]}\al_+(s).$ Define 
	$$\wt I_0 :=\left\{x\in [0, \lambda_+ +n^{2\e} \phi_n^2 +  n^{-2/3+2\epsilon}]: \text{dist}\left(x,\text{Spec}(\cal C^b_{X\cal Y})\right) \le n^{-1+\epsilon} \alpha_+^{-1}  \right\}.$$
	Note that $\wt I_0$ is a union of connected intervals. 
	%We define a continuous path $\mathbf t(s)$ such that $\mathbf t(0)$ satisfies \eqref{eq_multione2} and $\mathbf t(1)$ is the configuration we are interested in. For $s\in [0,1]$, we use the notation $\al_+(s)$ with $\al_+\equiv \al_+(1)$, and without loss of generality we can choose $\mathbf t(s)$ such that $\al(s)$ is decreasing in $s$. %Moreover, we require that the path satisfies the properties (i)-(iii) in the proof of Theorem \ref{thm_outlier}, and that if $\wt I_i (1)$ and $\wt I_j (1)$ are two disjoint connected intervals of $\wt I(1)$, then $\wt I_i(t)\cap \wt I_j(t)=\emptyset$ for all $t\in [0,1]$.
	Due to the interlacing \eqref{interlacing_eqb}, we have
	\begin{equation}\label{interlacing_t2}
		\lambda^b_{i+r}\le \wt\lambda_{i}(s) \leq \lambda^b_{i-r},\quad s\in [0,1]. %, \ i \geq r_++s^++1.
	\end{equation}
	By Lemma \ref{prop_eigensticking} and Lemma \ref{lem_counting}, we know
	\begin{equation}\nonumber
		|\wt\lambda_{i+r_+}(0)-\lambda^b_i| \leq n^{-1+2\epsilon} \alpha_+^{-1} ,
		\ee
		and 
		\be
		\text{dist}\left(\wt \lambda_{i+r_+}(s),\text{Spec}(\cal C^b_{X\cal Y})\right) \le n^{-1+\epsilon} \alpha_+^{-1},\quad s\in [0,1].
	\end{equation}
	%since $\wt \lambda_{i+r+s}(s)$ satisfies \eqref{interlacing_t2} and $\lambda_i$ satisfies \eqref{eq_stickingrigi2}. 
	In addition, by continuity of eigenvalues with respect to $s$, we know that $\wt \lambda_{i+r_+}(s)$ is in the same connected component of $\wt I_0$ as $\wt \lambda_{i+r_+}(0)$. For any $i$, let $B_{i}$ be the set of $j$ such that $\lambda^b_i$ and $\lambda^b_j$ are in the same connected component of $\wt I_0$. Then, we conclude that 
	\begin{align*}
		\wt \lambda_{i+r_+}(s) & \in \bigcup_{j\in B_{i}:|i +r_+ - j | \le r} \left[\lambda^b_j - n^{-1+2\epsilon} \alpha_+^{-1} , \lambda^b_j +  n^{-1+2\epsilon} \alpha_+^{-1} \right],
	\end{align*}
	%where we again used estimates that are similar to \eqref{etali}. 
	which gives that
	\be \left|\wt \lambda_{i+r_+}(s)-\lambda^b_i\right| \le 2r  n^{-1+2\epsilon} \alpha_+^{-1},\quad s\in [0,1].\label{case1_extreme}
	\ee
	%when $\alpha_+ \geq n^{-1/3+2\epsilon}$ and $i \leq n^{1-4\epsilon} \alpha_+^3$.

	\vspace{5pt}
	\noindent{\bf Step 4:} Finally, we consider the case where $\alpha_+ <  C_1  ( n^\e\phi_n +   n^{-1/3+\epsilon} ) $ or $i >n^{1-4\epsilon} \alpha_+^3$. Suppose first that $\alpha_+< C_1  ( n^\e\phi_n +   n^{-1/3+\epsilon} )$. Then, by the assumption of Theorem \ref{main_thm0.5}, if $\e$ is small enough such that $\e<\e_0$, we must have
	\be\label{special small} \phi_n \le n^{-1/3}, \quad  \al_+\lesssim n^{-1/3+\e}.\ee
	Now, using \eqref{special small}, \eqref{interlacing_eqb}, (\ref{eq_stickingrigi2}) and (\ref{eq_stickingrigi}), we find that 
	\begin{equation*}
		|\wt\lambda_{i+r_+}-\lambda_i^b| \lesssim n^{-2/3+\e} \lesssim  n^{-1+2\e}\al_+^{-1}. 
	\end{equation*}
	On the other hand, suppose $i >n^{1-4\epsilon} \alpha_+^3$. If $i\le r$, then we have $\al_+ \lesssim n^{-1/3 + 4\e/3}$, and with the same argument as above, we get
	\begin{equation*}
		|\wt\lambda_{i+r_+}-\lambda_i^b| \lesssim n^{-2/3+\e} \leq n^{-1+3\e}\al_+^{-1}. 
	\end{equation*}
	Otherwise, using \eqref{interlacing_eqb} and (\ref{eq_stickingrigi2}), we get
	\begin{equation*}
		|\wt\lambda_{i+r_+}-\lambda_i^b| \lesssim i^{-1/3} n^{-2/3+\epsilon/2}  \leq  n^{-1+2\e}\al_+^{-1}. 
	\end{equation*}
	Combining the above three estimates with \eqref{case1_extreme}, we conclude \eqref{boundstick}, since $\e>0$ can be arbitrarily small. 
\end{proof}

For the proof of Lemma \ref{lem_counting}, we shall use an argument that extends the one in the proof of Proposition 6.8 in \cite{KY}. However, the proof in \cite[Section 7]{CCA} may also work, where the authors proved essentially the same result but only for $i \le \varpi$ with $\varpi$ being a fixed integer. 

%Hence we will not write down all the details. %We remark that in \cite{KY2013}, the results are only proved for the eigenvalues near the edge with $i \leq (\log n)^{C \log \log n},$ for some constant $C>0.$  Here we will prove that the same results hold further into the bulk. 

\begin{proof}[Proof of Lemma \ref{lem_counting}]
	Note that in this lemma, we have $\al_+\equiv \al_+(0) \sim 1$. %If $j$ is large enough such that $j^{-1/3}n^{-2/3}\le \delta n^{3\e/2}\eta_l(\gamma_j)$ for some constant $\delta>0$, then we immediately obtain \eqref{eq_counting} by \eqref{eq_stickingrigi2} and interlacing, Lemma \ref{lem_weylmodi}.
	In the first step, we group together the eigenvalues $\lambda_i$ that are close to each other. More precisely, let $\mathcal{B}=\{B_k\}$ be the finest partition of $\{1,\cdots, q\}$ such that $i<j$ belong to the same block of $\mathcal B$ if 
	$$|\lambda^b_{i}-\lambda^b_{j}| \leq  n^{- 1 + 7\epsilon/6} \alpha_+^{-1} .$$ 
	Note that each block $B_k$ of $\mathcal B$ consists of a sequence of consecutive integers. We order the blocks in the descending order, that is, if $k<l$  then $\lambda^b_{i_k}>\lambda^b_{i_l}$ for all $i_k \in B_k$ and $i_l \in B_l$. 
	
	We first derive a bound on the sizes of the blocks. We define $k^*$ such that $n_0:= \lceil n^{1-4\epsilon} \alpha_+^3\rceil \in B_{k^*}$. For any $k\le k^*$, we take $i < j $ such that $i$ and $j$ both belong to the block $B_k$. Then, by \eqref{interlacing_eqb} and (\ref{eq_stickingrigi2}), we have that for some constants $c,C>0$,
	\begin{equation}\nonumber
		c \left[ \left( {j}/{n}\right)^{2/3}-\left( {i}/{n}\right)^{2/3}  \right] - C i^{-1/3} n^{-2/3+\epsilon/2} \leq \lambda^b_{i}-\lambda^b_{j} \leq C(j-i) n^{- 1 + 7\epsilon/6} \alpha_+^{-1} .
	\end{equation}
	%Suppose $j\ge 2i$, then \eqref{boundij} and \eqref{etalE} gives that
	%$$ (j/n)^{2/3}\le C \left( n^{-2/3+\epsilon/2} +n^{-1/2+\e/2} (j/n)^{1/3}\right) + j n^{-1+2\epsilon} \alpha_+^{-1}\Rightarrow j\le Cn^{\e},$$
	%where we used $\alpha_+ \geq n^{-1/3+\epsilon}$. Thus there are two cases to consider: (a) $1\le i < j \le n^{2\e}$, and (b) $n^{2\e}/2 \le i < j \le \min\{2i,\al^* \}$. In case (a), 
	Now, using $j^{2/3} - i^{2/3} \ge j^{-1/3}(j-i) $, we obtain that 
	$$  \left( j^{-1/3} -C n^{- 1/3 + 7\epsilon/6} \alpha_+^{-1} \right)(j-i) \leq C i^{-1/3} n^{\epsilon/2} .$$
	From this estimate, we conclude that if $i$ and $j$ satisfy \be\label{extraij}1\le i\le j \le n^{1-15\e/4}, \ee
	then we have
	\begin{equation}\label{boundij}
		j-i \le C(j/i)^{1/3}n^{\e/2} .
	\end{equation}
	%Since the blocks are contiguous, 
	%With \eqref{boundij}, we can show that 
	Now, we claim that
	\begin{equation}\label{eq_sizeeigen}
		|B_k| \leq Cn^{3\e/4}  \quad \text{for } \ k=1,\cdots, k^*,
	\end{equation}
	and for any given $i_k \in B_k$,
	\begin{equation}\label{rigidity_size}
		|\lambda^b_{i}-\gamma_{i_k}|\leq i^{-1/3} n^{-2/3+ \epsilon}  \quad \text{for all }\ i \in B_k .
	\end{equation}
	%for all $i \in A_k$. 
	To prove \eqref{eq_sizeeigen} and \eqref{rigidity_size}, we denote
	$  \al_k :=\max_{i\in B_k} i$ and $\beta_k :=\min_{i\in B_k} i .$
	If $i\in B_k$ satisfies $i\ge \al_k/2$, then \eqref{boundij} gives that $\al_k-i\le Cn^{\e/2}$, with which we obtain that
	$$|\gamma_i-\gamma_{\al_k}|\le Ci^{-1/3}n^{-2/3+\e/2}.$$
	On the other hand, if $i\in B_k$ satisfies $i\le \al_k/2$, then \eqref{boundij} gives that
	$\al_k - i\le \al_k\le Cn^{3\e/4}$. Thus, % we get
	$$|\gamma_i-\gamma_{\al_k}|\le |\gamma_1-\gamma_{\al_k}| \le Cn^{-2/3+\e/2} \le Ci^{-1/3}n^{-2/3+3\e/4}.$$
	Combining the above two estimates with (\ref{eq_stickingrigi2}), we obtain that
	\begin{align*}
		|\lambda^b_{i}-\gamma_{i_k}|  &\le |\lambda^b_{i}-\gamma_{i}| +|\gamma_{i}-\gamma_{\al_k}|+ |\gamma_{\al_k}-\gamma_{i_k}| \le Ci^{-1/3}n^{-2/3+3\e/4} \le i^{-1/3}n^{-2/3+\e}.
	\end{align*}
	From the above proof, we see that \eqref{eq_sizeeigen} and \eqref{rigidity_size} as long as \eqref{extraij} holds. We still need to prove \eqref{extraij} for $i,j \in B_{k}$ with $k\le k^*$. In fact, if there is $j\in B_{k^*}$ such that $j\ge n^{1-15\e/4}$, then we can find $j'\in B_{k^*}$ such that  $n^\e \le j'-n_0 \le 2n^{\e}$, which contradicts \eqref{boundij} and \eqref{eq_sizeeigen}.

	We are now ready to complete the proof. For any $1\le k \le k^*$, we denote 
	\begin{equation}\label{defalbeta}
		\fa_k:=\min_{i \in B_k} \lambda^b_{i} = \lambda^b_{\al_k}, \quad \fb_k:=\max_{i \in B_k} \lambda^b_{i}=\lambda^b_{\beta_k}. 
	\end{equation}
	We introduce a continuous path as 
	\begin{equation}\label{eq_defnx0x1}
		x_s^k:=(1-s)\left(\fa_k - \delta_n/3\right)+s\left(\fb_k  + \delta_n/3\right), \quad s \in [0,1],   
	\end{equation}
	where $\delta_n:=n^{- 1 + 7\epsilon/6} \alpha_+^{-1}.$
	%Note that $x_0^k= a^k - \delta_n/3$ and $x_1^k= b^k  + \delta_n/3$. 
	The interval $[x_0^k, x_1^k]$ contains precisely the eigenvalues of $\mathcal C^b_{X\cal Y}$ that are in $B_k$, and the endpoints $x_0^k$ and $x_1^k$ are at distances at least $\delta_n/3$ from any eigenvalue of $\mathcal C^b_{X\cal Y}$. Then, we have the following proposition. We postpone its proof until we finish the proof of Lemma \ref{lem_counting}.

	\begin{proposition} \label{prop_atleast} Almost surely, there are at least $|B_k|$ eigenvalues of $\mathcal C_{\cal X\cal Y}$ in $[x_0^k, x_1^k]$. 
		%where $\operatorname{Spec}(Q)$ denotes the spectrum of $Q.$  
	\end{proposition}
	
	%In the proof, we discard a measure zero event using the convention \eqref{abscont}, which leads to 
	Here, ``almost surely" in the statement is due to the assumption \eqref{abscont}: in the proof we discard a measure zero non-generic event. We postpone the proof of Proposition \ref{prop_atleast} until we complete the proof of Lemma \ref{lem_counting}.  
	
	%\vspace{5pt}
	
	%By taking $\omega \rightarrow 0$ and using a standard perturbation argument, we deduce that 
	%\be\label{counting_omega}
	%\text{$\ctQ_1$ has at least $|A_k|$ eigenvalues in $[x_0^k, x_1^k]$ for $1\le k \le k^*$}.
	%\ee 
	We now use a standard interlacing argument to show that $\cal C_{\cal X\cal Y}$ has at most $|B_k|$ eigenvalues  in $[x_0^k, x_1^k]$.  By \eqref{interlacing_eqb}, there are at most $|B_1|+r_+$ eigenvalues of $\cal C_{\cal X\cal Y}$ in $[x_0^1, \infty)$ (recall that the rank of $A(0)$ is $r_+$). Moreover, with the argument in Section \ref{sec mainthm}, we can prove that \eqref{eq_stickoutlier} holds in the case $A\equiv A(0)$, i.e. there are exactly $r_+$ outliers. Then, together with Proposition \ref{prop_atleast}, it gives that there are exactly $|B_1|$ eigenvalues of $\cal C_{\cal X\cal Y}$ in $[x_0^1, x_1^1].$ Repeating this argument, we can show that $\cal C_{\cal X\cal Y}$ has exactly $|B_k|$ eigenvalues in $[x_0^k, x_1^k]$ for all $k=2,\cdots, k^*$. Moreover, using (\ref{eq_sizeeigen}) we find that for any $i \in B_k$, 
	\begin{equation*}
		\sup\Big\{ |x-\lambda^b_{i}|: x \in [x_0^k, x_1^k]\Big\} \leq Cn^{3\e/4} \left(n^{- 1 + 7\epsilon/6} \alpha_+^{-1} \right) \le n^{- 1 + 2\e} \alpha_+^{-1} ,
	\end{equation*}
	which concludes Lemma \ref{lem_counting}. 
\end{proof}

%The proof of Proposition \ref{prop_atleast} is similar to the argument in \cite[Section 6.4]{KY}. Hence we will not write down all the details. %Hence we will not give all the details. 
Finally, we give the proof of Proposition \ref{prop_atleast}.
\begin{proof}[Proof of Proposition \ref{prop_atleast}] 
	For the spectral decomposition of $R^b(z)$ (which takes a similar form as \eqref{spectral1}), we define
	\be\label{PBkc}{P}_{B_k} R^b\left( z \right) := \sum\limits_{l \in B_k} \frac{1}{\lambda^b_l-z}\left( {\begin{array}{*{20}c}
			{{\xi ^b_l (\xi^b _l)^{\top}  }} & {-z^{-1/2}(\lambda_l^b)^{1/2} \xi ^b_l (\zeta^b _{ l})^{\top}}  \\
			{-z^{-1/2} (\lambda_l^b)^{1/2} \zeta^b _{l} (\xi^b _l)^{\top}  } & {\zeta^b _l (\zeta ^b_l)^{\top} }  \\
	\end{array}} \right) ,\ee
	and ${P}_{B_k^c} R^b(z) := R^b(z)- P_{B_k}R^b(z).$ We define ${P}_{B_k} G^b$ by replacing $R$ and $Y$ with  ${P}_{B_k} R^b$ and $\cal Y$ in \eqref{GL1}, \eqref{GR1} and \eqref{GLR1},  that is, 
	$${P}_{B_k} G^b:=\begin{pmatrix} {P}_{B_k} \cal G_L^b & {P}_{B_k} {\cal G}_{LR}^b \\ {P}_{B_k} {\cal G}_{RL}^b & {P}_{B_k} {\cal G}_{R}^b \end{pmatrix},$$
	where
	\be\label{GL1bbbb}
	{P}_{B_k} \cal G_L^b  :=  \begin{pmatrix} (XX^\top)^{-1/2} & 0 \\ 0 & (\cal Y \cal Y^\top)^{-1/2} \end{pmatrix}{P}_{B_k} R^b \begin{pmatrix} (XX^\top)^{-1/2} & 0 \\ 0 & (\cal Y \cal Y^\top)^{-1/2} \end{pmatrix},
\ee
\be\label{GR1bbbb}
{P}_{B_k} \cal G_R^b : =   \begin{pmatrix}  z  I_n & z^{\frac12}I_n\\ z^{\frac12}I_n &  z  I_n\end{pmatrix} +   \begin{pmatrix}  z  I_n & z^{\frac12}I_n\\ z^{\frac12}I_n &  z  I_n\end{pmatrix}  \begin{pmatrix} X^{\top} & 0 \\ 0 & \cal Y^{\top} \end{pmatrix} {P}_{B_k} \cal G_L^b \begin{pmatrix} X & 0 \\ 0 &  \cal Y \end{pmatrix} \begin{pmatrix}  z  I_n & z^{\frac12}I_n\\ z^{\frac12}I_n &  z  I_n\end{pmatrix}  ,
\ee
\be\label{GLR1bbbb}
\begin{split}
& {P}_{B_k} {\cal G}_{LR}^b(z):= -{P}_{B_k} \cal G_L^b(z) \begin{pmatrix} X & 0 \\ 0 &  \cal Y \end{pmatrix} \begin{pmatrix}  z  I_n & z^{\frac12}I_n\\ z^{\frac12}I_n &  z  I_n\end{pmatrix}  , \quad  {P}_{B_k} {\cal G}_{RL}^b(z):= {P}_{B_k} {\cal G}_{LR}^b(z)^\top . %-  \begin{pmatrix}  z  I_n & z^{\frac12}I_n\\ z^{\frac12}I_n &  z  I_n\end{pmatrix}  \begin{pmatrix} X^{\top} & 0 \\ 0 & Y^{\top} \end{pmatrix} {\cal G}_L(z).
\end{split}
\ee
Then, we define ${P}_{B_k^c} G^b(z):=G^b(z)-{P}_{B_k} G^b(z)$. Given any $x \in [x_0^k, x_1^k]$, we denote $z_x:= x+ \ii \eta_x$ with $\eta_x:=n^{- 1 + 7\epsilon/6} \alpha_+^{-1} .$
	We claim that
	\begin{align}\label{eq_bound22}
		\left\|\begin{pmatrix} \bU_a^{\top} & 0 \\ 0 & \bE_a^{\top}\end{pmatrix} \left[P_{B^c_k}G^b(z_x)-P_{B^c_k}G^b(x)\right] \begin{pmatrix} \bU_a & 0 \\ 0 & \bE_a\end{pmatrix} \right\|\lesssim  n^{\e/20}\im m^b (z_x) +   n^{\e/20}\eta_x . %& \lesssim \sum_{i=1}^2 \left[\im G_{\bu_i \bu_i}(z_x)+\im G_{\bv_i \bv_i}(z_x)\right].
	\end{align}
	The proof is very similar to that for (\ref{eq_bound2}). For example, for deterministic unit vectors $\bv$ and $\bu$ in \eqref{decompose_uv_addadd}, using \eqref{GL1bbbb}, \eqref{op rough3}  and \eqref{delocalav1},  we get  
	\begin{align*}
		&\left|  P_{B_k^c}G^b_{\mathbf u_1\mathbf v_1}(z_x) - P_{B_k^c}G^b_{\mathbf u_1\mathbf v_1}(x) \right| \\
		&\lesssim \sum_{l \notin B_k, l\le (1-\delta)q} \frac{\eta_x  |\langle \bv_1,S_{xx}^{-1/2}{\xi}^b_l\rangle| |\langle \bu_1,S_{xx}^{-1/2}{\xi}^b_l\rangle| }{|\lambda_l^b - x|\left[(\lambda_l^b-x)^2 + \eta_x^2 \right]^{1/2}} +  \eta_x \sum_{l> (1-\delta)q}{|\langle \bv_1,S_{xx}^{-1/2}{\xi}^b_l\rangle||\langle \bu_1,S_{xx}^{-1/2}{\xi}^b_l\rangle|} \\
		&\lesssim n^{-1+\e/20}  \sum_{l = 1}^{q} \frac{\eta_x  }{ (\lambda_l^b-x)^2 + \eta_x^2 } +  \eta_x \lesssim  n^{\e/20}\im m^b(z_x) + \eta_x,
	\end{align*}
	where in the second step we used $|\lambda_l^b - x| \gtrsim \eta_x$ for $l\notin B_k$. The proofs for the rest of the cases $ P_{B_k^c}G^b_{\bu_\al\bv_\beta}(z_x) - P_{B_k^c}G^b_{\bu_\al\bv_\beta}(x)$, $\al,\beta=1,2,3,4$, are similar, so we omit the details. 
	%For example, for the terms with $G_{\bu_1\bv_2}(\cdot)$, we have
	%\begin{align*}
	%& |P_{A_k^c}G_{\bu_1 \bv_2}(z_x)-P_{A_k^c}G_{\bu_1 \bv_2}(x)|   \lesssim \eta_x|G_{\bu_1 \bv_2}(z_x)| +\sum_{l \notin A_k}\sqrt{\lambda_l}\left|\langle \bu_1,\bxi_l\rangle \langle \bzeta_l,\bv_2\rangle\right| \left|\frac{\eta_x}{(\lambda_l-x-\ii\eta_x)(\lambda_l - x)}\right|\\
	%&\lesssim \sum_{l\notin A_k} \left(\left|\langle \bu_1,\xi_l\rangle \right|^2+\left|\langle \zeta_l,\bv_2\rangle\right|^2\right) \frac{\eta_x}{(\lambda_l-x)^2+(\eta_x)^2} = \operatorname{Im} G_{\bu_1 \bu_1}(z_x)+\operatorname{Im} G_{\bv_2 \bv_2}(z_x),
	%\end{align*}
	%%and estimate (..\eqref{etax2})
	%%\begin{align*}
	%%\left| \sum_{\alpha \notin A_q} \frac{\langle {\mathbf{w}}_{\alpha}, {\mathbf{v}}_i \rangle \langle {\mathbf{w}}_{\alpha}, {\mathbf{v}}_j \rangle}{\mu_{\alpha}-x}-\sum_{\alpha \notin A_q} \frac{\langle {\mathbf{w}}_{\alpha}, {\mathbf{v}}_i \rangle \langle {\mathbf{w}}_{\alpha}, {\mathbf{v}}_j \rangle}{\mu_{\alpha}-x-\ri \eta} \right| \leq 2\Big( \text{Im} \ G_{{\mathbf{v}}_i {\mathbf{v}}_i}(x+\ri \eta)+\text{Im} \ G_{{\mathbf{v}}_j {\mathbf{v}}_j}(x+\ri \eta) \Big),
	%%\end{align*}
	%where in the second step we used that $|x-\lambda_l| \gtrsim \eta_x$ for any $x \in [x_0^k, x_1^k]$ and $l\notin A_k$. For the rest of the cases with $G_{\bu_1\bv_1}(\cdot)$, $G_{\bu_2\bv_1}(\cdot)$ and $G_{\bu_2\bv_2}(\cdot)$, the proof of \eqref{eq_bound22} is similar. 
	%conclude from (\ref{eq_defnx0x1}) that $|x-\mu_{\alpha}| \geq 2 n^{-1+4 \epsilon} \alpha_+^{-1}/3$ for $\alpha \notin A_q.$ 
	
	Next, we claim that 
	\begin{equation}\label{eq_rxorder}
		\left| P_{B_k}G^b_{\bu\bv}(z_x) \right| + \left| P_{B_k}G^b_{\bu\bv}(x_0^k) \right| \le n^{-\epsilon/3}.
	\end{equation}
	For example, for the $z_x$ term, we have
	\begin{equation*}
		\left| P_{B_k}G^b_{\bu_1\bv_1}(z_x) \right| =	\left| \sum_{l \in B_k} \frac{\langle {\mathbf{u}}_1,S_{xx}^{-1/2} {\xi}^b_l \rangle \langle {\xi}^b_l S_{xx}^{-1/2}, {\mathbf{v}}_{1} \rangle}{\lambda^b_{l}-z_x} \right| \leq Cn^{3\e/4} \eta_x^{-1} n^{-1+\e/20} \ll n^{ - \epsilon/3},
	\end{equation*}
	where we used \eqref{delocalav1} and \eqref{eq_sizeeigen} in the second step. The proofs for the rest of the cases $ P_{B_k}G^b_{\bu_\al\bv_\beta}(z_x)$, $\al,\beta=1,2,3,4$, are similar. For $z=x_0^k$, the proof is the same except that we need to use $|\lambda^b_{l}-x_0^k| \gtrsim n^{- 1 + 7\epsilon/6} \alpha_+^{-1}$ for $l\in B_k$. 
	
	Now, we remove the zero singular values of $A$ and redefine that 
	$$\Sigma_a:= \diag(a_1, \cdots, a_{r_+}), \quad  U_a =  \begin{pmatrix}\mathbf u_1^a, \cdots, \mathbf u_{r_+}^a \end{pmatrix} , \quad E_a= \begin{pmatrix}Z^{\top}\mathbf v_1^a, \cdots,Z^{\top} \mathbf v_{r_+}^a \end{pmatrix} .$$
	Inspired by \eqref{detereq temp3}, for $x \notin \text{spec}(\mathcal C^b_{X\cal Y})$, we define
	$$\cal M(x):=\begin{pmatrix} 0 & \Sigma^{-1}_a \\ \Sigma^{-1}_a  & 0\end{pmatrix}  +  \begin{pmatrix} U_a^{\top} & 0 \\ 0 & E_a^{\top}\end{pmatrix} \begin{pmatrix} \cal G^b_1(x) & \cal G_{13}^b(x) \\ \cal G_{31}^b(x) &  \cal G_3^b(x) \end{pmatrix} \begin{pmatrix} U_a & 0 \\ 0 & E_a\end{pmatrix} ,$$
	where we recall that  $\cal G_{\al}^b$ is the $\cal I_\al\times \cal I_\al$ block of $G^b$ (cf. Definition \ref{resol_not}), and we have used $\cal G_{\al\beta}^b$ to denote the $\cal I_\al\times \cal I_\beta$ block of $G^b$. We know that almost surely, $x \in \operatorname{Spec}(\cal C_{\cal X\cal Y})\setminus\operatorname{Spec}(\cal C^b_{ X\cal Y})$ if and only if $\cal M (x)$ is singular. 
	To simplify notations, we denote
	$$ [G^b(z)]_{1, 3}:=\begin{pmatrix} \cal G^b_1(z) & \cal G_{13}^b(z) \\ \cal G_{31}^b(z) &  \cal G_3^b(z) \end{pmatrix}.$$
	%For the rest of the cases with $G_{\bu_1\bv_2}(\cdot)$, $G_{\bu_1\bv_2}(\cdot)$ and $G_{\bu_2\bv_2}(\cdot)$, the proof of \eqref{eq_rxorder} is similar. 
	Now, using \eqref{aver_lawb}, \eqref{aniso_lawevb}, \eqref{eq_bound22} and \eqref{eq_rxorder}, we obtain that
	% (\ref{s_bound2}), when $1 \leq i,j \leq r,$ we conclude from (\ref{eq_rxorder}) that 
	\begin{align}
		\cal M (x)&=\begin{pmatrix} 0 & \Sigma^{-1}_a \\ \Sigma^{-1}_a  & 0\end{pmatrix}  +  \begin{pmatrix} U_a^{\top} & 0 \\ 0 & E_a^{\top}\end{pmatrix} \left[ P_{B_k}G^b(x)  \right]_{1,3}\begin{pmatrix} U_a & 0 \\ 0 & E_a\end{pmatrix} \nonumber\\
		&\quad +  \begin{pmatrix} U_a^{\top} & 0 \\ 0 & E_a^{\top}\end{pmatrix} \left[  P_{B_k^c}\left(G^b(x)-G^b(z_x)\right) + G^b(z_x) - P_{B_k} G^b(z_x)   \right]_{1,3}\begin{pmatrix} U_a & 0 \\ 0 & E_a\end{pmatrix} \nonumber\\
		&=\begin{pmatrix} 0 & \Sigma^{-1}_a \\ \Sigma^{-1}_a  & 0\end{pmatrix}  + \begin{pmatrix} U_a^{\top} & 0 \\ 0 & E_a^{\top}\end{pmatrix} \left[P_{B_k}G^b(x)\right]_{1,3}\begin{pmatrix} U_a & 0 \\ 0 & E_a\end{pmatrix} +   \left[\Pi_r^b(z_x) \right]_{1,3}+ R_0(x) \nonumber\\
		&=\begin{pmatrix} 0 & \Sigma^{-1}_a \\ \Sigma^{-1}_a  & 0\end{pmatrix}  +  \begin{pmatrix} U_a^{\top} & 0 \\ 0 & E_a^{\top}\end{pmatrix} \left[P_{B_k}G^b(x)\right]_{1,3}\begin{pmatrix} U_a & 0 \\ 0 & E_a\end{pmatrix} +    \left[\Pi_r^b(\lambda_+) \right]_{1,3}+ R(x),\label{eq_m1}
	\end{align}
	%\begin{align}
	%&\cal M (x)=1+ \begin{pmatrix} 0 & \cal D_a \\ \cal D_a  & 0\end{pmatrix}  \begin{pmatrix} \bU_a^{\top} & 0 \\ 0 & \bE_a^{\top}\end{pmatrix} \left[ P_{B_k}G^b(x)+ P_{B_k^c}\left(G^b(x)-G^b(z_x)\right) + G^b(z_x) - P_{B_k} G^b(z_x)   \right]\begin{pmatrix} \bU_a & 0 \\ 0 & \bE_a\end{pmatrix} \nonumber\\
	%&=1+ \begin{pmatrix} 0 & \cal D_a \\ \cal D_a  & 0\end{pmatrix}  \begin{pmatrix} \bU_a^{\top} & 0 \\ 0 & \bE_a^{\top}\end{pmatrix} P_{B_k}G^b(x)\begin{pmatrix} \bU_a & 0 \\ 0 & \bE_a\end{pmatrix} + \begin{pmatrix} 0 & \cal D_a \\ \cal D_a  & 0\end{pmatrix}  \Pi_r^b(z_x) + R_0(x) \nonumber\\
	%&=1+ \begin{pmatrix} 0 & \cal D_a \\ \cal D_a  & 0\end{pmatrix}  \begin{pmatrix} \bU_a^{\top} & 0 \\ 0 & \bE_a^{\top}\end{pmatrix} P_{B_k}G^b(x)\begin{pmatrix} \bU_a & 0 \\ 0 & \bE_a\end{pmatrix} + \begin{pmatrix} 0 & \cal D_a \\ \cal D_a  & 0\end{pmatrix}  \Pi_r^b(\lambda_+) + R(x),\label{eq_m1}
	%\end{align}
	where 
	$$\left[\Pi_r^b(z)\right]_{1,3}:= \begin{pmatrix} c_1^{-1}m_{1c}(z)I_r  & 0\\0 &  m_{3c}(z)I_r - \frac{h^2(z)}{m_{4c}(z)}\cal M_r \frac{\Sigma_b^2}{1+\Sigma_b^2}\cal M_r^{\top}  \end{pmatrix},$$
	and $R_0$ and $R_1$ are two matrices satisfying that 
	\begin{align*}
	\|R_0(x)\|&=\OO\left(n^{\e/20}\eta_x+n^{\e/20} \im m_{c}(z_x) + n^{\e/2}\Psi(z_x)  + n^{\e/2}\phi_n + n^{-\e/3}\right) =\OO\left(n^{-\e/3}\right),\\
	%and 
	\|R(x)\|&=\left\|R_0(x) + \OO(\sqrt{\kappa_x + \eta_x})\right\| =\OO\left(n^{-\e/3}\right).
	\end{align*}
	In bounding the $\|R_0(x)\|$ and $\|R(x)\|$, we also used Lemma \ref{lem_mbehavior}, \eqref{eq_defpsi} and that 
	$$\kappa_x \le \max \left\{|\lambda_+ -x_0^k|, |\lambda_+ -x_1^k|\right\} \lesssim (n^{1-15\e/4}/n)^{2/3} + n^{-2/3+\e}+ n^{-1+7\e/6}\al_+^{-1} \ll n^{-\e/3},$$
	where in the second step we used \eqref{extraij}, \eqref{rigidity_size} and the definition \eqref{eq_defnx0x1}. Moreover, $R(x)$ is real symmetric (because all the other terms in the line \eqref{eq_m1} are real symmetric) and continuous in $x$ on the extended real line $ \overline{\mathbb R}$.
	%where $R(x)=O(n^{-\epsilon}/2)$  is continuous in $x$ and independent of $\Delta$.  
	
	The rest of the proof follows from a continuity argument, which is exactly the same as the one in \cite[Section 6.4]{KY}. Instead of writing down all the details, we shall give an almost rigorous argument to show how equation \eqref{eq_m1} implies Proposition \ref{prop_atleast}. 
	
	First, we claim that $\cal M(x)$ has some negative singular values when $x=x_0^k$. By \eqref{eq_rxorder}, equation \eqref{eq_m1} gives that
	$$\cal M (x_0^k)=\begin{pmatrix} 0 & \Sigma^{-1}_a \\ \Sigma^{-1}_a  & 0\end{pmatrix}   +    \left[\Pi_r^b(\lambda_+)\right]_{1,3} + \OO(n^{-\e/3}).$$
	Let $\bv_i$ be an eigenvector of 
	$$ \frac{\Sigma_a}{(1+\Sigma_a^2)^{1/2}}\cal M_r \frac{\Sigma_b^2}{1+\Sigma_b^2}\cal M_r^{\top} \frac{\Sigma_a}{(1+\Sigma_a^2)^{1/2}}$$ 
	with eigenvalue $t_i$. Then, for $\bu_i =: \begin{pmatrix} m_{3c}(\lambda_+) (1+\Sigma_a^2)^{-1/2}  \bv_i \\ \Sigma_a (1+\Sigma_a^2)^{-1/2} \bv_i\end{pmatrix}$, we can verify that
	$$ \bu_i^{\top} \cal M (x_0^k) \bu_i =  \frac{h^2(\lambda_+)}{m_{4c}(\lambda_+)} \left( f_{c}(\lambda_+) -  t_i \right)\|\bv_i\|^2+ \OO(n^{-\e/3})\|\bv_i\|^2<0,$$
	where we used $m_{4c}(\lambda_+)>0$, $t_i>t_c=f_c(\lambda_+)$ and $t_i - t_c \sim 1$ for $1\le i\le r_+$.

	Next, we claim that for $l\in B_k$, almost surely, $\cal M(x)$ is positive definite when $x\uparrow \lambda_l^b$ and negative definite when $x\downarrow \lambda_l^b$. To see why, we pick any unit vector 
	$\bv=\begin{pmatrix}\bv_1^\top,\bv_2^\top\end{pmatrix}^\top$, $\bv_1,\bv_2\in \R^{r_+}$, and denote $\wt\bv=\begin{pmatrix}\bv_1^\top , \mathbf 0_{r_+} , \bv_2^\top , \mathbf 0_{r_+}\end{pmatrix}^\top$. Then, 
	\begin{align}
		\bv^{\top} \cal M(x) \bv &=\OO(1) + \wt\bv^{\top}  \begin{pmatrix} \bU_a^{\top} & 0 \\ 0 & \bE_a^{\top}\end{pmatrix} P_{B_k}G^b(x) \begin{pmatrix} \bU_a & 0 \\ 0 & \bE_a\end{pmatrix} \wt\bv \nonumber\\
		&=\OO(1) + \wt\bw^{\top} \begin{pmatrix}P_{B_k} \cal G^b_L(x) & - P_{B_k}\cal G^b_L(x) \\ -P_{B_k}\cal G^b_L(x) & P_{B_k}\cal G^b_L(x)\end{pmatrix}\wt\bw = \OO(1) + \bw^{\top} P_{B_k}R^b(x) \bw , \label{spec_extra1}
	\end{align}
	where in the second step we used \eqref{GR1bbbb} and \eqref{GLR1bbbb} with 
	$$\wt \bw=\begin{pmatrix} \bw_1\\ \bw_2 \end{pmatrix} := \begin{bmatrix} I_{p+q} & 0 \\ 0 & \begin{pmatrix} X & 0 \\ 0 & \cal Y \end{pmatrix} \begin{pmatrix} xI_n & x^{1/2}I_n \\ x^{1/2}I_n & xI_n\end{pmatrix} \end{bmatrix}  \begin{pmatrix} \bU_a & 0 \\ 0 & \bE_a\end{pmatrix} \wt\bv, \quad \bw_{1} , \bw_{2}  \in \R^{p+q},$$
	and in the third step we used \eqref{GL1bbbb} with 
	$$ \bw :=\begin{pmatrix} S_{xx}^{-1/2} & 0 \\ 0 & (S^b_{yy})^{-1/2} \end{pmatrix}(\bw_1-\bw_2).$$
	Using the spectral decomposition \eqref{PBkc}, we can write
	\begin{equation}
		\begin{split}
	&	P_{B_k}R^b\left( x \right)  	= \frac{1}{2} \sum\limits_{l \in B_k}  \left[\frac{x^{-1/2}}{(\lambda_l^b)^{1/2}- x^{1/2}}\begin{pmatrix} \xi_l^b \\ -\zeta_l^b\end{pmatrix} \begin{pmatrix} \xi_l^b\\ -\zeta_l^b\end{pmatrix}^{\top} - \frac{x^{-1/2}}{(\lambda_l^b)^{1/2}+x^{1/2}}\begin{pmatrix} \xi_l^b \\ \zeta_l^b\end{pmatrix} \begin{pmatrix}  \xi_l^b\\ \zeta_l^b \end{pmatrix}^{\top}\right].  \label{spectral1extra}
		\end{split}
	\end{equation}
	In particular, it has poles at $x=\lambda_l^b$ for $l\in B_k$. Combining \eqref{spec_extra1} and \eqref{spectral1extra}, we conclude the claim.
	
	With the above two claims and a simple continuity argument, we see that there exists $x\in (x_0^k, \lambda_{\al_k}^b)$ (recall \eqref{defalbeta}) such that $\cal M(x)$ is singular, and	for any $l,l-1\in B_k$, there exists $x\in (\lambda_{l}^b, \lambda_{l-1}^b)$ such that $\cal M(x)$ is singular. This gives at least $|B_k|$ eigenvalues of $\mathcal C_{\cal X\cal Y}$ inside $[x_0^k, x_1^k]$ and hence completes the proof. Writing down a rigorous continuity argument involves discussions on some non-generic measure zero events, and we refer the reader to \cite[Section 6.4]{KY} for more details. 
\end{proof}

\section{Proof of Corollary \ref{main_cor}}\label{appd_cor213}
	For $\phi_n$ and $\psi_n$ in \eqref{phipsi}, we define the truncated matrices $\wt X$, $\wt Y$ and $\wt Z$ with entries  
	$$\wt x_{ij} :=x_{ij}\mathbf 1_{ |x_{ij}|\le \phi_n n^{\e}}  , \quad \wt y_{ij} :=y_{ij}\mathbf 1_{|y_{ij}| \le \phi_n n^\e}  ,\quad \wt z_{ij}:=z_{ij}\mathbf 1_{ |z_{ij}|\le \psi_n n^\e}, $$
	for a sufficiently small constant $\e>0$. 
	%$$ \quad \Omega :=\left\{\max_{i,j} |x_{ij}|\le \phi_n \log n,\max_{i,j} |y_{ij}|\phi_n \log n, \max_{i,j}|z_{ij}|\le \psi_n \log n \right\}.$$
	%where, for convenience, we again use $j$ to denote the column index of $X$. 
	Combining the moment conditions in (\ref{condition_4e}) with Markov's inequality, we obtain that
	\begin{equation}\label{XneX}
		\mathbb P(\wt X \ne X, \wt Y \ne Y, \wt Z\ne Z) =\OO \big( n^{-a\e}+  n^{-b\e} \big) 
	\end{equation}
	by a simple union bound. Using (\ref{condition_4e}) and integration by parts, we can also check that  
	\begin{align}\label{error_meanvar}
		\mathbb E  \left|x_{ij}\right|1_{|x_{ij}|> \phi_n n^\e} \le n^{-2-\e} , \quad \mathbb E \left|x_{ij}\right|^2 1_{|x_{ij}|> \phi_n n^\e} \le n^{-2-\e} .
	\end{align}
	For example, for the first estimate in \eqref{error_meanvar}, we have that
	\begin{align*}
		&  \mathbb E \left| \mathbf 1\left( |x_{ij}|>  \phi_n n^\e \right)x_{ij}\right| = \int_0^\infty \P\left( \left| \mathbf 1\left( |x_{ij}|>  \phi_n n^\e \right)x_{ij}\right|  > s\right)\dd s \\
		= &\int_0^{\phi_n n^\e}\P\left( |x_{ij}|> \phi_n n^\e \right)\dd s +\int_{\phi_n n^\e}^\infty \P\left(|x_{ij}| > s\right)\dd s  \\
		\lesssim &\int_0^{\phi_n n^\e}\left(   n^{1/2+\e}\phi_n   \right)^{-a}\dd s +\int_{\phi_n n^\e}^\infty \left(\sqrt{n}s\right)^{-a}\dd s \le n^{-\frac12 - 2\frac{a-1}{a} - (a-1)\e} \le n^{-2-\e},
	\end{align*}
	where in the third step we used (\ref{condition_4e}) %\eqref{assmAhigh} \HZ{check; I don't see which one} 
	and Markov's inequality, and in the last step we used $a>4$. The second estimate of \eqref{error_meanvar} can be proved in a similar way. Note that \eqref{error_meanvar} implies 
	$$|\mathbb E  \tilde x_{ij}| \le n^{-2-\e}, \quad  \mathbb E |\tilde x_{ij}|^2 = n^{-1} + \OO(n^{-2-\e}).$$
	%and
	%$$\left| \mathbb E\tilde x_{ij}^2\right| =O( n^{-2-\omega/2}), \ \ \text{if $x_{ij}$ is complex.} $$
	Moreover, we trivially have
	$$\mathbb E  |\tilde x_{ij}|^3 \le \mathbb E  |x_{ij}|^3 =\OO(n^{-3/2}),\quad \mathbb E  |\tilde x_{ij}|^4 \le \mathbb E  |x_{ij}|^4 =\OO(n^{-2}).$$
	Similar estimates also hold for the entries of $\wt Y$. Hence, $\wt X$ and $\wt Y$ are random matrices satisfying Assumption \ref{main_assm} (i) and condition \eqref{conditionA3}. For $\wt Z$, using (\ref{condition_4e}) and a similar argument, we can check that
	$$ |\mathbb E  \tilde z_{ij}| \le n^{-1-\e} , \quad  \mathbb E |\wt z_{ij}|^2 = n^{-1} + \OO (n^{-1-(b-2)\e} ). $$
	Hence, $Z$ is a random matrix satisfying Assumption \ref{main_assm} (ii). Now, combining \eqref{XneX} with Theorem \ref{main_thm}, we conclude \eqref{boundinp000}. % since $\e$ is arbitrary.
	%Then using a large deviate estimates (see e.g. Lemma \ref{largedeviation} below), we have %Moreover, by law of large numbers, there exists a sequence of $\psi_n \to 0$ such that
	%$$ \| \wt Z \wt Z^{\top} - I_{r}\| \prec n^{-\e} .$$
	Combining \eqref{XneX} with Theorem \ref{main_thm0.5}, we obtain that 
	$$|\wt\lambda_{r_+ + i } -\lambda_i^b| \prec n^{-1}\al_+^{-1} \le n^{-2/3-\e_0},\quad 1\le i \le k,$$
	for $\al_+$ satisfying \eqref{well-sep000}. 
	Together with Lemma \ref{lem null0}, it concludes \eqref{boundinTW}.
%\end{proof}

\section{Proof of Theorem \ref{main_thm1}}\label{sec mainthm3}
%In this section, we present the proof of Theorem \ref{main_thm1}
For the proof of Theorem \ref{main_thm1}, we adopt a similar argument as that for Theorem 2.7 in \cite{PartIII}, that is, we decompose $(X,Y,Z)$ (in distribution) into well-behaved random matrices $(X^s,Y^s,Z^s)$ with bounded support $ n^{-\e}$ plus a perturbation matrix. 
However, our setting here is more complicated. We now define the precise decomposition of $X$. 
First, we introduce a cutoff on the matrix entries of $X$ at the level $n^{-\epsilon}$ for a sufficiently small constant $\epsilon>0$:
\begin{equation*}
	\alpha^{(1)}_n:=\mathbb{P}\left( \vert \wh x_{11} \vert > n^{1/2-\epsilon}\right), \quad \beta^{(1)}_n:=\mathbb{E}\left[\mathbf{1}{\left( |\wh x_{11}|> n^{1/2-\epsilon}\right)}\wh x_{11} \right].
\end{equation*}
Using (\ref{tail_cond}), we can check with integration by parts that for any small constant $\delta>0$, %and large enough $n$,
\begin{equation}
	\alpha^{(1)}_n \leq \delta n^{-2+4\epsilon}, \quad \vert \beta^{(1)}_n \vert \leq \delta n^{-{3}/{2}+3 \epsilon} . \label{BBBOUNDS}
\end{equation}
Now, we define independent random variables $\wh x_{ij}^s$, $\wh x_{ij}^l$, $c^{(1)}_{ij}$, $1\le i \le p,\ 1\le j \le n$, as follows.
\begin{definition}\label{defn_cutoff_center_adddd}
	We define $\wh x_{ij}^s$ as a random variable that has law $\rho_s^{(1)}$ defined through %, which is defined such that
	\begin{equation}
		\rho_s^{(1)}(\Omega)= \frac1{1-\alpha^{(1)}_n}\int \mathbf 1\left( x+\frac{\beta^{(1)}_n}{1-\alpha^{(1)}_n} \in \Omega \right)\mathbf{1}\left(\left| x \right| \leq n^{{1}/{2}-\epsilon} \right)  \rho^{(1)}(\dd x) \nonumber%\label{DEF1}
	\end{equation}
	for any event $\Omega$, where $\rho^{(1)}(\dd x)$ is the law of $\wh x_{ij}$. 
	\iffalse
	Note that if $\wh x_{11}$ has density $\rho(x)$, then the density for $\wh x_{11}^s$ is 
	\begin{equation}
		\rho_s^{(1)}(x)= \mathbf{1}\left(\left| x-\frac{\beta^{(1)}_n}{1-\alpha^{(1)}_n}  \right| \leq n^{{1}/{2}-\epsilon} \right) \frac{\rho\left(x-\frac{\beta^{(1)}_n}{1-\alpha^{(1)}_n}\right)}{1-\alpha^{(1)}_n}. \nonumber%\label{DEF1}
	\end{equation}
	\fi
	We define $\wh x_{ij}^l$ as a random variable that has law $\rho_l^{(1)}$ defined through 
	%\begin{equation}
	%\cob \rho_l(x)= \mathbf{1}\left(\left| x-\frac{\beta_N}{1-\alpha_N}  \right| > n^{{1}/{2}-\epsilon} \right)\frac{\rho\left(x-\frac{\beta_N}{1-\alpha_N}\right)}{\alpha_N}; \nc \nonumber%\label{DEF2}
	%\end{equation}
	\begin{equation}
		\rho_l^{(1)}(\Omega)= \frac1{\alpha^{(1)}_n}\int \mathbf 1\left( x+\frac{\beta^{(1)}_n}{1-\al^{(1)}_n} \in\Omega \right)\mathbf{1}\left(\left| x \right| > n^{{1}/{2}-\epsilon} \right)  \rho^{(1)}(\dd x) \nonumber%\label{DEF1}
	\end{equation}
	for any event $\Omega$. We define $c^{(1)}_{ij}$ as a Bernoulli 0-1 random variable with 
	$$\mathbb{P}(c^{(1)}_{ij}=1)=\alpha^{(1)}_n,\quad \mathbb{P}(c^{(1)}_{ij}=0)=1-\alpha^{(1)}_n.$$ 
	Finally, let $X^s$, $X^l$ and $X^c$ be independent random matrices with entries
	$$x^s_{ij} = n^{-1/2}\wh x_{ij}^s, \quad x^l_{ij} = n^{-1/2}\wh x_{ij}^l,\quad x^c_{ij} = c^{(1)}_{ij}.$$
\end{definition}

\begin{remark}
%The purpose of the above decomposition (in distribution) is to write $(X,Y,Z)$ into well-behaved random matrices $(X^s,Y^s,Z^s)$ with bounded support $q=\OO(n^{-\e})$ plus a perturbation matrix. 
With the above definition, $X$ can be decomposed as $X^s + (X^l- X^s)\circ X^c$ in distribution up to a negligible deterministic matrix, see \eqref{T3} below. 
%We will show that $(X^l- X^s)\circ X^c$ leads to a negligible perturbation of the edge eigenvalues .
%For example, for $X$, the perturbation is of the form $(X^l- X^s)\circ X^c$ up to a negligible deterministic term. 
The matrix $X^c$ gives the locations of the nonzero entries and its rank is at most $n^{5\e}$ with high probability, see \eqref{LDP_B} below. The matrix $X^l$ contains the large entries above the cutoff, but the tail condition \eqref{tail_cond} guarantees that the sizes of these entries are of order $\oo(1)$ in probability, see \eqref{EneR}. Hence, the perturbation is of low rank and has small signal strengths. We expect that, as in the famous BBP transition \cite{BBP}, the effect of this weak perturbation on the largest few eigenvalues is negligible. 
\end{remark}

In Definition \ref{defn_cutoff_center_adddd}, $\rho_s^{(1)}$ and $\rho_l^{(1)}$ are defined in a way such that $\wh x_{ij}^s$ and $\wh x_{ij}^l$ are both centered random variables. We can easily check that
%It is easy to check that for independent $X^s$, $X^l$ and $X^c$,
\begin{equation}
	x_{ij} \stackrel{d}{=} x^s_{ij}\left(1-x^c_{ij}\right)+x^l_{ij}x^c_{ij} - \frac{1}{\sqrt{n}}\frac{\beta^{(1)}_n}{1-\alpha^{(1)}_n}, \label{T3}
\end{equation}
where ``$\stackrel{d}{=}$" means ``equal in distribution". Similarly, we decompose $Y$ as
\begin{equation}
	y_{ij} \stackrel{d}{=} y^s_{ij}\left(1-y^c_{ij}\right)+y^l_{ij}y^c_{ij} - \frac{1}{\sqrt{n}}\frac{\beta^{(2)}_n}{1-\alpha^{(2)}_n},  \label{T33}
\end{equation}
where the entries $y_{ij}^s$, $y_{ij}^l$ and $y^{c}_{ij}$ of the independent random matrices $Y^s$, $Y^l$ and $Y^c$ are defined in similar ways using
\begin{equation*}
	\alpha^{(2)}_n:=\mathbb{P}\left( \vert \wh y_{11} \vert > n^{1/2-\epsilon}\right), \ \ \beta^{(2)}_n:=\mathbb{E}\left[\mathbf{1}{\left( |\wh y_{11}|> n^{1/2-\epsilon}\right)}\wh y_{11} \right].
\end{equation*}
Notice that the deterministic matrix $\cal M_1$ with  entries
$$(\cal M_1)_{ij}=- \frac{1}{\sqrt{n}}\frac{\beta^{(1)}_n}{1-\alpha^{(1)}_n},\quad 1\le i \le p,\ 1\le j \le n,$$
has operator norm $\OO(n^{-1+3\e})$, which, by Weyl's inequality,  perturbs the singular values of $X$ at most by $\OO(n^{-1+3\e})$. Such a small error is always negligible for our proof, so we will omit the constant term in \eqref{T3}. Similarly, we will also omit the constant term in \eqref{T33}.  Finally, we decompose $Z$ as $Z= Z^s + Z^l $, where
\be\nonumber%\label{T333}
Z^s_{ij}= \mathbf 1(|Z_{ij}|\le n^{-\e})Z_{ij} + \beta^{(3)}_{n} ,\quad Z^l_{ij}= \mathbf 1(|Z_{ij}|> n^{-\e})Z_{ij}- \beta^{(3)}_{n} ,
\ee
for $\beta^{(3)}_n$ defined as
$$ \beta^{(3)}_n:=\mathbb{E}[ \mathbf 1(|Z_{ij}|> n^{-\e})Z_{ij}].$$
Using \eqref{assmZ} and integration by parts, one can check that
$ \beta_n^{(3)}  =\OO(n^{-1+\e}) . $
The deterministic vector $(\beta_n^{(3)},\cdots, \beta_n^{(3)})^\top\in \R^n$ has Euclidean norm $\OO(n^{-1/2+\e})$ and is also negligible for the following proof. Hence, for simplicity of notations, we will also omit it throughout the proof.
%where $X_2= (\mathbb{E}\vert q^l_{ij} \vert^2)^{-\frac{1}{2}}X^l$ and $X_{ij}^l=\frac{1}{\sqrt{N}} q_{ij}^l$.
%where by (\ref{BBBOUNDS}), we have
%\begin{equation*}
%\left\vert \frac{1}{\sqrt{N}}\frac{\beta_N}{1-\alpha_N} \right\vert \leq 2\delta n^{-2+3\epsilon}.
%\end{equation*}

%If we define the $n\times N$ matrix $Y=(Y_{ij})$ by 
%$$Y_{ij}=\frac{1}{\sqrt{N}}\frac{\beta_N}{1-\alpha_N}=\OO(\delta n^{-2+3\epsilon}), \quad 1\le i \le n,\ \ 1\le  j \le N,$$
%then we have $ \| Y \| =\OO (n^{-1+3\epsilon}) $. % depending on $\delta$ and $d$. 
%%Using the bound (\ref{boundH}), 
%In the proof below, one will see that (recall \eqref{Qtilde})
%$$ \left\| \Sig^{1/2} U^{*}(X+Y) V\wt \Sig^{1/2} \right\| = \lambda^{1/2}_1\left(\wt {\mathcal Q}_1(X+Y)\right) =\OO(1)$$
%with probability $1-\oo(1)$. %, where $\lambda_1$ denotes the largest eigenvalue of $\wt {\mathcal Q}_1$. 
%Thus with probability $1-\oo(1)$, we have 
%%it is easy to see that with $\xi_1$-high probability%from the representation (\ref{linearize_block}), it is easy to see that
%\begin{align}\label{const_err}
%\left|\lambda_1\left(\wt {\mathcal Q}_1(X+Y)\right) - \lambda_1\left(\wt {\mathcal Q}_1(X)\right)\right| = \OO\left( n^{-1+3\epsilon} \right).
%\end{align}
%Hence the deterministic part in (\ref{T3}) is negligible under the scaling $n^{2/3}$.

With (\ref{tail_cond}) and integration by parts, we can obtain that
\begin{align}\label{estimate_qs}
	\mathbb{E} \wh x^{s}_{11} =0, \quad \mathbb{E}\vert \wh x^s_{11}\vert^2=1 - \OO(n^{-1+2 \epsilon}), \quad \mathbb{E}\vert \wh x^s_{11} \vert^3=\OO(1),\quad \mathbb{E}\vert \wh x^s_{11} \vert^4=\OO(\log n).
\end{align}
Similar estimates hold for $\wh y^s_{11}$. Hence, $X_1:=(\mathbb{E}\vert \wh x^s_{11} \vert^2)^{-{1}/{2}}X^s$ and $Y_1:=(\mathbb{E}\vert \wh y^s_{11} \vert^2)^{-{1}/{2}}Y^s$ are random matrices that satisfy the assumptions for $X$ and $Y$ in Lemma \ref{lem null}, Theorem \ref{main_thm} and Theorem \ref{main_thm0.5} with $\phi_n= \OO(n^{-\e})$. Moreover, the small errors $\OO(n^{-1+2 \epsilon})$ in $\mathbb{E}\vert \wh x^s_{11} \vert^2$ and $\mathbb{E}\vert \wh y^s_{11} \vert^2$ are negligible for our purpose. For $Z$, 
%using \eqref{assmZ} and integration by parts, we get
%$$ \mathbb E  |Z^l_{ij}|  =\OO(n^{-1+\e}), \quad  |\mathbb E  Z^s_{ij}|=\OO(n^{-1+\e})  . $$
using $\lim_{t\to \infty}\mathbb E\left[|\wh z_{11}|^2\mathbf 1(|\wh z_{11}|>t) \right]=0$, we get that 
$$\mathbb E  |\wh z^s_{11}|^2=1-\oo(1),\quad \mathbb E  |\wh z^l_{11}|^2  =\oo(1),  $$
where we have used notations $\wh z_{11}^s:=\sqrt{n}Z^s_{11}$ and $\wh z_{11}^l:=\sqrt{n}Z^l_{11}$. Then, $Z_1:=(\mathbb{E}\vert \wh z^s_{11} \vert^2)^{-{1}/{2}}  Z^s $ satisfies the assumptions for $Z$ in Lemma \ref{lem null}, Theorem \ref{main_thm} and Theorem \ref{main_thm0.5} with $\psi_n= \OO(n^{-\e})$. Note that the scaling of $Z^s$ by $(\mathbb{E}\vert \wh z^s_{11} \vert^2)^{-{1}/{2}} $ amounts to a rescaling of $A$ and $B$ by $(\mathbb{E}\vert \wh z^s_{11} \vert^2)^{{1}/{2}}$, i.e., 
$$A\to A_1= (\mathbb{E}\vert \wh z^s_{11} \vert^2)^{{1}/{2}} A,\quad B\to B_1= (\mathbb{E}\vert \wh z^s_{11} \vert^2)^{{1}/{2}} B,$$ 
so that $A_1 Z_1= A Z^s$ and $B_1 Z_1= B Z^s$. In particular, we have that %the PCC matrix $\mathbf \Sigma_{\cal X\cal Y}$ and 
\be\label{perturb1}
\text{the $t_i$'s in (\ref{t1r}) are only perturbed by an amount of $\oo(1)$.}
\ee 

%Together with the estimate for $\mathbb{E}\vert q^s_{11}\vert^2$ in (\ref{estimate_qs}), 
Denote by \smash{$\cal C^s_{\cal X\cal Y}$ and $\cal C^s_{ XY}$} %and $\cal C^{b,s}_{X\cal Y}$} 
the SCC matrices obtained by replacing $(X,Y,Z)$ with $(X^s,Y^s,Z^s)$ in the corresponding definitions. Let $ \wt\lambda_i^s$ and $\lambda_i^s$ %and $\lambda_i^{b,s}$ 
be their eigenvalues. Then, by Theorem \ref{main_thm} and \eqref{perturb1}, for any $1\le i \le r_+$, we have that
\be\label{univ_outlier}
|\wt\lambda_i^s - \theta_i| =\oo(1) \quad \text{with high probability},
\ee
and by Lemma \ref{lem null0}, we have that for all $s_1\in \R$,
\be\label{univ_small0}
\begin{split}
	\lim_{n\to \infty}\mathbb{P}&\left(  n^{{2}/{3}}\frac{ \lambda^s_{1} - \lambda_+}{c_{TW}} \leq s_1 \right) = \lim_{n\to \infty} \mathbb{P}^{GOE}\left( n^{{2}/{3}}(\lambda_1 - 2) \leq s_1  \right).
\end{split}
\ee
Moreover, applying Theorem \ref{main_thm0.5} gives  \smash{$ | \wt\lambda^s_{1+r_+}  - \lambda_1^{s} | \prec n^{-1}$}.
%\eqref{boundstick}, we get \smash{$ | \wt\lambda^s_{1+r_+}  - \lambda_1^{b,s} | \prec n^{-1},$} and applying \eqref{boundstick} again gives $ |\lambda_1^{b,s}- \lambda_{1}^s| \prec n^{-1}.$ 
Combining it with \eqref{univ_small0}, we obtain that  
\be\label{univ_small}
\begin{split}
	\lim_{n\to \infty}\mathbb{P}&\left(  n^{{2}/{3}}\frac{ \wt\lambda^s_{1+r_+} - \lambda_+}{c_{TW}} \leq s_1 \right) = \lim_{n\to \infty} \mathbb{P}^{GOE}\left( n^{{2}/{3}}(\lambda_1 - 2) \leq s_1  \right).
\end{split}
\ee
(For simplicity, we only consider the largest non-outlier eigenvalue. The extension to the case with multiple non-outlier eigenvalues is straightforward.)
We write the right-hand sides of (\ref{T3}) and \eqref{T33} as
\begin{align*}
& x^s_{ij}\left(1-x^c_{ij}\right)+x^l_{ij}x^c_{ij}  = x^s_{ij} + \Delta^{(1)}_{ij} x^c_{ij}, \quad \Delta^{(1)}_{ij}:=x^l_{ij}-x^s_{ij},\\
%and %Similarly, we have
& y^s_{ij}\left(1-y^c_{ij}\right)+y^l_{ij}y^c_{ij}  = y^s_{ij} + \Delta^{(2)}_{ij} y^c_{ij}, \quad \Delta^{(2)}_{ij}:=y^l_{ij}-y^s_{ij}.
\end{align*}
%where $ R_{ij}:=X^l_{ij}-X^s_{ij}.$ 
We define matrices 
$$\cal E^{(1)} :=(\Delta^{(1)}_{ij} x^c_{ij}: 1\le i \le p, 1\le j \le n),\quad \cal E^{(2)} :=(\Delta^{(2)}_{ij} y^c_{ij}: 1\le i \le q, 1\le j \le n).$$ 
It suffices to show that the effect of $\cal E^{(1)}$, $\cal E^{(2)}$ and $Z^l$ on $\wt \lambda_i$, $1\le i \le r_+ $ and $\wt\lambda_{r_+ +1}$ is negligible. %Note that $X^c_{ij}$ is independent of $X^s_{ij}$ and $R_{ij}$. 

Define the event
\begin{align*}
	\mathscr A:=& \left\{\#\{(i,j):x^c_{ij}=1\}\leq n^{5\epsilon}\right\} \cap \left\{x^c_{ij}=x^c_{kl}=1  {\Rightarrow} \{i,j\}=\{k,l\} \ \text{or} \  \{i,j\} \cap \{k,l\}=\emptyset \right\}.
\end{align*}
%If we regard the matrix $X^c$ as a sequence $\mathbf X^c$ of $nN$ {\it{i.i.d.}} Bernoulli random variables, it is easy to obtain from the 
By a Chernoff bound, we get that
\begin{equation}\label{LDP_B}
	\mathbb{P}\left(\left\{\#\{(i,j):x^c_{ij}=1\}\leq n^{5\epsilon}\right\} \right) \geq 1- \exp(- n^{\epsilon}).
\end{equation}
%for sufficiently large $n$. 
If the number $n_0$ of the nonzero elements in $X^c$ satisfies $n_0 \le n^{5\epsilon} $, then we can check that
\begin{align}
&	\mathbb P\left(\exists \, i = k, j\ne l \ \text{or} \  i\ne k, j =l \text{ so that } x^c_{ij}=x^c_{kl}=1 \left| \#\{(i,j):x^c_{ij}=1\} = n_0 \right. \right)  = \OO(n_0^2/n). \label{LDP_C}
\end{align}
Combining the estimates (\ref{LDP_B}) and (\ref{LDP_C}), we get that 
\begin{equation}\label{prob_A}
	\mathbb P(\mathscr A) \ge 1 -  \OO(n^{-1+10\epsilon}).
\end{equation}
% and hence
%\begin{equation}
%\mathbb P(X^c = \wt X^c) \ge 1 - \mathcal \OO(n^{-1+10\epsilon}).
%\end{equation}
%By independence, we can put the matrix $C_{ij}$ into a vector  $C^1 \in \mathbb{R}^{MN}$, then we have
%\begin{equation*}
%\mathbb{P}(\{\text{Rank(C)} \geq n^{5 \epsilon}\}) \leq \mathbb{P}(\sum_{i=1}^{MN} C^1_i \geq n^{5 \epsilon}) \leq \exp(-n^2 I(n^{5\epsilon-2}))
%\end{equation*}
%where $I(x)=x\log x +(1-x)\log (1-x)+ \log 2$ and it is easy to check that, when $N$ is sufficiently large, $I(n^{5\epsilon-2})>0$.
%By definition, we have
%\begin{equation*}
%\mathbb{P}(A) \geq 1- \mathbb{P}(\{\text{Rank(C)} \geq n^{5 \epsilon}\})
%\end{equation*}
%Hence if we introduce the matrix 
%$$\wt{\cal E}^{(1)} = \mathbf 1\left(\mathscr A\cap \left\{\max_{i,j}|\cal E^{(1)}_{ij}|  \leq \omega \right\}\right) \cal E^{(1)} ,$$
%then we have 
%\begin{equation}\label{EneR}
%\mathbb P\left( \wt{\cal E}^{(1)} = {\cal E}^{(1)}\right)=1-\oo(1)
%\end{equation}
%by (\ref{prob_A}) and (\ref{prob_R}). 
Similarly, for the event
$$ \mathscr B:= \left\{\#\{(i,j):y^c_{ij}=1\}\leq n^{5\epsilon}\right\} \cap \left\{y^c_{ij}=y^c_{kl}=1  {\Rightarrow} \{i,j\}=\{k,l\} \ \text{or} \  \{i,j\} \cap \{k,l\}=\emptyset \right\},$$
we have 
\begin{equation}\label{prob_B}
	\mathbb P(\mathscr B) \ge 1 -  \OO(n^{-1+10\epsilon}).
\end{equation}
%if the nonzero elements in $Y^c$ is at most $n^{5\epsilon} $. 
On the other hand, using condition (\ref{tail_cond}) and Markov's inequality, we get
\begin{equation}\nonumber
	\mathbb{P}\left(|\cal E^{(1)}_{ij}| \geq \omega \right) + \mathbb{P}\left(|\cal E^{(2)}_{ij}| \geq \omega \right) \leq \mathbb{P}\left(|\wh x_{ij}| \geq \frac{\omega}{2}n^{1/2}\right) + \mathbb{P}\left(|\wh y_{ij}| \geq \frac{\omega}{2}n^{1/2}\right) = \oo(n^{-2}) ,
\end{equation}
for any fixed constant $\omega>0$. With a simple union bound, we get
\begin{equation}\label{EneR}
	\mathbb P\left(\max_{i,j} |\cal E^{(1)}_{ij}| \geq \omega \right) + \mathbb P\left(\max_{i,j} |\cal E^{(2)}_{ij}| \geq \omega \right)= \oo(1).
\end{equation}
Define the event
$$\mathscr C_1 := \left\{\max_{i,j}|\cal E^{(1)}_{ij}| \leq \omega\right\}\cap \left\{ \max_{i,j}|\cal E^{(2)}_{ij}| \leq \omega\right\}.$$
Combining (\ref{prob_A}), (\ref{prob_B}) and \eqref{EneR}, we get
\be\label{prob_R}\mathbb P( \mathscr A \cap  \mathscr B\cap  \mathscr C_1)=1-\oo(1) .\ee
%and introduce the matrix 
%$$\wt{\cal E}^{(1)} = \mathbf 1\left(\mathscr B\cap \left\{\max_{i,j}|\cal E^{(2)}_{ij}|  \leq \omega \right\}\right) \cal E^{(2)} .$$
%Then as in \eqref{EneR}, we have 
%\begin{equation}\label{EneR2}
%\mathbb P\left( \wt{\cal E}^{(2)} = {\cal E}^{(2)}\right)=1-\oo(1).
%\end{equation}
We also define the event 
\be\label{defC2} \mathscr C_2:=\left\{\|(Z^s)^{\top} Z^s - I_r\| \le w, \|(Z^l)^{\top} Z^l \| \le w^2, \|(Z^s)^{\top} Z^l \| \le w\right\}.\ee
By the law of large numbers, we have $\mathbb P(\mathscr C_2)=1-\oo(1).$ 

Recalling \eqref{deteq}, we only need to study the zeros of  $\det [\wt H_1 (\lambda)]$ on event $\mathscr A \cap  \mathscr B \cap \mathscr C_1\cap \mathscr C_2$. Here, we define $\wt H_t (\lambda)$, $t\in [0,1]$, as
\be\nonumber
\begin{split} 
	\wt H_t (\lambda): = \wt H^s(\lambda)&+ t \begin{bmatrix} 0 & \begin{pmatrix} {\cal E}^{(1)}+AZ^l  & 0\\ 0 &  {\cal E}^{(2)} + BZ^l  \end{pmatrix}\\ \begin{pmatrix} ( {\cal E}^{(1)}+AZ^l)^{\top} & 0\\ 0 &  ( {\cal E}^{(2)}+B Z^l)^{\top}\end{pmatrix}  & 0\end{bmatrix} ,
\end{split}
\ee
%\be
%\begin{split} 
%\wt H_t (\lambda): = H^s(\lambda)&+ \begin{pmatrix} 0 & \begin{pmatrix} A(Z^s+tZ^l) & 0\\ 0 &  B(Z^s+tZ^l)  \end{pmatrix}\\ \begin{pmatrix} ( A(Z^s+tZ^l))^{\top} & 0\\ 0 &  ( B(Z^s+tZ^l))^{\top}\end{pmatrix}  & 0\end{pmatrix}  \\
%&+ t\begin{pmatrix} 0 & \begin{pmatrix} {\cal E}^{(1)}  & 0\\ 0 &  {\cal E}^{(2)}  \end{pmatrix}\\ \begin{pmatrix} ({\cal E}^{(1)} )^{\top} & 0\\ 0 &  ({\cal E}^{(2)}  )^{\top}\end{pmatrix}  & 0\end{pmatrix} 
%\end{split}
%\ee
where 
$$\wt H^s(\lambda):=H^s(\lambda) + \begin{bmatrix} 0 & \begin{pmatrix} AZ^s  & 0\\ 0 &  BZ^s   \end{pmatrix}\\ \begin{pmatrix} ( AZ^s )^{\top} & 0\\ 0 &  ( BZ^s )^{\top}\end{pmatrix}  & 0\end{bmatrix},$$ 
with
$$H^s(\lambda):=\begin{bmatrix} 0 & \begin{pmatrix}  X^s  & 0\\ 0 &  Y^s \end{pmatrix}\\ \begin{pmatrix} (X^s)^{\top} & 0\\ 0 &  (Y^s)^{\top}\end{pmatrix}  & \begin{pmatrix}  \lambda  I_n & \lambda^{1/2}I_n\\ \lambda^{1/2} I_n &  \lambda I_n\end{pmatrix}^{-1}\end{bmatrix}.$$
We would like to extend \eqref{univ_outlier} and \eqref{univ_small} at $t=0$ all the way to $t=1$ using a continuity argument. Correspondingly, for any $t\in[0,1]$, we define the PCC matrix $\cal C_{\cal X\cal Y}(t)$ for $\cal X(t):= X^s + t \wt {\cal E}^{(1)} + A(Z^s+tZ^l) $ and $\cal Y(t):= Y^s + t \wt {\cal E}^{(2)} + B(Z^s+tZ^l) $, and we denote its eigenvalues by $\wt \lambda_i(t)$. Note that $\wt\lambda_i=\wt\lambda_i(1)$ are the eigenvalues we are interested in, and the eigenvalues $\wt\lambda_i^s=\wt\lambda_i(0)$ satisfy \eqref{univ_outlier} and \eqref{univ_small}. Moreover, $\wt\lambda_i(t)$ is continuous with respect to $t$ on the extended real line $ \overline{\mathbb R}$.

\begin{proof}[Proof of \eqref{boundinp}]
	For any $1\le i \le r_+$, we pick a sufficiently small constant $\delta>0$ such that the following properties hold for large enough $n$: (i) the interval $J_i:=[\theta_i -\delta,\theta_i + \delta]$ only contains $\theta_j$'s that converge to the same limit as $\theta_i$ when $n\to \infty$, (ii) $J_i$ is away from all the other $\theta_j$'s at least by $\delta$, and (iii) $J_i$ is away from $\lambda_+$ at least by $\delta$. By \eqref{univ_outlier}, we know \smash{$\wt\lambda_i(0) \in J_i$} with high probability. Now, for $\mu:= \theta_i \pm \delta$, we claim that
	\begin{equation}\label{eq_claim}
		{\mathbb P\left(  \det \wt H_t( \mu) \ne 0 \text{ for all }0\le t \le 1 \right) = 1-\oo(1).} 
		\end{equation}
		%\begin{equation}\label{eq_claim}
		%\mathbb P\left(\left|\lambda_1^s - \lambda_1^E\right| \le n^{-3/4}\right) = 1-\oo(1), \quad \lambda_1^s:=\lambda_1\left(\wt {\mathcal Q}_1(X^s) \right),\quad  \lambda_1^E:=\lambda_1\left(\wt {\mathcal Q}_1(X^s+E) \right).
		%\end{equation}
		If (\ref{eq_claim}) holds, then $\mu$ is not an eigenvalue of $\cal C_{\cal X\cal Y}(t)$ for all $t\in [0,1]$ with probability $1-\oo(1)$. By continuity of $\wt\lambda_i(t)$ with respect to $t$, we have $\wt\lambda_i = \wt\lambda_i(1)\in J_i$ with probability $1-\oo(1)$, that is,
		$$\mathbb P( |\wt\lambda_i - \theta_i| \le \delta)=1-\oo(1).$$
		This concludes \eqref{boundinp} since $\delta$ can be arbitrarily small.
		
		For the proof of (\ref{eq_claim}), we will condition on $\mathscr A \cap  \mathscr B$ and the event $\mathscr C_{n_x n_y}$ that $X^c$ and $Y^c$ have $n_x$ and $n_y$ nonzero entries with $\max\{n_x,n_y\}\le n^{5\e}$. Moreover, we assume that the indices of the $n_x$ nonzero entries of $X^c$ are $(\sigma_x(1),\pi_x(1)), \ldots, (\sigma_x(n_x),\pi_x(n_x))$, and the indices of the $n_y$ nonzero entries of $Y^c$ are $(\sigma_y(1),\pi_y(1)),  \ldots, (\sigma_y(n_y),\pi_y(n_x))$. Here, $\sigma_x:\{1,\ldots , n_x\}\to \{1,\ldots, p\}$, $\pi_x: \{1,\ldots, n_x\}\to \{1,\ldots, n\}$, $\sigma_y:\{1,\ldots , n_y\}\to \{1,\ldots, q\}$ and $\pi_y: \{1,\ldots, n_y\}\to \{1,\ldots, n\}$ are %uniformly random 
		all injective functions. 
		%which correspond to the following entries of $E$: \nc
		%\begin{equation}
		%e_{11}, \ e_{22}, \ \cdots,  \ e_{n_0n_0}, \ \ n_0 \le n^{5\epsilon}. \label{STRUCTURE}
		%\end{equation}
		%For the other choices of the positions of nonzero entries, the proof is exactly the same, but we make this assumption to simplify the notations. 
		%By the definition of $E$, we have $|e_{ii}| \le \omega$, $1\le i \le n_0$. 
		Then, we can rewrite that 
		$$\wt H_t(\mu)= H^s(\mu) +O_t\begin{bmatrix} 0 & \begin{pmatrix} \cal D  & 0\\ 0 &  t\cal D_e   \end{pmatrix}\\ \begin{pmatrix} \cal D & 0\\ 0 &  t \cal D_e\end{pmatrix}  & 0\end{bmatrix} O_t^{\top}, \quad O_t:=  \begin{bmatrix}\begin{pmatrix} \bU  ,\mathbf F_1   \end{pmatrix} & 0  \\ 0 & \begin{pmatrix} \bE_t ,  \mathbf F_2\end{pmatrix}  \end{bmatrix}  ,$$ 
		where $\cal D$ and $\bU$ have been defined in \eqref{defncalD} and \eqref{defnUE};  $\cal D_e: =\begin{pmatrix} \Sigma_e^{(1)} & 0 \\ 0 & \Sigma_e^{(2)} \end{pmatrix}$ with
		$$ \Sigma_e^{(1)}:= \diag \left( \cal E^{(1)}_{\sigma_x(1)\pi_x(1)}, \cdots, \cal E^{(1)}_{\sigma_x(n_x)\pi_x(n_x)}\right), \ \ \Sigma_e^{(2)}:= \diag \left( \cal E^{(2)}_{\sigma_y(1)\pi_y(1)}, \cdots, \cal E^{(2)}_{\sigma_y(n_y)\pi_y(n_y)}\right);$$
		$$\bE_t:= \begin{bmatrix} \begin{pmatrix}Z_t^{\top}\mathbf v_1^a, \cdots,Z_t^{\top} \mathbf v_r^a \end{pmatrix} & 0 \\ 0 & \begin{pmatrix}Z_t^{\top}\mathbf v_1^b, \cdots, Z_t^{\top}\mathbf v_r^b \end{pmatrix} \end{bmatrix}, \quad \text{with}\quad Z_t:= Z^s + tZ^l;$$
		$$\mathbf F_1 := \begin{bmatrix} \begin{pmatrix}\mathbf e_{\sigma_x(1)}^{(p)}, \cdots, \mathbf e_{\sigma_x(n_x)}^{(p)} \end{pmatrix} & 0 \\ 0 & \begin{pmatrix}\mathbf e_{\sigma_y(1)}^{(q)}, \cdots, \mathbf e_{\sigma_y(n_y)}^{(q)} \end{pmatrix} \end{bmatrix}; $$ $$\mathbf F_2 := \begin{bmatrix} \begin{pmatrix}\mathbf e_{\pi_x(1)}^{(n)}, \cdots, \mathbf e_{\pi_x(n_x)}^{(n)} \end{pmatrix} & 0 \\ 0 & \begin{pmatrix}\mathbf e_{\pi_y(1)}^{(n)}, \cdots, \mathbf e_{\pi_y(n_y)}^{(n)}\end{pmatrix} \end{bmatrix}.$$ 
		Here, we used $\mathbf e_i^{(l)}$ to denote  the standard unit vector along $i$-th coordinate in $\R^l$. 
		
		Applying the identity $\det(1+\cal A\cal B)=\det(1+\cal B\cal A)$, we obtain that  
		\be\label{nece extra1}
		\det  \wt H_t(\mu)  =\det\left[  G^s(\mu) \right]\cdot\det\left[1+\wt F_t(\mu) + \cal E_t(\mu) \right],
		\ee
		where 
		\begin{align*} 
		&\wt F_t(\mu):=\begin{bmatrix} 0 & \begin{pmatrix} \cal D  & 0\\ 0 & t \cal D_e   \end{pmatrix}\\ \begin{pmatrix} \cal D & 0\\ 0 & t \cal D_e\end{pmatrix}  & 0\end{bmatrix}  O_t^{\top} \Pi(\mu) O_t, \\
		%and$$ 
		&\cal E_t(\mu):=\begin{bmatrix} 0 & \begin{pmatrix} \cal D  & 0\\ 0 & t \cal D_e   \end{pmatrix}\\ \begin{pmatrix} \cal D & 0\\ 0 & t \cal D_e\end{pmatrix}  & 0\end{bmatrix}  O_t^{\top} \left[G^s(\mu)-\Pi(\mu)\right] O_t.
		\end{align*}
		%We write $G^s(\mu)=\left[G^s(\mu)-\Pi(\mu)\right] + \Pi(\mu)$. 
		%Note that 
		Since $O_t$ is deterministic conditioning on $Z$, by Lemma \ref{thm_largebound}, we have that (recall \eqref{defC2})
		$$\mathbb E \left[ \left. \left| \left[O_t^{\top} \left( G^s(\mu) - \Pi(\mu)\right) O_t\right]_{ij}  \right|^2 \right| \mathscr C_{n_x n_y}, Z, \mathscr C_2\right]\prec  n^{-1}, \quad 1\le i,j \le 2r + n_x + n_y.$$
		%where we used (\ref{Immc}) and \eqref{eq_defpsi} in the second step. 
		Applying Markov's inequality to this estimate and using a simple union bound, we get that
		\begin{equation}\label{BOUND2000}
			 \max_{1\le i , j \le 2r+n_x+n_y} \left| \left[O_t^{\top} \left( G^s(\mu) - \Pi(\mu)\right) O_t\right]_{ij} \right| \le n^{-1/4}
		\end{equation}
	 with probability $1- \OO(n^{-1/2+11\e})$ conditioning on $\mathscr C_{n_x n_y}$, $Z$ and  $\mathscr C_2$. Next, we claim that on $\mathscr C_1\cap\mathscr C_2$,   
		\be\label{bound OPO}
		\sup_{0\le t \le 1}\left\| \wt F_t(\mu)- \wt F_0(\mu) \right\| \le C \omega ,
		\ee
		for some constant $C>0$ that does not depend on $\omega$. In fact, expanding $\wt F_t(\mu)$ and using $\|\Pi(\mu)\|=\OO(1)$, $\|t\Sigma_e^{(1)}\|\le \omega$, $\|t\Sigma_e^{(2)}\|\le \omega$ and $\|\bE_t - \bE_0\|=\OO(\omega)$ on $\mathscr C_1\cap \mathscr C_2$, we can easily obtain \eqref{bound OPO}. Then, combining \eqref{BOUND2000} and \eqref{bound OPO}, we get that on the event $\mathscr A\cap \mathscr B \cap \mathscr C_1\cap \mathscr C_2$,
		\be\label{Gz part000}
		\det\left(1+\wt F_t(\mu) +\cal E_t(\mu) \right) = \det\left(1+\wt F_0(\mu)   + \OO(\omega)\right) \ \ \ \text{for all } \ \ t\in[0,1],
		\ee
		with probability $1-\oo(1)$. When $t=0$, the discussion at the beginning of Section \ref{sec mainthm} (i.e. the argument leading to \eqref{masterx4}) gives that at $\mu=\theta_i\pm \delta$, $\|(1+\wt F_0(\mu) )^{-1}\|\le C_\delta$ with high probability for some constant $C_\delta>0$. Thus, by \eqref{Gz part000}, as long as $\omega$ is sufficiently small, we have that  with probability $1-\oo(1)$, $\det(1+\wt F_t(\mu) + \cal E_t(\mu)) \ne 0$ for all $t\in[0,1]$. This concludes (\ref{eq_claim}), which further gives \eqref{boundinp}.
		\end{proof}
	
	\begin{proof}[Proof of \eqref{boundinTW} for Theorem \ref{main_thm1}]
		Similar to \eqref{eq_claim}, we claim that 
		\begin{equation}\label{suff_claim}
			{\mathbb P\left(  \det \wt H_t( \mu) \ne 0 \text{ for all }0\le t \le 1 \right) = 1-\oo(1)},
			\end{equation}
		for $\mu:= \lambda_{1}(0) \pm n^{-3/4} = \lambda_{1}^s \pm n^{-3/4}.$  
			At $t=0$, Theorem \ref{main_thm0.5} gives that  
			%we have $ | \wt\lambda_{1+r_+} (0) - \lambda_1^b(0) | \prec n^{-1}.$ Applying Theorem \ref{main_thm0.5} again gives $ |\lambda_1^b(0)- \lambda_{1} (0) | \prec n^{-1}.$ Thus we have that 
			$$\wt\lambda_{1+r_+}(0) \in [\lambda^s_{1} - n^{-3/4},\lambda^s_{1} + n^{-3/4}] \quad \text{with high probability.}$$ 
			If \eqref{suff_claim} holds, then due to the continuity of $\wt\lambda_{1+r_+}(t)$ with respect to $t$, we have that
			$$ \wt\lambda_{1+r_+}\equiv \wt\lambda_{1+r_+}(1) \in [ \lambda^s_{ 1} - n^{-3/4}, \lambda^s_{1} + n^{-3/4}]\quad \text{with probability $1-\oo(1)$},$$
			 which concludes \eqref{boundinTW} for $k=1$ together with \eqref{univ_small0}.

			In the following proof, we choose $z=\lambda_+ + \ii n^{-{2}/{3}}$. As in \eqref{nece extra1}, 
			we need to study 
			$$\det\left\{1+\wt F_t(z) + \cal E_t(z) + \begin{bmatrix} 0 & \begin{pmatrix} \cal D  & 0\\ 0 & t \cal D_e   \end{pmatrix}\\ \begin{pmatrix} \cal D & 0\\ 0 & t \cal D_e\end{pmatrix}  & 0\end{bmatrix} O_t^{\top} \left[G^s(\mu) - G^s(z)\right] O_t \right\},$$
			where we used the simple identity 
			$$O_t^{\top} G^s(\mu) O_t = O_t^{\top} \left[G^s(\mu) - G^s(z)\right] O_t + O_t^{\top} G^s(z) O_t  .$$ 
			Repeating the proof below \eqref{nece extra1}, we can show that with probability $1-\oo(1)$,
			\be\label{finally3} 1+\wt F_t(z)+\cal E_t(z)   = 1+\wt F_0(z)   + \OO(\omega) \quad \text{for all }\ \ t\in [0,1],\ee
			and $\|(1+\wt F_0(z) )^{-1}\|\le C$ with high probability for some constant $C>0$ that is independent of $\omega$. Moreover, we have that 
			\be\label{final claim4}
			\|O_t^{\top} \left[G^s(\mu) - G^s(z)\right] O_t \| \le n^{-1/6} \ \ \ \text{with probability $1-\oo(1)$},
			\ee
			which is proved as (5.16) in \cite{PartIII}.  %If \eqref{final claim4} holds, together 
			Combining \eqref{finally3} and \eqref{final claim4}, we get that with probability $1-\oo(1)$, 
			$$\det\left(1+\wt F_t(\mu) +\cal E_t(\mu) \right) = \det\left(1+ \wt F_0(z)   + \OO(\omega) \right)\ne 0 \ \ \ \text{for all} \ \ t\in[0,1],$$
			as long as $\omega$ is sufficiently small. This concludes \eqref{suff_claim}, which completes the proof of \eqref{boundinTW} for the $k=1$ case. It is easy to extend the above proof to the $k>1$ case, and we omit the details.
		\end{proof}
		
		%\appendix
		\section{Proof of Lemma \ref{lem null}   }\label{sec pflemma}
		
		Finally, in this section, we present the proof of Lemma \ref{lem null}. It has been proved in \cite{PartIII} for the $B=0$ case, and we need to show that adding the $BZ$ term to $Y$ does not affect the results. 
		\iffalse
		By Theorem 2.5 of \cite{PartIII}, \eqref{joint TW} holds for $\lambda_i$, the eigenvalues of $\cal C_{XY}$. On the other hand, by Theorem \ref{main_thm0.5} we have
		$$|\lambda_i^b - \lambda_i | \prec n^{-1}\al_+^{-1} \lesssim n^{-1},$$
		where in the second step we used that $t_i=0$ for all $1\le i \le r$ and hence $\al_+=t_c \sim 1$. This shows that \eqref{joint TW} also holds for $\lambda_i^b$.
		\fi
		We remark that, since \eqref{rigidityb} has been used in the proof of \eqref{boundstick}, we cannot apply Theorem \ref{main_thm0.5} and \eqref{rigidity} to conclude Lemma \ref{lem null}. A separate argument is needed. We first prove an averaged local law for $G^b(z)$ as in \eqref{aver_in} and \eqref{aver_out0}, using the following resolvent estimates.
		
		\begin{lemma}[Lemma 3.8 of \cite{PartIII}]\label{lem_gbound0}
			%For any deterministic unit vector $\bv_\al \in \C^{\mathcal I_\al}$, $\al=1,2$, we have %for $\beta=1,2,3,4$,
			%\be 
			%\sum_{\fa\in \cal I} |G_{\fa\mathbf v_\al}(z)|^2 = \sum_{\fa\in \cal I} | G_{\mathbf v_\al \fa}(z)|^2 \prec  |G_{\mathbf v_\al \mathbf v_\al}(z)|+\frac{ \im G_{\mathbf v_\al \mathbf v_\al}(z) }{\eta},\quad z= E+ \ii\eta .\label{eq_sgsq1}
			%\ee
			%Fix any $z= E+ \ii\eta\in \wt S(\e,\wt \e)$.  
			For any deterministic unit vectors $\bv_\beta \in \C^{\mathcal I_\beta}$, $\beta=3,4$, we have that %for $\al=1,2,3,4$,
			\be 
			\begin{split}
				\sum_{\fa \in \mathcal I } {\left| {G_{\fa \mathbf v_\beta } } \right|^2 }  \prec 1 +  \frac{\left|\im (\cal U \cal G_R)_{\bv_\beta\bv_\beta}\right|}{\eta} , \quad   \sum_{\fa \in \mathcal I} {\left| {G_{\mathbf v_\beta  \fa } } \right|^2 }  \prec 1  + \frac{\left|\im ( \cal G_R\cal U^{\top})_{\bv_\beta\bv_\beta}\right|}{\eta} , \label{eq_sgsq2} 
			\end{split}
			\ee
			where
			$$ \cal U := z^{1/2} \begin{pmatrix}  \overline z  I_n &  \overline z^{1/2}I_n\\  \overline z^{1/2}I_n &   \overline z  I_n\end{pmatrix}    \begin{pmatrix}  z  I_n & z^{1/2}I_n\\ z^{1/2}I_n &  z  I_n\end{pmatrix}^{-1}  . $$
			%and $\mu\in \mathcal I_2$, we define $\mathbf u_i=U^* \mathbf e_i  \in \mathbb C^{\mathcal I_1}$ and $\mathbf v_\mu=V^* \mathbf e_\mu  \in \mathbb C^{\mathcal I_2}$, i.e. $\mathbf u_i$ is the $i$-th row vector of $U$ and $\mathbf v_\mu$ is the $\mu$-th row vector of $V$. Let $\mathbf x \in \mathbb C^{\mathcal I_1}$ and $\mathbf y \in \mathbb C^{\mathcal I_2}$. Then we have %for some constant $C>0$,
			%  \begin{align}
			% & \sum_{i \in \mathcal I_1 }  \left| {G_{\mathbf x \mathbf u_i} } \right|^2  =\sum_{i \in \mathcal I_1 }  \left| {G_{ \mathbf u_i \mathbf x} } \right|^2  = \frac{|z|^2}{\eta}\im\left(\frac{ G_{\mathbf x\mathbf x}}{z}\right) , \label{eq_sgsq2} \\
			%& \sum_{\mu  \in \mathcal I_2 } {\left| {G_{\mathbf y \mathbf v_\mu } } \right|^2 }=\sum_{\mu  \in \mathcal I_2 } {\left| {G_{\mathbf v_\mu \mathbf y } } \right|^2 }  = \frac{{\im G_{\mathbf y\mathbf y} }}{\eta }, \label{eq_sgsq1}\\ 
			%& \sum_{i \in \mathcal I_1 } {\left| {G_{\mathbf y \mathbf u_i} } \right|^2 } =\sum_{i \in \mathcal I_1 } {\left| {G_{ \mathbf u_i \mathbf y} } \right|^2 } = {G}_{\mathbf y\mathbf y}  +\frac{\overline z}{\eta} \im G_{\mathbf y\mathbf y}  , \label{eq_sgsq3} \\
			%& \sum_{\mu \in \mathcal I_2 } {\left| {G_{\mathbf x \mathbf v_\mu} } \right|^2 }= \sum_{\mu \in \mathcal I_2 } {\left| {G_{\mathbf v_\mu \mathbf x } } \right|^2 }= \frac{G_{\mathbf x\mathbf x}}{z}  + \frac{\overline z}{\eta} \im \left(\frac{G_{\mathbf x\mathbf x}}{z}\right) .\label{eq_sgsq4}
			% \end{align}
			%All of the above estimates remain true for $G^{(\mathbb T)}$ instead of $G$ for any $\mathbb T \subseteq \mathcal I$. 
		\end{lemma}
		We calculate $m_3^b(z)=n^{-1}\sum_{\mu\in \cal I_3}G^b_{\mu\mu}(z)$ using \eqref{woodbury2}.  By the anisotropic local law \eqref{aniso_lawb}, 
		we have that with high probability,
		$$\left\| \left[1 +  \begin{pmatrix} 0 & \cal D_b\\ \cal D_b  & 0\end{pmatrix}\begin{pmatrix} \bU_b^{\top} & 0 \\ 0 & \bE_b^{\top}\end{pmatrix}G(z)\begin{pmatrix} \bU_b & 0 \\ 0 & \bE_b\end{pmatrix}  \right]^{-1}  \begin{pmatrix} 0 & \cal D_b\\ \cal D_b  & 0\end{pmatrix} \right\| =\OO(1).$$
		Hence, using \eqref{woodbury2}, we obtain that %(recall \eqref{def mb})
		$$ |m_3^b(z) - m_3(z)| \prec \frac1n\max_{1\le k \le r}\sum_{\mu\in \cal I_3} \left(| G_{\mu \mathbf u_k^b}(z) |^2 +  | G_{\mu \wt\bv_k^b}(z) |^2\right),$$
		where we have abbreviated that $\wt \bv_k^b :=Z^{\top}\mathbf v_k^b$. Note that $\wt\bv_k^b$ are approximately orthonormal vectors by \eqref{eq_iso}. Then, using \eqref{eq_sgsq2}, we obtain that for $z\in \wt S(\e,\wt\e)$,
		\begin{align} 
			|m_3^b(z) - m_3(z)| & \prec \frac{1}{n}+ \max_{1\le k \le r} \frac{|\im \left( \cal U\cal G_R\right)_{\mathbf u_k^b\mathbf u_k^b} |+| \im \left( \cal U\cal G_R \right)_{\wt{\mathbf v}_k^b\wt{\mathbf v}_k^b}| }{n\eta} \nonumber\\
			& \prec \frac{1}{n}+ \max_{1\le k \le r} \frac{\eta + \im m_c(z) + \Psi(z) + \psi_n + \phi_n}{n\eta} \lesssim \Psi^2(z) + \frac{ \psi_n + \phi_n}{n\eta},\label{m3b-m3}
		\end{align}
		where in the second step we used  the local law \eqref{aniso_lawb} and that 
		$$\left|\im \left( \cal U\Pi_R^b(z)\right)_{\mathbf u_k^b\mathbf u_k^b} \right|+\left| \im \left( \cal U  \Pi^b_R\right)_{\wt{\mathbf v}_k^b\wt{\mathbf v}_k^b}\right| \lesssim \im m_c(z) +\eta.$$
		Here, $\Pi_R^b(z)$ denotes the $(\cal I_3\cup \cal I_4)\times (\cal I_3\cup \cal I_4)$ block of $\Pi^b$. 
		Combining \eqref{m3b-m3} with the averaged local laws \eqref{aver_in}--\eqref{aver_out0} for $m_3(z)$ and equation \eqref{m3m} for $m_3^b(z)$ and $m^b(z)$, we obtain the following local laws: for any fixed $\epsilon,\wt \e>0$, 
		\begin{equation}
			\vert m^b(z)-m_{c}(z) \vert \prec (n \eta)^{-1} \label{aver_inbbbb}
		\end{equation}
		uniformly in $z \in \wt S(\epsilon,\wt \e)$, and
		\begin{equation}\label{aver_outbbbb}
			| m^b(z)-m_{c}(z)|\prec \frac{ \psi_n + \phi_n}{n\eta}+ \frac{1}{n(\kappa +\eta)} + \frac{1}{(n\eta)^2\sqrt{\kappa +\eta}}
		\end{equation}
		uniformly in $z\in \wt S_{out}(\epsilon,\wt \e)$. %\cor Now the rigidity estimate \eqref{rigidityb} can be proved using \eqref{aver_inbbbb}-\eqref{aver_outbbbb} in the same way that \eqref{rigidity} was proved using \eqref{aver_in}-\eqref{aver_out0} as in the proof of Theorem 2.5 in \cite{PartIII}. %\ref{thm_largerigidity}. 

		\begin{definition}[Regularized resolvents]\label{resol_not2}
			For $z = E+ \ii \eta \in \mathbb C_+,$ we define the regularized resolvent $\wh G(z)$ as
			$$ \wh G(z) := \left[H(z)- z n^{-10} \begin{pmatrix} I_{p+q} & 0 \\ 0 & 0 \end{pmatrix} \right]^{-1} .$$
			Moreover, we define 
			$${\wsH }:=\wh S_{xx} ^{-1/2}S_{xy}\wh S_{yy}^{-1/2}, \quad \wh S_{xx}:= S_{xx}+n^{-10}, \quad \wh S_{yy}:= S_{yy}+n^{-10}.$$
			%and the resolvent $\wh R(z)$ together with its spectral decomposition as
			%\be\label{specrtal2}
			%\begin{split}
			%\wh R(z):&=\begin{pmatrix}  \wh R_1  & - z^{-1/2} \wh R_1{\wsH }  \\ - z^{-1/2} {\wsH }^{\top}  \wh R_1  &   \wh R_2 \end{pmatrix} \\
			%& = \sum\limits_{k = 1}^q \frac{1}{ \wh\lambda_k-z}\left( {\begin{array}{*{20}c}
			%   {{ \wh\xi _k \wh\xi _k^{\top}  }} & {-z^{-1/2}\sqrt { \wh\lambda _k } \wh\xi _k \wh\zeta _{ k}^{\top}}  \\
			%   {-z^{-1/2} \sqrt { \wh\lambda _k }  \wh\zeta _{k}  \wh\xi _k^{\top}  } & { \wh\zeta _k \wh\zeta _k^{\top} }  \\
			%\end{array}} \right)  - \frac1z \left( {\begin{array}{*{20}c}
			%   {\sum_{k=q+1}^p{\wh\xi _k \wh\xi _k^{\top}  }} & 0  \\
			%   {0 } & {0}  \\
			%\end{array}} \right).
			%\end{split}
			%\ee
			The resolvents $\wh R(z)$, $\wh G^b(z)$ and $\wh R^b(z)$ etc.\;can be defined in the obvious way as in Definition \ref{resol_not}. 
		\end{definition}

		With the Schur complement formula, we can obtain similar expressions for $ {\wsG}_L$, $ {\wsG}_R$ and ${\wsG}_{LR}$ as in \eqref{GL1}--\eqref{GLR1}. 
		The main reason for introducing regularized resolvents is that they satisfy the following deterministic bounds: for some constant $C>0$,%it is easy to prove the following result using the same argument for the proof of Lemma \ref{lem op1}.
		\be\label{op G}
		\left\| \wh G(z)\right\| \le \frac{Cn^{10}}{\eta},\quad \left\| \wh G^b(z)\right\| \le \frac{Cn^{10}}{\eta}. %, \quad \left\| \partial_z \wh G(z)\right\| \le \frac{Cn^{10}}{\eta^2}.
		\ee 
		This estimate has been proved in Lemma 3.6 of \cite{PartIII}. With a standard perturbation argument, we can control the difference between $\wh G(z)$ and $G(z)$ as in the following claim.
		%The results for $\wh G(z)$ can be extended to $G(z)$ , and vice versa. 

		%\begin{lemma}[Lemma 3.6 of \cite{PartIII}]\label{lem op}
		%For $z=E+\ii \eta\in \C_+$ such that $c\le |z|\le c^{-1}$ for a constant $c>0$, \eqref{op RH1} and \eqref{op G1} hold for $\wh R(z)$ and $\wh G(z)$. Moreover, we have that
		%\be\label{op G}
		%\left\| \wh G(z)\right\| \le \frac{Cn^{10}}{\eta}. %, \quad \left\| \partial_z \wh G(z)\right\| \le \frac{Cn^{10}}{\eta^2}.
		%\ee 
		%for some constant $C>0$.
		%\end{lemma}
		
		%In the proof, we will always take $\eta\gg n^{-1}$, and the deterministic bound \eqref{op G} then justifies the application of Lemma \ref{lem_stodomin} (iii) when we calculate the expectation of polynomials of $\wh G$ entries. For simplicity of presentation, we will not repeat this again in the proof. 
		
		\begin{claim}\label{removehat}
			Suppose there exists a high probability event $\Xi$ on which $ \| G(z)\|_{\max}=\OO(1)$ for $z$ in some subset, where $\|G\|_{\max}:=\max_{i,j}|G_{ij}|$ denotes the max-norm. 
			Then, we have that
			\be\label{g-whg} \|G(z) - \wh G(z)\|_{\max} \le n^{-8}\quad \text{ on } \quad \Xi.\ee
			The same bound holds for $\|G^b(z) - \wh G^b(z)\|_{\max}$ on the events $\{ \| G^b(z)\|_{\max}=\OO(1)\}$ and $\{ \| \wh G^b(z)\|_{\max}=\OO(1)\}$.
		\end{claim}
		\begin{proof}
			For $t\in [0,1]$, we define 
			$$ G_t(z) := \left[H(z)- t z n^{-10} \begin{pmatrix} I_{p+q} & 0 \\ 0 & 0 \end{pmatrix} \right]^{-1}, \quad \text{with}\quad G_0(z)= G(z), \quad G_1(z)=\wh  G(z) .$$
			Taking derivatives with respect to $t$, we get that
			\be\label{partialtG}\partial_t G_t(z) = zn^{-10} G_t(z) \begin{pmatrix} I_{p+q} & 0 \\ 0 & 0 \end{pmatrix} G_t(z) . \ee
			Thus, applying Gronwall's inequality to 
			$$ \|G_t(z)\|_{\max}\le \| G(z)\|_{\max} + C n^{-9} \int^t_0 \|G_s(z)\|_{\max}^2  \dd s, $$
			we obtain that $\|G_t(z)\|_{\max}\le C$ for all $0\le t \le 1$ on $\Xi.$ Then, using \eqref{partialtG} again, we get \eqref{g-whg}.
			%$ \|G(z) - \wh G(z)\|_{\max} \le n^{-8}$ on $\Xi$. Such a small error will not affect any of our results.  
			%In particular, ....... the estimates also hold for $\wh G(z)$, $\wh m(z)$ and $\wh m_{\al }(z)$, $\al=1,2,3,4$. 
		\end{proof}
		Note that the bound \eqref{g-whg} is purely deterministic on $\Xi$, so we do not lose any probability here. Moreover, such a small error $n^{-8}$ will not affect any of our results.  
		
		\begin{proof}[Proof of Lemma \ref{lem null}]
			%By Theorem \ref{thm_local} with $\phi_n=n^{-1/2}$, one can choose $J= \Psi(z)$ and 
			%\begin{equation*}
			%K=\frac{1}{n \eta}, \quad \text{ or }  \quad   \frac{1}{N(\kappa +\eta)} + \frac{1}{(N\eta)^2\sqrt{\kappa +\eta}} \ \ \text{for} \ \ z\in S_{out}(\e).
			%\end{equation*}
			%Then using \eqref{KEYEYEYEY} and $|m(z)-m_c(z)|\lesssim |1-z|^{-1} |m_3(z)-m_{3c}(z)|$ by \eqref{m3m}, we get \eqref{aver_in} and \eqref{aver_out0}. Note that due to the $|1-z|^{-1}$ factor, we need to stay away from $z=1$, which is the main reason why we need to restrict ourself to the domain $\wt S(\epsilon,\wt \e)$ or $\wt S_{out}(\epsilon,\wt \e)$.
			%With \eqref{Kdist} and the method in \cite{EKYY1,EYY}, we can prove the following rigidity estimate: for any fixed $\delta>0$ and all $n^\delta \le i \le (1-\delta)q $, 
			%\eqref{rigidity} holds. To obtain this estimate for the first $n^\delta$ eigenvalues, we still need to provide the following upper bound 
			%for any constant $\e>0$,
			%\be\label{upper1}
			%\lambda_1 \le \lambda_+ + n^{-2/3+\e} \quad \text{with high probability.}
			%\ee
			%Given this bound, the estimate \eqref{Kdist} and the method in \cite{EKYY1,EYY} allow us to conclude \eqref{rigidity} also for $1 \le j \le n^\delta$.
			With the same argument as those for \cite[Theorems 2.12 and 2.13]{EKYY1}, \cite[Theorem 2.2]{EYY} and \cite[Theorem 3.3]{PY}, from the averaged local law \eqref{aver_inbbbb} we can derive that for any small constants $\delta,\e>0$, \eqref{rigidityb} holds for all $n^\e \le i \le (1-\delta)q$. To conclude \eqref{rigidityb} for the first $n^\e$ eigenvalues, we still need to prove an upper bound on them. More precisely, it suffices to show that for any small constant $\e>0$, 
			\be\label{upper1}
			\lambda_1^b \le \lambda_+ + n^{-2/3+\e},\quad w.h.p.
			\ee
			Combining this estimate with the rigidity estimate for $\lambda_{n^\e}^b$, we can conclude that \eqref{rigidityb} holds all $1 \le i < (1-\delta)q$ since $\e$ can be arbitrarily small.
			
			First, using the local law \eqref{aver_outbbbb}, we can obtain that for any small constants $c,\e>0$, %with high probability,
			\be\label{remain1}
			\#\{i:\lambda^b_i \in [\lambda_+ + n^{-2/3+\e},1-c]\}=0,\quad w.h.p.
			\ee
			The proof is standard and similar to the one for (4.7) of \cite{PartIII}, so we omit the details.
			\iffalse
			We choose $\eta=n^{-2/3}$ and $ E=\lambda_+ + \kappa \le 1-c$ outside of the spectrum with $\kappa\ge n^{-2/3+2\e} \gg n^\e \eta$. Then using \eqref{aver_outbbbb}, we get that
			\be\label{add1}|\im m^b(z) - \im m_c(z)| \prec \frac{n^{-c_\phi}+n^{-c_\psi}}{n\eta}+ \frac{1}{n(\kappa +\eta)} + \frac{1}{(n\eta)^2\sqrt{\kappa +\eta}} \lesssim \frac{n^{-\e}}{n\eta} .\ee
			On the other hand, if there is an eigenvalue $\lambda_j$ satisfying $|\lambda_j - E| \le \eta$ for some $1\le j \le q$, then 
			\be\label{add2}\im m^b(z) = \frac1q\sum_{i=1}^q \frac{\eta}{|\lambda^b_i-E|^2 + \eta^2} \gtrsim \frac{1}{n\eta}.\ee
			On the other hand, by \eqref{Immc} we have 
			$$\im m_c(z) =\OO\left( \frac{\eta}{\sqrt{\kappa + \eta}}\right) = \OO\left( \frac{n^{-\e}}{n\eta}\right).$$
			Together with \eqref{add2}, this contradicts \eqref{add1}. Hence we conclude \eqref{remain1} since $\e$ can be arbitrarily small. 
			\fi
			It remains to prove that for a sufficiently small constant $c>0$,
			\be\label{remain2}
			\#\{i:\lambda_i^b \in [1-c,1]\}=0,\quad w.h.p.
			\ee

			We define a continuous path of interpolated random matrices between $Y$ and $Y+BZ$ as 
			$$\cal Y_t: = Y + t B Z,   \quad t\in [0,1].$$
			By replacing $\cal Y$ with $\cal Y_t$ in \eqref{eqn_defGb} and Definition \ref{resol_not2}, we can define $H^b_t(z)$, $G^b_t(z)$, $\wh H^b_t(z)$ and $\wh G^b_t(z)$ correspondingly. First, we claim the following result.
			\begin{claim}\label{claim revise}
				With high probability, we have that
				\be\label{remain3}
				\|G_t^b(1-c)\|_{\max} <\infty \ \text{  for all}\ \ t\in [0,1].
				\ee
			\end{claim}
			
			We postpone the proof of this claim until we complete the proof of \eqref{rigidityb}. Let  $\lambda_1^b(t)\ge \lambda_2^b(t)\ge \cdots\ge \lambda_q^b(t)$ be the eigenvalues of $\cal C_{\cal Y_t X}$. For any $1\le i \le q$, $ \lambda_i^{b}(t):[0,1]\to \R$  is a continuous function with respect to $t$ on the extended real line $ \overline{\mathbb R}$. 
			%\be\label{remain3cont}
			% \ \text{ is a continuous function },  
			%\ee
			By \eqref{rigidity}, the eigenvalues $\lambda_i^{b}(0)$ of $\cal C_{XY}$ are all inside $[0,\lambda_+ + n^{-2/3+\e}]$  with high probability. If \eqref{remain3} holds, then we have that
			$$m_t^b(1-c)=\frac1q\sum_{i=1}^q \frac{1}{\lambda_i^b(t)-(1-c)}\quad \text{ is finite for all $t\in [0,1]$.}$$ 
			It means that the eigenvalue $\lambda_1^{b}(t)$ does not cross the point $E=1-c$ for all $t\in [0,1]$. %Hence using the continuity of eigenvalues in \eqref{remain3cont}, 
			Thus, we conclude \eqref{remain2}, which further concludes \eqref{upper1} together with \eqref{remain1}. 
			\end{proof}

		Finally, we give the proof of Claim \ref{claim revise}.
		\begin{proof}[Proof of Claim \ref{claim revise}]
			\iffalse
			By the definition of $\cal C_{X\cal Y_t}$, to prove \eqref{remain3cont}, it suffices to prove that with high probability, $(Y_tY^{\top}_t)^{-1}$ is continuous in $t$ for all $t\in [0,1]$. For this purpose, we only need to show that with high probability,
			\be \nonumber
			\cal Y_t \cal Y_t^{\top} \ \ \text{is non-singular for all}\  t\in [0,1].
			\ee
			We consider discrete times $t_k = kn^{-10}$. Note that $\cal Y_t$ satisfies the assumptions of Lemma \ref{SxxSyy}. Hence with a simple union bound we get that 
			$$ \min_{0\le k \le n^{10}} \lambda_q(\cal Y_{t_k} \cal Y_{t_k}^{\top}) \ge c_0  $$
			for some constant $c_0>0$ on a high probability event $\Xi$. Moreover, by the bounded support condition of $Y$ and $Z$, we have that on a high probability event $\Xi'$, $ \|\cal Y\| =\OO(n)$ and $\|Z\|=\OO(\sqrt{n})$. 
			%we get that there exists a high probability evern $\Xi_2$ such that
			%\be\label{roughXt} \mathbf 1(\Xi_2)\max_{i,\mu}|(X_t)_{i\mu}| \le 1 \ \Rightarrow \ \mathbf 1(\Xi_2)\sup_{0 \le t \le 1} \|X_t \| =\OO(n).\ee
			Thus we have that  
			$$\mathbf 1(\Xi') \sup_{t_{k-1} \le t \le t_k}\|\cal Y_{t} \cal Y_{t}^{\top} - \cal Y_{t_k} \cal Y_{t_k}^{\top}\| \le n^{-10}\cdot n^2 =n^{-8}.$$
			Combining the above two estimates, we get that on $\Xi\cap \Xi'$,
			$$  \inf_{0\le t\le 1}\lambda_p(\cal Y_t \cal Y_t^{\top}) = \min_{1\le k \le n^{10}} \inf_{t_{k-1} \le t \le t_k} \lambda_p(\cal Y_t \cal Y_t^{\top})  \ge c_0 - n^{-8}. $$
			This concludes \eqref{remain3cont}.
			\fi
			Take a discrete net of $t$, $t_k = kn^{-50}$, for $0\le k\le n^{50}$. %Since $X$ and $\cal Y_t$ satisfy the assumptions of
			\iffalse 
			By Lemma \ref{thm_localb}, for any $t\in [0,1]$, the local law \eqref{aniso_outstrongb} holds for $G_t^b(z)$ for some matrix limit $\Pi_t^b(z)$ with $\|\Pi_t^b(z)\|=\OO(1)$. By Claim \ref{removehat}, this local law also holds for $\wh G_t^b(z)$. %Then with the proof for Theorem \ref{thm_localout}, we get that
			\fi
			First, we claim that there exists a high probability event $\Xi_1$, so that
			\be\label{remain4} 
			\mathbf 1(\Xi_1)\max_{0\le k \le n^{50}}\| \wh G^b_{t_k}(E+ \ii n^{-10})\|_{\max} \le C \quad \text{for }\ E:=1-c,
			\ee
			for some large constant $C>0$. 
			%Finally, we still need to prove \eqref{remain4}. 
			In fact, notice that $\cal Y_t$ also satisfies the assumptions for $\cal Y$ in Lemma \ref{lem null}. Hence, using \eqref{remain1}, we obtain that for any $t_k$, the eigenvalues $\lambda_i^b(t_k)$ are  inside $[0,\lambda_+ + n^{-2/3+\e}]\cup [1-c/2,1]$ with high probability. By taking a union bound, we get that
			\be\label{est_eig_revise}\min_{0\le k \le n^{50}}\min_{1\le i \le q}|E-\lambda_i^b({t_k})| \gtrsim 1 \quad  w.h.p.\ee
			%through a simple union bound. 
			%Together with \eqref{spectral1} for $R^b_t$, 
			Applying the spectral decomposition \eqref{spectral1} to $R^b$, we obtain from \eqref{est_eig_revise} that %the above estimate gives immediately that 
			\begin{align}\nonumber%\label{zz0}
				\max_{0\le k \le n^{50}}\left\| R^b_{t_k}(z) \right\| \le C \quad \text{for}\quad z= E + \ii n^{-10}.
			\end{align}
			Combining this bound with \eqref{GL1}--\eqref{GLR1} and using Lemma \ref{SxxSyy}, we get that
			\begin{align*} 
				\max_{0\le k \le n^{50}}\left\| G^b_{t_k}(z) \right\| \le C,\quad w.h.p. %, \quad z= E+ \ii n^{-10}.
			\end{align*}
			Next, applying Claim \ref{removehat}, we get \eqref{remain4} for $\wh G^b$.

			Now, given \eqref{remain4}, 
			%with high probability for any $k$. 
			using the deterministic bound \eqref{op G} for $\wh G^b$, we get that on $\Xi_1$, 
			\begin{equation*}\begin{split} 
					\left\|\wh G_t^b(E+ \ii n^{-10}) - \wh G_{t_k}^b(E+ \ii n^{-10})\right\|_{\max} & \lesssim n^{-50}\|\wh G_t^b(E+ \ii n^{-10})\|  \cdot \|Z\| \cdot \|\wh G_{t_k}^b(E+ \ii n^{-10})\| \\
					& \lesssim   n^{-50}\cdot \left( {n^{20}} \right)^2\cdot \|Z\| \lesssim n^{-10}\|Z\|,
					%\\& \le n^{-50}\cdot \left( {Cn^{20}} \right)^2\cdot C\sqrt{n} \le n^{-9},
				\end{split}
			\end{equation*}
			for any $t_{k-1}\le t \le t_k$. By the bounded support condition of $Z$, we have that $\|Z\|=\OO(\sqrt{n})$ on a high probability event $\Xi_2$. 
			%$$\left|H_t(1-c+ \ii n^{-10}) - H_{s}(1-c+\ii n^{-10})\right|\le C\max \left\{\left|\sqrt{t}-\sqrt{s}\right|,\left|\sqrt{1-t}-\sqrt{1-s}\right|\right\} \ \  \text{ for all } \ t,s \in [0,1].$$
			%Then with the deterministic bound \eqref{op G} and a standard perturbation argument, we get 
			%$$\max_{t\in [t_k,t_k+1]} \|\wh G_t(1-c+ \ii n^{-10}) - \wh G_{t_k}(1-c+ \ii n^{-10})\|_{\max} \le Cn^{30}\cdot n^{100} =\oo(1),$$
			%with high probability. 
			Thus, on the high probability event $ \Xi_1 \cap \Xi_2$, 
			\begin{align*} 
				\left\|\wh G_t^b(E+ \ii n^{-10}) - \wh G_{t_k}^b(E+ \ii n^{-10})\right\|_{\max} & \lesssim n^{-50}\cdot \left( {n^{20}} \right)^2\cdot \sqrt{n} \le n^{-9},
			\end{align*}
			which gives that
			$$\mathbf 1( \Xi_1 \cap \Xi_2)\max_{0\le t\le 1} \|\wh G_{t}^b(E+ \ii n^{-10})\|_{\max} \le C.$$ 
			Finally, using the same perturbation argument as in the proof of Claim \ref{removehat}, we can remove both the $\ii n^{-10}$ and the regularization in $\wh G$, which gives \eqref{remain3} on $\Xi_1\cap \Xi_2$.  
		\end{proof}

%%%%%%%%%%%%%%%%%%%%%%%%%%%%%%%%%%%%%%%%%%%%%%
%% Support information (funding), if any,   %%
%% should be provided in the                %%
%% Acknowledgements section.                %%
%%%%%%%%%%%%%%%%%%%%%%%%%%%%%%%%%%%%%%%%%%%%%%
\section*{Acknowledgments}
We would like to thank the editor, the associated editor and an anonymous referee for their helpful comments, which have resulted in a significant improvement of the paper. The second author is supported in part by the Wharton Dean's Fund for Postdoctoral Research.
%\end{acks}

%\begin{funding}
%The second author was supported in part by the Wharton Dean's Fund for Postdoctoral Research.
%\end{funding}

%%%%%%%%%%%%%%%%%%%%%%%%%%%%%%%%%%%%%%%%%%%%%%
%% Please use \tableofcontents for articles %%
%% with 50 pages and more                   %%
%%%%%%%%%%%%%%%%%%%%%%%%%%%%%%%%%%%%%%%%%%%%%%
%\tableofcontents

%\begin{supplement}
%\textbf{Supplement to ``Sample canonical correlation coefficients of high-dimensional random vectors with finite rank correlations"}.
%In the supplementary file \cite{MY_aapsuppl}, we provide the proofs of Theorem \ref{main_thm0.5},  Corollary \ref{main_cor} and Theorem \ref{main_thm1}.
%\end{supplement}

%\bibliographystyle{imsart-number}
%\bibliographystyle{abbrv}
%\bibliography{aniso_bib.bib}

\end{document}